%% file: singer_lattice.tex
\documentclass[numbook,envcountsame]{svjour3}
\pdfoutput=1

\usepackage[utf8]{inputenc}
\usepackage{amsfonts,amssymb}
\usepackage{amsmath}
\usepackage{graphicx}
\usepackage{array}
\usepackage[small,balance]{diagrams}
\usepackage{enumitem}
\usepackage{xargs}
\usepackage{ifthen}
\usepackage{url}
\usepackage{longtable}
\usepackage{pdflscape}
\usepackage{afterpage}

\usepackage{listings}
\lstdefinelanguage{GAP}{%
  morekeywords={%
    Assert,Info,IsBound,QUIT,%
    TryNextMethod,Unbind,and,break,%
    continue,do,elif,%
    else,end,false,fi,for,%
    function,if,in,local,%
    mod,not,od,or,%
    quit,rec,repeat,return,%
    then,true,until,while%
  },%
  sensitive,%
  morecomment=[l]\#,%
  morestring=[b]",%
  morestring=[b]',%
}[keywords,comments,strings] 

\usepackage{comment}
\usepackage[usenames,dvipsnames,table]{xcolor}

\usepackage[left=2.8cm,right=2.8cm,top=2.8cm,bottom=3.8cm]{geometry}
\usepackage{empheq}
\numberwithin{equation}{section}

\newtheorem{introtheorem}{Theorem}{\bf}{\it}

\newtheorem{observation}[theorem]{Observation}
\newtheorem{fact}[theorem]{Fact}
\spnewtheorem*{fatconjecture}{Conjecture}{\bf}{\it}
\newcommand{\Q}{\mathbb{Q}}
\newcommand{\Z}{\mathbb{Z}}
\newcommand{\N}{\mathbb{N}}

\newcommand{\defeq}{\mathrel{\mathop{:}}=}

\newcommand{\abs}[1]{\lvert #1 \rvert}
\newcommand{\gen}[1]{\langle #1 \rangle}

\newcommand{\lk}{\operatorname{lk}}

\newcommand{\F}{\mathbb{F}}
\newcommand{\id}{\operatorname{id}}
\newcommand{\calC}{\mathcal{C}}

\newcommand{\im}{\operatorname{im}}

\newcommand{\Aff}{\operatorname{AGL}}

\newcommand{\calO}{\mathcal{O}}
\newcommand{\frakm}{\mathfrak{m}}
\newcommand{\GL}{\operatorname{GL}}
\newcommand{\PGL}{\operatorname{PGL}}

\newcommand{\PGaL}{\operatorname{P\Gamma{}L}}
\newcommand{\SL}{\operatorname{SL}}
\newcommand{\PSL}{\operatorname{PSL}}
\newcommand{\Aut}{\operatorname{Aut}}
\newcommand{\Out}{\operatorname{Out}}
\newcommand{\Inn}{\operatorname{Inn}}
\newcommand{\Gal}{\operatorname{Gal}}
\renewcommand{\P}{\mathbb{P}}
\newcommand{\calP}{\mathcal{P}}
\newcommand{\calL}{\mathcal{L}}
\newcommand{\hjelm}{\mathcal{H}}
\newcommand{\hpts}{\mathcal{P}}
\newcommand{\hlns}{\mathcal{L}}
\newcommand{\typ}{\operatorname{typ}}

\newcommand{\mmod}{\mathrel{\,\operatorname{mod}\,}}

\newcommand{\Groupoid}{\mathcal{G}}
\newcommand{\SingTrip}{\mathcal{ST}}
\newcommand{\Inner}{\mathcal{IST}}
\newcommand{\DiffMat}{\mathcal{DM}}
\newcommand{\DiffMatSet}{\mathrm{DM}}
\newcommand{\DiffMatBSet}{\mathrm{DM}_0}
\newcommand{\Normal}{\mathcal{N}}

\newcommand{\ProjPla}{\Pi}
\newcommand{\inc}{\mathrel{I}}
\newcommand{\dbar}[1]{\bar{\bar{#1}}}

\numberwithin{equation}{section}

\newcommand{\chr}{\operatorname{char}}
\newcommand{\llb}{(\mkern-2mu(}
\newcommand{\rrb}{)\mkern-2mu)}
\newcommand{\lls}{[\mkern-2mu[}
\newcommand{\rrs}{]\mkern-2mu]}
\newcommand{\alg}{\mathcal{A}}
\newcommand{\lseries}[1]{\llb #1 \rrb}
\newcommand{\pseries}[1]{\lls #1\rrs}
\newcommand{\Tr}{\operatorname{Tr}}
\newcommand{\act}{\curvearrowright}
\newcommandx{\stab}[3][2={},3={}]{\ifthenelse{\equal{#2#3}{}}%
  {{\ifthenelse{\equal{#1}{}}{P}{P_{#1}}}}%
  {{{\ifthenelse{\equal{#1}{}}{P^{#3\ifthenelse{\equal{#2}{}}{}{(#2)}}}{P_{#1}^{#3\ifthenelse{\equal{#2}{}}{}{(#2)}}}}}}}

\newenvironment{smallarray}[1]
 {\null\,\vcenter\bgroup\scriptsize
  \arraycolsep=.13885em
  \hbox\bgroup$\array{@{}#1@{}}}
 {\endarray$\egroup\egroup\,\null}

\smartqed
\begin{document}

\title{On panel-regular $\tilde{A}_2$ lattices}

\author{Stefan Witzel\thanks{The research for this article is part of the DFG project WI 4079/2. It was also supported by the SFB~701.}}
\institute{S. Witzel \at Department of Mathematics\\Bielefeld University, PO Box 100131, 33501 Bielefeld, Germany \email{switzel@math.uni-bielefeld.de}}

\date{}

\maketitle

\begin{abstract}
We study lattices on $\tilde{A}_2$ buildings that preserve types, act regularly on each type of edge, and whose vertex stabilizers are cyclic. We show that several of their properties, such as their automorphism group and isomorphism class, can be determined from purely combinatorial data. As a consequence we can show that the number of such lattices (up to isomorphism) grows super-exponentially with the thickness parameter $q$.

We look in more detail at the $3295$ lattices with $q \in \{2,3,4,5\}$. We show that with one exception for each $q$ these are all exotic. For the exotic examples we prove that the automorphism group of the lattice and of the building coincide, and that two lattices are quasi-isometric only if they are isomorphic.
\end{abstract}




\section{Introduction}

A simply connected triangle complex $X$ is a building of type $\tilde{A}_2$ if every vertex link is isomorphic to the incidence graph of a projective plane. A group $\Gamma$ acting on $X$ is a uniform lattice if there are only finitely many $\Gamma$-orbits of cells and if the stabilizer in $\Gamma$ of any cell is finite. If $X$ admits a uniform lattice, it has to be locally finite.

There are algebraic ways to construct uniform lattices on buildings as groups of integral points of anisotropic algebraic groups but it is not easy to understand these explicitly, for example to determine presentations, fundamental domains, or stabilizers. It is therefore useful to be able to construct a lattice more explicitly from combinatorial data for example by amalgamating together well-chosen stabilizers. Ronan \cite[Section~3]{ronan84} and Kantor \cite[Section~C.3]{kantor86} were the first to use difference sets and Singer groups for this purpose. More recently Essert \cite{essert13} has determined the complexes of groups that correspond to certain uniform building lattices. We build on his work and investigate a subclass of his lattices. We call a \emph{Singer lattice} a lattice that preserves types and acts regularly (transitively and freely) on the edges of each type. We call it a \emph{Singer cyclic lattice} if in addition every vertex stabilizer is cyclic.

Essert shows that Singer cyclic lattices have very nice presentations of the form
\[
\Gamma = \gen{\sigma_0, \sigma_1, \sigma_2 \mid \sigma_i^{q^2+q+1}, 0 \le i \le 2, \sigma_0^{E_{j,0}} \sigma_1^{E_{j,1}} \sigma_2^{E_{j,2}}, 1 \le j \le q}
\]
and conversely, that every group with such a presentation is a Singer cyclic lattice. Here $q$ is a prime power and $E$ is what we call a \emph{based difference matrix} (defined in Section~\ref{sec:singer_lattices}). The geometric meaning of the parameter $q$ is that every edge is contained in $q+1$ triangles.

The reason why Singer cyclic lattices can be studied so efficiently is that many of their properties can be derived from the difference matrix $E$ by combinatorial means. Most importantly, the difference matrix allows to reconstruct in a very elementary manner large balls in the associated building (see Theorem~\ref{thm:hjelmslev}), which is fundamental for most of the results below. More immediately, we define a notion of equivalence of difference matrices and define a group $\Aut(E)$ that satisfy:

\begin{introtheorem}
\label{thm:intro_equivalence}
There is a bijective correspondence between Singer cyclic lattices up to isomorphism and difference matrices up to equivalence.

If $\Gamma$ corresponds to $E$ then $\Out(\Gamma) = \Aut(E)$.
\end{introtheorem}

Estimating the number of difference matrices up to equivalence we get

\begin{introtheorem}
\label{thm:intro_bound}
The number of Singer cyclic lattices with parameter $q = p^\eta$ is bounded below by
\[
\frac{1}{162 \eta^3}((q+1)!)^2 \sim \frac{\pi}{81 \eta^3}\frac{(q+1)^{2q+3}}{e^{2q+2}}\text{.}
\]
In particular, the number grows super-exponentially.
\end{introtheorem}

For example, for $q = 7$ there are more than $10^8$ Singer cyclic lattices.

The rest of the article is concerned with a more explicit study of Singer cyclic lattices with parameter $q \in \{2,3,4,5\}$. Of these parameters, the only one that admits Singer lattices that are not Singer cyclic is $q = 4$;
for the others ``Singer cyclic lattice'' could be replaced by ``Singer lattice'' in the following statements. Most of the results are computer aided although not computationally intensive.

\begin{introtheorem}
\label{thm:intro_count}
The number of Singer cyclic lattices up to isomorphism is
\begin{itemize}
\item $2$ for $q = 2$,
\item $7$ for $q = 3$,
\item $17$ for $q = 4$,
\item $3269$ for $q = 5$.
\end{itemize}
\end{introtheorem}

The two lattices for $q = 2$ are each contained in one of the four well-known chamber regular lattices first studied by Ronan~\cite{ronan84}, Tits~\cite{tits85,tits86b}, and Köhler--Meixner--Wester \cite{koemeiwes84a}.

Lattices on two-dimensional euclidean buildings live in an interesting border region for rigidity: the classification of euclidean buildings \cite{weiss09} which implies that a building of type $\tilde{A}_d$ is \emph{Bruhat--Tits}, namely that it comes from $\PGL_{d+1}(K)$ for some (finite-dimensional division algebra over a) local field $K$, only applies for $d \ge 3$. In dimension two there are other buildings, which we call \emph{exotic}. But if a lattice is contained in $\PGL_3(K)$ then Margulis arithmeticity \cite{margulis} implies that it actually is an arithmetic subgroup. It is known (and we explain explicitly) that for every $q$ there is a Singer cyclic lattice that is arithmetic, contained in $\PGL_3(\F_q\lseries{t})$. We provide a general method for finding embeddings of Singer cyclic lattices into $\PGL_3(K)$ but also prove

\begin{introtheorem}
\label{thm:intro_exotic}
For each $q \le 5$ among the Singer cyclic lattices there is a single arithmetic one while all others act on exotic buildings.
\end{introtheorem}

We have seen that the automorphism groups of Bruhat--Tits buildings are non-discrete locally compact groups. In most known cases the automorphism group of an exotic building is discrete and in particular is a finite extension of any uniform lattice on it (see \cite[Section~7]{vanmaldeghem90} for examples with non-discrete, vertex-transitive automorphism groups). We confirm that the same is true here. In fact more is true:

\begin{introtheorem}
\label{thm:intro_aut}
Let $\Gamma \curvearrowright X$ be a Singer cyclic lattice with $q \le 5$ acting on an exotic building. Then $\Aut(\Gamma) = \Aut(X)$ and this group is not transitive on types.
\end{introtheorem}

One consequence is that none of the exotic Singer cyclic lattices with $q \le 5$ are quasi-isometric (or commensurable) to any of the vertex-regular lattices studied by Cartwright--Steger--Mantero--Zappa \cite{carmanstezap93a}.

Finally, we distinguish the buildings that Singer cyclic lattices can act on. Since buildings are QI-rigid, this also yields a quasi-isometry classification of Singer cyclic lattices.

\begin{introtheorem}
\label{thm:intro_qi}
Singer cyclic lattices with $q \le 5$ acting on isomorphic buildings are isomorphic. As a consequence, quasi-isometric (or commensurable) Singer cyclic lattices with $q \le 5$ are isomorphic.
\end{introtheorem}

Based on Theorems~\ref{thm:intro_exotic} and~\ref{thm:intro_qi} we formulate the following:

\begin{fatconjecture}
Almost all Singer cyclic lattices are exotic and pairwise not quasi-isometric in the following sense:
\[
\lim_{q \to \infty} \frac{\abs{\{\text{exotic Singer cyclic lattices with parameter }q\}/\text{QI}}}{\abs{\{\text{Singer lattices with parameter }q\}}}  = 1
\]
where $q$ ranges over prime powers.
\end{fatconjecture}

The paper is organized as follows. In Section~\ref{sec:diff_sets} we recall some classical facts about incidence geometry, Singer groups, and difference sets. Singer lattices and their parametrization through based difference matrices (leading to Theorem~\ref{thm:intro_equivalence}) are introduced in Section~\ref{sec:singer_lattices}. Theorem~\ref{thm:intro_count} and a more precise version of the bound in Theorem~\ref{thm:intro_bound} are proven in Section~\ref{sec:census}. Section~\ref{sec:brutit_examples} explains how Bruhat--Tits examples can be produced using the construction in \cite{carmanstezap93a}. In Section~\ref{sec:ronan} we describe explicitly how the two Singer cyclic lattices for $q = 2$ can be extended do become chamber transitive, making the connection with the Köhler--Meixner--Wester/Ronan lattices. A general method to decide, given a Singer cyclic lattice $\Gamma$ and a local field $K$, whether $\Gamma$ embeds into $\PGL_3(K)$ is developed in Section~\ref{sec:linearity}. As an illustration we provide explicit embeddings (in terms of matrices) of the Bruhat--Tits Singer cyclic lattices for $q \in \{2,3\}$. Sections~\ref{sec:hjelmslev} and~\ref{sec:building_auts} are concerned with identifying the buildings that Singer cyclic lattices act on, and their automorphism groups. We obtain Theorems~\ref{thm:intro_exotic},~\ref{thm:intro_aut}, and~\ref{thm:intro_qi} as a consequence. The code for the computer experiments, including some documentation, can be found in the GitHub repository \cite{github_sl}. It is written in Python~\cite{python} and GAP~\cite{gap} using the library GRAPE~\cite{grape} which in turn depends on the library nauty~\cite{nauty}.

\begin{acknowledgements}
Linus Kramer first suggested to me to look at Singer lattices. I am especially indebted to Pierre-Emmanuel Caprace and Hendrik van Maldeghem who made many suggestions that were crucial for the success of the project. Nicolas Radu suggested to use the software GRAPE to compute stabilizers and pointed out to me the existence of exotic buildings with indiscrete automorphism group. I would like to thank all of them for their help.
\end{acknowledgements}


\section{Singer cycles and difference sets}
\label{sec:diff_sets}

We start by introducing some classical concepts from incidence geometry. A good general reference for the contents of this section is \cite{dembowski68}. Let $\ProjPla = (P,L,I)$ be a finite projective plane: $P$ is the set of points, $L$ the set of lines, and $I$ the incidence relation satisfying the familiar axioms. It is \emph{Desarguesian} if it is isomorphic to $\P^2\F_q$ for some prime power $q$ and in that case $\abs{P} = \abs{L} = q^2+q+1$. We therefore introduce the following notation that will be used throughout the article:
\begin{gather*}
q= p^\eta, \quad p\text{ prime}, \quad \eta \in \N \setminus\{0\}, \quad \delta= q^2+q+1
\end{gather*}

If there is an action of a group $S$ on a projective plane $\ProjPla$ that is regular (that is transitive and free) on the set of points $P$ then $\ProjPla$ is a \emph{Singer plane} and $S$ is a \emph{Singer group} for $\ProjPla$.

\begin{proposition}[{\cite[4.2.7]{dembowski68}}]
If $S$ is a Singer group on a projective plane $\ProjPla = (P,L,I)$ then the action of $S$ on the set of lines $L$ is regular as well.
\end{proposition}

Historically mathematicians (such as Singer) have been most interested in the case where $S$ is cyclic. In that case any generator of $S$ will be called a \emph{Singer cycle}.

\begin{theorem}[{\cite{singer38}}]
Every finite Desarguesian projective plane admits a Singer cycle.
\end{theorem}

\begin{proof}
If $\ProjPla$ is finite Desarguesian it is isomorphic to the projective geometry of a $3$-dimensional $\F_q$-vector space for some $q$. Taking the $\F_q$-vector space to be $\F_{q^3}$ we get  an action of the cyclic group $\F_{q^3}^\times$ on $\ProjPla$ that is transitive on projective points. The stabilizer of a point is $\F_q^\times$ so the group $\F_{q^3}^\times/\F_q^\times \cong C_{q^2+q+1}$ acts freely.\qed
\end{proof}

This construction has the following generalization due to Ellers and Karzel~\cite{ellkar64}. Since their work substantially uses Dickson--Veblen near-fields we give a description that avoids these (but does not recover the classification result).

\begin{lemma}
\label{lem:singer_group}
Let $n$ be a divisor of $\gcd(3 \eta, \delta)$ and put $\eta_0 \defeq 3\eta/n$ and $q_0 \defeq p^{\eta_0}$. Let $\xi$ be a generator of $\F_{q^3}^\times / \F_q^\times$ and let $\varphi \in \Gal(\F_{q^3}/\F_{q_0})$ be the Frobenius automorphism.  The subgroup $S = \gen{g,h}$ of $\Gal(\F_{q^3}/\F_{q_0}) \ltimes \F_{q^3}^\times/\F_q^\times$ generated by $g \defeq (1, \xi^n)$ and $h \defeq (\varphi, \xi)$ is a Singer group.
\end{lemma}

\begin{proof}
Define $a \defeq \delta/n = (q^2+q+1)/n$ and $b = (q_0^{n-1} + \ldots + q_0 + 1)/n$. First note that
\begin{equation}
\label{eq:frobenius_conjugation}
\zeta^{\varphi} = \zeta^{q_0}
\end{equation}
for any $\zeta \in \F_{q^3}^\times / \F_q^\times$. 

Now we verify relations of $S$. Trivially
\[
g^a = 1
\]
From \eqref{eq:frobenius_conjugation} we also see
\[
g^h = g^{q_0}\text{.}
\]
Finally iterated applications of \eqref{eq:frobenius_conjugation} yield
\[
h^n = (\varphi \xi)^n = \varphi^{n} \xi^{1 + q_0 + \ldots + q_0^{n-1}} = g^b\text{.}
\]
From this we see that $S$ has the right order $a \cdot n = \delta$ and that $\gen{g}$ is normal in $S$. Since $\gen{g}$ acts transitively on itself, it remains to see that $S/\gen{g} \cong \gen{\varphi}$ acts transitively on $(\F_{q^3}^\times/\F_{q}^\times)/\gen{g}$. This is the case by \eqref{eq:frobenius_conjugation} because $q_0$ is a power of $p$ and $(\F_{q^3}^\times/\F_{q}^\times)/\gen{g}$ has order $n$ which is relatively prime to $p$ (since $n$ divides $\delta$).\qed
\end{proof}

The case where $S$ is cyclic is recovered with $n = 1$. Note that $S$ is linear over $\F_q$ if and only if $n \in \{1,3\}$. Taking $n = 3$ is possible whenever $q \equiv 1 \mmod 3$. The least prime power where $\eta$ and $\delta$ are not relatively prime is $128$.

We can now state a consequence of \cite[Sätze~5,6]{ellkar64} as:

\begin{theorem}
Every Singer group on a finite Desarguesian projective plane is equivalent to one as in Lemma~\ref{lem:singer_group}.
\end{theorem}

If $S$ is a Singer group on a projective plane $\ProjPla$ we can fix a point $x \in P$ and a line $y \in L$ and define the \emph{difference set} $D = \{s \in S \mid s.x \inc y\}$ (this is not the most general notion of a difference set, see \cite[§2.3.29]{dembowski68}). It is clear that $\ProjPla$ can be recovered (up to $S$-equivariant isomorphism) from $S$ and $D$. In this article we will only be interested in the case where $S$ is cyclic of order $\delta \defeq q^2+q+1$ for some prime $q$. In fact, given a Singer cycle $s \in S$ we will identify $S$ with $\Z/\delta\Z$ via $\Z/\delta\Z \to S, i \mapsto s^i$. Then for $D \subseteq \Z/\delta\Z$ to be a difference set (in our sense) it is necessary and sufficient that every $s \in \Z/\delta\Z \setminus \{0\}$ can be written in a unique way as $d - d'$ with $d,d' \in D$ (which also explains the name).

There are manipulations of the difference set $D$ that do not essentially change the projective plane it defines. One is translation: for $a \in \Z/\delta\Z$ the set $a + D = \{a+d \mid d \in D\}$ clearly is again a difference set; it is obtained from the action of $\Z/\delta\Z$ on $\ProjPla$ by choosing the base line $-a.\ell$ instead of $\ell$. The other operation is acting by automorphisms: for $m \in \Z/\delta\Z^\times$ the set $m \cdot D = \{m \cdot d \mid d \in D\}$ is a difference set; from our perspective this is best interpreted as replacing the Singer cycle $s$ by $s^{m^{-1}}$. In summary we have an action of the group $\Aff_1(\Z/\delta\Z) = (\Z/\delta\Z)^\times \ltimes \Z/\delta\Z$ on the difference sets in $\Z/\delta\Z$. We say that two difference sets are \emph{equivalent} if they lie in the same $\Aff_1(\Z/\delta\Z)$-orbit. A difference set $D$ is \emph{based} if $0 \in D$; every difference set is equivalent to a based one. For future reference we record the following easy fact:

\begin{lemma}
\label{lem:diff_set_difference}
Let $D \subseteq \Z/\delta\Z$ be a difference set and let $a \in D$. Then $D - a \defeq \{d - a \mid d \in D\}$ does not equal $D$.
\end{lemma}

\begin{proof}
Obviously $D-a$ is based so if $D$ is not based, we are done. If $D$ is based, then it contains $0$ and $a$ and $D - a$ contains $0$ and $-a$. So $D$ and $D - a$ cannot be equal because a difference set cannot contain all three of $-a$, $0$, and $a$ at the same time.\qed
\end{proof}

A difference set is \emph{Desarguesian} if the projective plane that it gives rise to is Desarguesian. It is a long-standing open conjecture that every finite projective plane admitting a point-transitive group of automorphisms is Desarguesian. In particular, this would apply to finite projective planes admitting a cyclic Singer group but the problem is open even for these. A result in this direction is that if a finite projective plane admits two distinct cyclic Singer groups then it is Desarguesian \cite{ott75}.

All the projective planes that we will concretely be concerned with are Desarguesian by virtue of being small:

\begin{theorem}[{\cite[Theorem~3.2.15]{dembowski68}}]
\label{thm:small_desarguesian}
Every projective plane of order $\le 8$ is Desarguesian.
\end{theorem}

For the set of Desarguesian difference set we have the following:

\begin{theorem}[{\cite{berman53}}]
\label{thm:difference_sets_affine}
The group $\Aff_1(\Z/\delta\Z)$ acts transitively on the Desarguesian difference sets in $\Z/\delta\Z$.
\end{theorem}

In what follows, we will talk about buildings of type $A_2$ rather than projective planes which is just a shift in notation: if $\ProjPla = (P,L,I)$ is a projective plane, the corresponding building $\Delta$ (of type $A_2$) is the graph with vertex set $P \sqcup L$ and edge set $\{\{x,y\} \mid x \inc y\}$. Conversely, every building of type $A_2$ gives rise to a projective plane.


\section{Singer lattices}
\label{sec:singer_lattices}

A $2$-dimensional simplicial complex is a building of type $\tilde{A}_2$ if it simply connected and the link of every vertex is a building of type $A_2$ (this uses \cite[Theorem~1]{tits81}, see also \cite[Theorem~4.9]{ronan}). The triangles are called \emph{chambers}, the edges are called \emph{panels}. The vertices can be colored (but not canonically) by elements of $\Z/3\Z$ and then the \emph{type} of a simplex is the set of colors of its vertices. Let $X$ be a locally finite building of type $\tilde{A}_2$ and let $\Gamma$ be a group acting on $X$. We say that $X$ is a \emph{Singer lattice} if the action preserves types and is regular (transitive and free) on the three sets of edges of a given type. Note that this implies in particular that the action is transitive on vertices of each type.

\begin{observation}
Let $\Gamma$ be a Singer lattice on $X$.
\begin{enumerate}
\item For every vertex $x \in X$ the stabilizer $\Gamma_x$ acts as a Singer group on the building $\lk x$.
\item If $\{x,y,z\}$ is a chamber in $X$ then $\Gamma$ is generated by $\Gamma_x$, $\Gamma_y$, and $\Gamma_z$.
\end{enumerate}
\end{observation}

We say that $\Gamma$ is a \emph{Singer cyclic lattice} if it is a Singer lattice and in addition the vertex stabilizer of each (type of) vertex is cyclic.

The goal of this section is to understand in detail the relationship between Singer cyclic lattices and difference matrices. Throughout the section we fix a prime power $q$ and, as always, put $\delta = q^2+q+1$. A \emph{difference matrix} is a $(q+1) \times 3$-matrix with entries in $\Z/\delta\Z$ such that each column (as a set) forms a difference set. A difference matrix is \emph{based} if some row is $(0,0,0)$. Difference matrices naturally arise in the following way:

\begin{lemma}
\label{lem:canonical_difference_matrix}
Let $\Gamma$ be a Singer cyclic lattice on $X$, let $x_0$, $x_1$, $x_2$ be the vertices of a chamber, and let $\sigma_j$ generate $\Gamma_{x_j}$ for $j \in \{0,1,2\}$. Associated to these data is a based difference matrix $E$ such that the relations
\begin{equation}
\label{eq:difference_cycle}
\sigma_0^{E_{i,0}} \sigma_1^{E_{i,1}} \sigma_2^{E_{i,2}} = 1 \quad 1 \le i \le q+1
\end{equation}
hold in $\Gamma$. The matrix is uniquely determined up to row permutations. The relations
\begin{equation}
\label{eq:order_relation}
\sigma_j^\delta = 1 \quad 0 \le j \le 2
\end{equation}
also hold.
\end{lemma}

When $\Gamma$ acts as a Singer cyclic lattice on $X$ and $\sigma_0, \sigma_1, \sigma_2$ are the generators of the vertex stabilizers of a chamber, we will call the triple $(\sigma_0, \sigma_1, \sigma_2)$ a \emph{chamber triple}.

\begin{proof}
Let $q+1$ be the order of the building $X$. Then the link of an edge has $q+1$ elements and the link of a vertex has $v = q^2 + q + 1$ vertices of each type.

In order for a relation $\sigma_0^{e_0} \sigma_1^{e_1} \sigma_2^{e_2} = 1$ to hold, it is necessary that $x \defeq \sigma_0^{e_0}.x_1 = \sigma_0^{e_0} \sigma_1^{e_1}.x_1 = \sigma_2^{-e_2}.x_1$. Clearly $x$ is adjacent to both $x_0$ and $x_2$. Since $\Gamma$ is regular on panels of type $\{0,1\}$ as well as on panels of type $\{1,2\}$, such elements $\sigma_0^{e_0}$ and $\sigma_2^{-e_2}$ do in fact exist for any vertex $x$ in the link of $[x_0,x_2]$ and are uniquely determined by it. Finally $\sigma_1^{e_1} = \sigma_0^{-e_0} \sigma_2^{-e_2} \in \Gamma_1$ is uniquely determined by the other two elements. The zero row is the one corresponding to the vertex $x = x_1$.

The second statement is clear because $\Gamma_{x_j}$ is cyclic and acts regularly on the points of each type in the link of $x_j$.\qed
\end{proof}

\begin{remark}
The columns of difference matrices will be indexed by numbers $0, 1, 2$ throughout (as in Lemma~\ref{lem:canonical_difference_matrix}). These are really representatives for the type set $\Z/3\Z$ so the reader who prefers positive indices may just identify the indices $0$ and $3$.
\end{remark}

The following is Essert's classification result in the formulation for Singer cyclic lattices, see \cite[Theorems~5.6,~5.8]{essert13}:

\begin{theorem}[Essert's theorem]
\label{thm:essert}
If $\Gamma$ is a Singer cyclic lattice then the generators $\sigma_j$ together with the relations \eqref{eq:difference_cycle} and \eqref{eq:order_relation} form a presentation for $\Gamma$.

Conversely if $E$ is a based difference matrix then the group presented by the generators $\sigma_j$ subject to the relations \eqref{eq:difference_cycle} and \eqref{eq:order_relation} is a Singer cyclic lattice.
\end{theorem}

Another result that will be of fundamental importance throughout the paper is the following, see \cite[Theorem~1.1.3]{klelee97}, \cite[Theorem~III]{krawei14}:

\begin{theorem}[QI-rigidity for buildings]
\label{thm:qi_rigidity}
Let $X$ and $Y$ be thick, irreducible Euclidean buildings of dimension at least $2$.
\begin{enumerate}
\item Any quasi-isometry $X \to Y$ is at bounded distance from an isomorphism.
\item Two isomorphisms $X \to Y$ that are at a bounded distance are the same.
\end{enumerate}
\end{theorem}

We call a triple $(\sigma_0, \sigma_1, \sigma_2)$ in an abstract group $\Gamma$ (without action) a \emph{presenting triple} if there is a difference matrix $E$ such that the generators $\sigma_0, \sigma_1, \sigma_2$ together with the relations \eqref{eq:difference_cycle} and \eqref{eq:order_relation} present $\Gamma$. Thus Essert's theorem can be phrased as saying that every chamber triple for $\Gamma \act X$ is a presenting triple in $\Gamma$ and that for every presenting triple in $\Gamma$ there is an action $\Gamma \act X$ for which it is a chamber triple. The following lemma allows us to completely drop the distinction.

\begin{lemma}
\label{lem:triples}
If $\Gamma \act X$ is a presenting triple in $\Gamma$ then it is a chamber triple for the action.
\end{lemma}

\begin{proof}
Let $(\sigma_i)_{0 \le i \le 2}$ be the presenting triple in question. The second part of Essert's theorem tells us that there is an action $\Gamma \act Y$ for which $(\sigma_i)_{0 \le i \le 2}$ is a chamber triple. Now $\Gamma$ acting cocompactly on $X$ as well as on $Y$ we obtain quasi-isometries
\[
X \leftarrow \Gamma \rightarrow Y\text{.}
\]
By Theorem~\ref{thm:qi_rigidity} there exists an isomorphism $\alpha \colon X \to Y$ at bounded distance from this quasi-isometry (in particular $X$ and $Y$ are isomorphic). Furthermore looking at the diagram
\begin{diagram}
X & \lTo & \Gamma & \rTo & Y\\
\dTo^{\sigma_i} && \dTo^{\sigma_i} && \dTo^{\sigma_i}\\
X & \lTo & \Gamma & \rTo & Y
\end{diagram}
we see that $\sigma_i \circ \alpha$ and $\alpha \circ \sigma_i$ are two isomorphisms at bounded distance from each other. By the uniqueness statement of Theorem~\ref{thm:qi_rigidity} this means they are equal. In other words the actions of $\sigma_i$ on $X$ and on $Y$ are conjugate to each other via $\alpha$, irrespective of $i$. Thus if $(\sigma_i)_{0 \le i \le 2}$ is a chamber triple for $\Gamma \act Y$ then it is one for $\Gamma \act X$.\qed
\end{proof}

As a corollary we obtain the following rigidity statement.

\begin{proposition}
\label{prop:recover_building}
Let $\Gamma$ be a Singer cyclic lattice. The building $X$ that $\Gamma$ acts on can be recovered (up to $\Gamma$-equivariant isomorphism) from $\Gamma$. In particular, $\Aut(\Gamma)$ is naturally a subgroup of $\Aut(X)$.
\end{proposition}

\begin{proof}
Every finite subgroup of $\Gamma$ needs to fix a point of $X$ by the Bruhat--Tits fixed point theorem \cite[Lemme 3.2.3]{brutit72} (see also \cite[Corollary~II.2.8]{brihae}). But the action preserves types and is free on panels, so the only fixed point sets of non-trivial subgroups of $\Gamma$ are vertices. Hence every maximal finite subgroup of $\Gamma$ is a vertex stabilizer.

Since the action of $\Gamma$ preserves types and is transitive on vertices of each type, there are three conjugacy classes of maximal finite subgroups, one for each vertex type.

It remains to recover edges or equivalently (since $X$ is flag and every edge is contained in a chamber) the chambers. By Lemma~\ref{lem:triples} the vertices fixed by the maximal finite subgroups $\gen{\sigma_0}$, $\gen{\sigma_1}$, $\gen{\sigma_2}$ span a chamber if and only if $(\sigma_i)_{0 \le i \le 2}$ is a presenting triple.\qed
\end{proof}

\begin{remark}
It would be nice to find a more elementary way to recover the edges. One might expect that $\gen{\sigma_0}$ and $\gen{\sigma_1}$ fix adjacent vertices if any only if there is a vertex such that $\gen{\sigma_0}.v \cap \gen{\sigma_1}.v$ has at least $q + 1$ elements (the only ``only if'' part is clear by construction). However we have not managed to prove this.
\end{remark}

Whenever a Singer cyclic lattice $\Gamma$ is equipped with a presenting triple $(\sigma_0, \sigma_1, \sigma_2)$ (for example because it is given by a difference matrix), we equip the associated building $X$ with a type function such that the fixed point set of $\sigma_i$ has type $i$.

Before we move on, we record two basic facts.

\begin{observation}
Let $E$ be a based difference matrix and let $\Gamma$ be the associated Singer cyclic lattice. Then $H_1(\Gamma) = (\Z/\delta\Z)^3/\im E^T$.
\end{observation}

\begin{proof}
This follows by abelianizing the presentation \eqref{eq:difference_cycle}, \eqref{eq:order_relation}.\qed
\end{proof}

A similar argument shows that a Singer lattice has trivial center, but more is true:

\begin{lemma}[{\cite[Corollary~2.7]{capmon09b}}]
A Singer cyclic lattice has no non-trivial normal amenable subgroup.
\end{lemma}

We have seen that a Singer cyclic lattice together with a presenting triple determines a based difference matrix and conversely a difference matrix gives rise to a Singer cyclic lattice with a distinguished presenting triple. We want to pin down this correspondence more precisely mainly to obtain two pieces of information: the isomorphism classes of Singer cyclic lattices, and for each Singer cyclic lattice its automorphism group. For this purpose it will be useful to employ the language of groupoids. We refer to \cite{higgins71} as a reference but we will not need much theory. The only non-trivial concept that we will make use of is that of a quotient groupoid, see \cite[Chapter~12]{higgins71}. Rather than introducing the concepts in general, we will discuss an elementary example.

\begin{example}
\label{exmpl:action_groupoid}
A group $H$ acting on a set $M$ gives rise to a groupoid ${}^HM$ in an obvious way: the objects of ${}^HM$ are the elements of $M$ and the morphisms $m \to n$ are the elements $h \in H$ with $h.m = n$. If $h$ is one of them, the set of these elements $H_{m,n}$ equals $h H_m$ and also $H_n h$. In particular, the automorphisms group of an object $m$ is just its stabilizer $H_m$.

Now suppose that $K < H$ is a subgroup with the property that the normal span of all stabilizers $\gen{K_m^H, m \in M}$ is contained in $K$ (i.e.\ an automorphism in ${}^KM$ conjugated by a morphism in ${}^HM$ is in ${}^KM$). Then there is a groupoid $K \backslash {}^HM$ whose elements are orbits in $K \backslash M$ and the morphisms $Km \to Kn$ are equivalence classes of morphisms modulo precomposition by $K_m$ or (equivalently) postcomposition by $K_n$ (the equivalence comes from the fact that if $h.m = n$ then ${}^hK_m = K_n$). Thus the set of morphisms corresponds to $H_{m,n}/K_m = K_n \backslash H_{m,n}$. The assumption on $K$ asserts that composition is well-defined.

An important special case is when the action of $K$ on $M$ is free. Then the normality condition of the last paragraph is automatically satisfied. Moreover, the natural map on morphisms $\hom_{{}^HM}(m,n) \to \hom_{K \backslash{}^HM}(K_m m, K_n n)$ is a bijection.
\end{example}

The phenomenon in the last paragraph of the example generalizes as follows:

\begin{observation}
\label{obs:free_action_groupoid}
Let $\Groupoid$ be a groupoid and let $\Normal$ be a full subgroupoid of $\Groupoid$ (meaning it contains all the objects of $\Groupoid$) that has no automorphisms other than identity morphisms. Then $\Normal$ is normal in $\Groupoid$ and the quotient map $q \colon \Groupoid \to \Normal\backslash\Groupoid$ induces isomorphisms $\hom_{\Groupoid}(x,y) \to \hom_{\Normal \backslash \Groupoid}(q(x), q(y))$.
\end{observation}

The first groupoid we want to construct will be denoted $\SingTrip$.

The objects of $\SingTrip$ are equivalence classes of pairs $(\Gamma,T)$ where $\Gamma$ is a Singer cyclic lattice and $T \subseteq \Gamma$ is a presenting triple. (The reader who is concerned with this not becoming a small category will fix his or her favorite countable set and take ``Singer cyclic lattice'' to mean ``group structure isomorphic to a Singer cyclic lattice on that set''.) Two pairs $(\Gamma,T)$ and $(\Gamma',T')$ are equivalent if there is an isomorphism $\alpha \colon \Gamma \to \Gamma'$ that takes $T$ to $T'$ (as ordered triples). Note that such a morphism, if it exists, is unique since $T$ and $T'$ generate $\Gamma$. Note also, that it exists precisely if $(\Gamma,T)$ and $(\Gamma,T')$ give rise to the same difference matrix. We denote the equivalence class of $(\Gamma,T)$ by $[\Gamma,T]$.

A morphism consists of a Singer cyclic lattice $\Gamma$ together with two presenting triples $T$ and $T'$ and we write them as $\Gamma, T \to \Gamma,T'$. The symbols $\Gamma,T \to \Gamma,T'$ and $\Lambda, S \to \Lambda, S'$ represent the same morphism if there is an isomorphism $\alpha \colon \Gamma \to \Lambda$ that takes $T$ to $S$ and $T'$ to $'S$ (again, as ordered tuples). We denote the morphism defined by $\Gamma,T \to \Gamma,T'$ as $[\Gamma,T \to \Gamma,T']$.

Composition is defined by $[\Gamma, T' \to \Gamma,T''] \circ [\Gamma,T \to \Gamma,T'] = [\Gamma,T \to \Gamma,T'']$. Thus in terms of general representatives, if $[\Gamma,T \to \Gamma,T']$ and $[\Lambda,S \to \Lambda,S']$ are morphisms with $[\Gamma,T'] = [\Lambda,S]$ then the product $[\Lambda,S \to \Lambda,S'] \circ [\Gamma,T \to \Gamma,T']$ equals $[\Lambda,\alpha(T) \to \Lambda,S'] = [\Gamma,T \to \Gamma,\alpha^{-1}(S')]$ where $\alpha \colon \Gamma \to \Lambda$ is the isomorphism taking $T'$ to $S$.

A good way to think about this groupoid is as follows: if $\alpha = [\Gamma,T \to \Gamma,T']$ is an automorphism in $\SingTrip$, i.e.\ if $[\Gamma,T] = [\Gamma,T']$, then $\alpha$ in fact represents an automorphism of $\Gamma$, namely the one taking $T$ to $T'$. Thus $\SingTrip$ naturally contains the groupoid consisting of automorphism groups of Singer cyclic lattices. The morphisms between different objects may be thought of as changing the presenting triple.

\begin{observation}
For every object $[\Gamma,T]$ in $\SingTrip$ we have $\Aut_\SingTrip([\Gamma,T]) \cong \Aut(\Gamma)$.
\end{observation}

Second we consider a subgroupoid $\Inner$ of $\SingTrip$ which has the same objects but whose only morphisms are inner isomorphisms. That is, $[\Gamma,T \to \Gamma,T']$ is in $\Inner$ if and only if there is a $g \in \Gamma$ with $T^g = T'$.

\begin{lemma}
The subgroupoid $\Inner$ is normal in $\SingTrip$.
\end{lemma}

\begin{proof}
We have to verify that if $\alpha = [\Gamma,T \to \Gamma,T'] \in \SingTrip$ and $\beta = [\Gamma,T'' \to \Gamma,T] \in \Inner$ with $[\Gamma,T] = [\Gamma,T'']$ then $\alpha \beta \alpha^{-1} \in \Inner$. Let $g \in \Gamma$ be such that ${T''}^g = T$. Then $\alpha \beta \alpha^{-1} = [\Gamma,T' \to \Gamma,{T'}^g] \in \Inner$.\qed
\end{proof}

The next step is to see how much larger $\SingTrip$ is than $\Inner$. We have

\begin{lemma}
\label{lem:extra_morphisms}
Any morphism $\alpha$ in $\SingTrip$ can be decomposed uniquely as  $\alpha = \alpha_4 \alpha_3 \alpha_2 \alpha_1$ where
\begin{enumerate}
\item $\alpha_1 \in \Inner$, so $\alpha_1 = [\Gamma, (\sigma_0,\sigma_1,\sigma_2) \to \Gamma, (\sigma_0^g, \sigma_1^g, \sigma_2^g)]$ for some $g \in G$,
\item $\alpha_2 = [\Gamma, (\sigma_0,\sigma_1,\sigma_2) \to \Gamma, (\sigma_{\pi^{-1}(0)},\sigma_{\pi^{-1}(1)},\sigma_{\pi^{-2}(2)})]$ for some $\Gamma$ and $\pi \in S_3$,
\item $\alpha_3 = [\Gamma, (\sigma_0,\sigma_1,\sigma_2) \to \Gamma, (\sigma_0^{\sigma_1^{e_1}}, \sigma_1, \sigma_2)]$ $\Gamma$ where $e_1$ is an entry in the first column of the difference matrix of $\Gamma$ with respect to $(\sigma_0,\sigma_1,\sigma_2)$.
\item $\alpha_4 = [\Gamma, (\sigma_0,\sigma_1,\sigma_2) \to \Gamma, (\sigma_0^{e_0}, \sigma_1^{e_1}, \sigma_2^{e_2})]$ $\Gamma$ and $(e_0, e_1,e_2) \in (\Z/\delta\Z^\times)^3$,
\end{enumerate}
\end{lemma}

Of course there is nothing special about the $0$th element being conjugated by the first in the definition of $\alpha_3$, we just make a choice for definiteness.

\begin{proof}
Let $\alpha = [\Gamma,T \to \Gamma,T']$ with $T = (\sigma_i)_{0 \le i \le 2}$ and $T' = (\sigma_i')_{0 \le i \le 2}$. We will successively multiply by elements as described in the statement to simplify $\alpha$ until in the end we are left with the identity morphism.

Using Lemma~\ref{lem:triples} we may think of $\Gamma$ as acting on a building $X$. In particular, each triple determines an ordered tuple of vertices of a chamber (the fixed point sets of the triple). At first we will look at what happens to the chamber (as a coarser version of $T$ as it were). 
Since $\Gamma$ acts regularly on chambers of each type, after multiplying by an inner automorphism we may assume that the vertices fixed by $\sigma_1$ and $\sigma_2$ are also fixed vertices of two of the $\sigma_i'$.

After permuting $T'$ by an appropriate $\alpha_2$ we may assume that the fixed points of $\sigma_1$ and $\sigma_1'$ and of $\sigma_2$ and $\sigma_2'$ are the same; in particular, $\alpha$ is type-preserving at this point.

Let $v_0$ be the vertex fixed by $\sigma_0$ and let $v_0'$ be the vertex fixed by $\sigma_0'$. Note that both $v_0$ and $v_0'$ are adjacent to the vertices fixed by $\sigma_1$ and by $\sigma_2$. Thus, since $\gen{\sigma_1}$  and $\gen{\sigma_2}$ act regularly on vertices of each type in their link, there are unique elements $e_1$ and $e_2$ with $\sigma_1^{-e_1}.v_0 = v_0'$ and $\sigma_2^{e_2}.v_0 = v_0'$. It follows that $\sigma_1^{e_1} \sigma_2^{e_2} \in \gen{\sigma_0}$, say it is $\sigma_0^{-e_0}$ and so $(e_0, e_1, e_2)$ is a row of the difference matrix of $\Gamma$ with respect to $T$. Thus composing with an appropriate $\alpha_3$ we may assume that the chambers of fixed points of $T$ and of $T'$ coincide.

Finally replacing each $\sigma_i$ by a different generator in its span using $\alpha_4$ we achieve that $T = T'$.

It remains to verify uniqueness. Note that $\pi$ can be recovered from the action on types: if we assign types to the vertices of $X$ in such a way that the fixed vertex of $\sigma_i$ has type $i$ then $\pi(j)$ is the type of the fixed point vertex of $\gen{\sigma_j'}$. Since $\alpha_1$, $\alpha_3$, and $\alpha_4$ preserve types, this determines $\alpha_2$.

Similarly, $\alpha_3$ can be recovered from chamber orbits: there are $q+1$ orbits of chambers of $X$ under the action of $\Gamma$. Each of the $q+1$ choices of $e_1$ in $\alpha_3$ replaces the chamber corresponding to $T$ by one in a different orbit. Since $\alpha_1$, $\alpha_2$, and $\alpha_4$ preserve the chamber orbit, this determines $\alpha_3$.

It remains to show that $\alpha_4\alpha_1 = \id$ then $\alpha_1$ and $\alpha_4$ are trivial. If conjugation by $g$ leaves $\gen{\sigma_i}$ invariant then $g$ fixes the chamber whose vertices a fixed by the $\gen{\sigma_i}$. Since $\Gamma$ acts freely on chambers we deduce that $g = 1$.\qed
\end{proof}

With $\Inner$ being normal in $\SingTrip$ there is a canonical quotient groupoid and it is what we are after. We first describe what we claim to be the quotient groupoid and then prove that it is what we claim in the end. The following Lemma gives a good idea of where we are heading.

\begin{lemma}
\label{lem:equivalence}
Let $\Gamma$ be a Singer cyclic lattice acting on $X$. Let $T$ be a presenting triple in $\Gamma$ and let $E$ be the associated based difference matrix
\begin{enumerate}
\setcounter{enumi}{-1}
\item which is only well defined up to permutation of rows.\label{item:row_permutation}
\end{enumerate}
The morphisms in Lemma~\ref{lem:extra_morphisms} have the following effect on $E$.
\begin{enumerate}
\item $\alpha_1$ has no effect on $E$.\label{item:new_chamber_in_orbit}
\item $\alpha_2$ amounts to permuting the columns of $E$ by $\pi$.\label{item:column_permutation}
\item $\alpha_3$ amounts to subtracting the row $(e_0, e_1, e_2)$ of $E$ from all rows of $E$.\label{item:row_subtraction}
\item $\alpha_4$ amounts to multiplying the $i$th column of $E$ by $(e_i + \delta\Z)^{-1}$.\label{item:new_generator}
\end{enumerate}
\end{lemma}

\begin{proof}
Only the point concerning $\alpha_3$ needs justification. Let $(e_0, e_1, e_2)$ be the row of $E$ concerning the exponent in question and let $(f_0,f_1,f_2)$ be any row of $E$. We compute the relation
\begin{align*}
\sigma_2^{-e_2} \sigma_1^{-e_1} \sigma_0^{-e_0} \sigma_0^{f_0} \sigma_1^{f_1} \sigma_2^{f_2} & \approx \sigma_1^{-e_1} \sigma_0^{f_0-e_0} \sigma_1^{f_1} \sigma_2^{f_2-e_2}\\
& = (\sigma_0^{f_0-e_0})^{\sigma_1^{e_1}} \sigma_1^{f_1-e_1} \sigma_2^{f_2-e_2}\\
& = (\sigma_0^{(\sigma_1^{e_1})})^{f_0-e_0} \sigma_1^{f_1-e_1} \sigma_2^{f_2-e_2}
\end{align*}
where $\approx$ means equality up to cyclic permutation (we only care that the expressions are relators). We see that the difference matrix associated to the generators $\sigma_0'$, $\sigma_1$, $\sigma_2$ is obtained from $E$ by subtracting the row $(e_0, e_1, e_2)$ from all rows. In particular, this row becomes the new zero row.\qed
\end{proof}

Let $\DiffMatBSet$ be the set of all based difference matrices (for the fixed parameter $q$). The effects of the morphisms $\alpha_2$ and $\alpha_4$ give rise to a right action of the wreath product $C \defeq (\Z/\delta\Z^\times)^3 \rtimes S_3$ on $\DiffMatBSet$ ($C$ for acting on columns). This action can be thought of as multiplication by $3$-by-$3$ monomial matrices over $\Z/\delta\Z$ from the right.

Similarly \eqref{item:row_permutation} gives rise to an action of $S_{q+1}$ on $\DiffMatBSet$. In terms of matrices it is multiplication by $(q+1)$-by-$(q+1)$ permutation matrices from the left. We let $\bar{R}$ denote $S_{q+1}$ acting on $\DiffMatBSet$ in this way. In order to take $\alpha_3$ into account we note the following:

\begin{lemma}
\label{lem:big_symmetric}
The subgroup of $\GL_{q+1}(\Z/\delta\Z)$ generated by the matrices
\[
P_i \defeq
\begin{pmatrix}
I_{i-1}&&&\\
& 0 & 1 &\\
& 1 & 0 &\\
& &  &I_{q-i}\\
\end{pmatrix}
\text{ and }
M_i \defeq
\begin{pmatrix}
I_{i-1}& \vdots  & \\
 & -1 & \\
& \vdots& I_{q-i+1}\\
\end{pmatrix},
\]
is isomorphic to $S_{q+2}$.
\end{lemma}

\begin{proof}
The natural action of $S_{q+2}$ on $(\Z/\delta\Z)^{q+2}$ preserves the submodule of those vectors $(x_1, \ldots, x_{q+2})$ that satisfy $\sum_{x_i} = 0$. This submodule has a basis consisting of the vectors $(\ldots, 0, 1, 0, \ldots, 0, -1)$. The standard generators with respect to this basis are $P_i, 1 \le i \le q$ and $M_{q+1}$.\qed
\end{proof}

Note that the matrices $M_i$ in Lemma~\eqref{lem:big_symmetric} correspond to the effect of $\alpha_3$ in Lemma~\ref{lem:equivalence}. We define $R$ to be $S_{q+2}$ acting on $\DiffMatBSet$ as above. So all the operations of Lemma~\ref{lem:equivalence} are induced by an action of $R \times C$ on $\DiffMatBSet$.

The third groupoid we want to consider will be denoted $\DiffMat$. It is defined to be $\bar{R} \backslash{}^{R \times C}\DiffMatBSet$ in the notation of Example~\ref{exmpl:action_groupoid} (note that since $C$ acts from the right, its elements need to be inverted). Thus objects are orbits in $\bar{R} \backslash \DiffMatBSet$ and morphisms are induced by the action of $R \times C$.

We now define a groupoid map $q \colon \SingTrip \to \DiffMat$. On objects it takes $[\Gamma,T]$ to the difference matrix associated to the pair $(\Gamma,T)$ by Lemma~\ref{lem:canonical_difference_matrix}. On morphisms, it takes $\alpha_4\alpha_3\alpha_2\alpha_1$ as in Lemma~\ref{lem:extra_morphisms} to the morphism associated to it via Lemma~\ref{lem:equivalence}.

\begin{theorem}
\label{thm:groupoid_exact_sequence}
The sequence
\[
\Inner \lhd \SingTrip \stackrel{q}{\to}  \DiffMat\text{.}
\]
is exact in the sense that $q$ is the quotient morphism of the normal inclusion on the left.
\end{theorem}

\begin{proof}
First we need to check that $\Inner$ is the kernel of $q$. That $\Inner$ lies in the kernel is clear by construction. The converse is true by the uniqueness statement in Lemma~\ref{lem:equivalence}.

Using \cite[Proposition~25]{higgins71} it remains to show that two objects in $\SingTrip$ having same image in $\DiffMat$ lie in the same component of $\Inner$ and that $q$ is surjective (meaning that any morphism in $\DiffMat$ is the image of a morphism in $\SingTrip$). The first part is clear: if $[\Gamma,T]$ and $[\Gamma',T']$ have same image in $\DiffMat$ then $[\Gamma,T] = [\Gamma,T']$. For surjectivity we just have to read Lemma~\ref{lem:equivalence} backwards: let $E$ be a difference matrix and let $c \defeq ((e_0,e_1,e_2), \pi) \in C$ be an element. We take $(\Gamma,(\sigma_0,\sigma_1,\sigma_2))$ to be a Singer lattice with presenting triple obtained from $E$ via Essert's theorem. Now $[\Gamma,(\sigma_0,\sigma_1,\sigma_2) \to \Gamma, (\sigma_{\pi^{-1}(0)}^{\pi^{-1}(e_0^{-1})}, \sigma_{\pi^{-1}(1)}^{\pi^{-1}(e_1^{-1})}, \sigma_{\pi^{-1}(2)}^{\pi^{-1}(e_2^{-1})})]$ has image the morphism induced by $q$. Similarly the morphisms induced by $R$ are images of the morphisms $[\Gamma, (\sigma_0,\sigma_1,\sigma_2) \to \Gamma, (\sigma_0^{\sigma_1^{e_1}},\sigma_1,\sigma_2)]$ where $e_1$ ranges over the $1$st column of $E$. Since these morphisms generate the component of $\bar{R}.E$, this shows surjectivity.\qed
\end{proof}

From this discussion of groupoids we draw two concrete conclusions about Singer lattices. We say that two based difference matrices are \emph{equivalent} if they lie in the same $R \times C$ orbit. That is the set of equivalence classes is $R \backslash \DiffMatBSet /C$.

\begin{corollary}
\label{cor:count_classes_isom_equiv}
There is a bijective correspondence between isomorphism classes of Singer cyclic lattices and equivalence classes $R \backslash \DiffMatBSet / C$.
\end{corollary}

\begin{proof}
Isomorphism classes of Singer cyclic lattices correspond to components of $\SingTrip$, the set $R \backslash \DiffMatBSet / C$ corresponds to components of $\DiffMat$. Since a quotient morphism of groupoids establishes a bijective correspondence between components, the statement follows from Theorem~\ref{thm:groupoid_exact_sequence}.\qed
\end{proof}

To a difference matrix $E$ we want to associate a group $\Aut(E)$ of automorphisms. The natural candidate is the automorphism group of $\bar{R}.E$ in $\DiffMat$. We will take a further quotient to obtain a more handy description. Lemma~\ref{lem:diff_set_difference} implies that the action of $R$ on $\DiffMatBSet$ is free. Thus by Observation~\ref{obs:free_action_groupoid} we may consider the groupoid $R \backslash {}^{R \times C} \DiffMatBSet \cong {}^C(R \backslash \DiffMatBSet)$. In fact, the observation implies that the quotient $\DiffMat \to {}^{C} (R \backslash \DiffMatBSet)$ induces isomorphisms in automorphism groups, which on the right hand side are just subgroups of $C$. Concretely, this means that we can define $\Aut(E)$ to be the subgroup of $C$ that preserves the $R$-orbit of $E$.

\begin{table}
\caption{For small values of $q$: the parameter $\delta = q^2+q+1$, a difference set, the number of Singer cyclic lattices, and the bound (as a rounded decimal fraction) of Theorem~\ref{thm:bound}.}
\label{tab:parameters}
\begin{gather*}
\begin{array}{cr@{}lcrr}
q & \multicolumn{2}{c}{\delta} & \text{difference set}& \# & \text{bound}\\
2 & 7 &= 7 & (0, 1, 3) & 2 & 0.4\\
3 & 13 &= 13 & (0, 1, 3, 9) & 7 & 4.1\\
4 & 21 &= 3 \cdot 7 & (0, 1, 4, 14, 16) & 17 & 11.4\\
5 & 31 &= 31 & (0, 1, 3, 8, 12, 18) & 3269 & 3214.3\\
7 & 57 &= 3 \cdot 19 & (0, 1, 3, 13, 32, 36, 43, 52) &  & 10035961.9\\
8 & 73 &= 73 & (0, 1, 3, 7, 15, 31, 36, 54, 63) &  & 30105851.5\\
9 & 91 &= 7 \cdot 13 & (0, 1, 3, 9, 27, 49, 56, 61, 77, 81) &  & 10160648447.9\\
11 & 133 &= 7 \cdot 19 & (0, 1, 3, 12, 20, 34, 38, 81, 88, 94, 104, 109) &  & 1416311939759987.8\end{array}
\end{gather*}
\end{table}

\begin{corollary}
\label{thm:group_exact_sequence}
Let $E$ be a difference matrix and let $\Gamma$ be the associated Singer cyclic lattice. There is an exact sequence of groups
\[
\Inn(\Gamma) \to \Aut(\Gamma) \to \Aut(E)
\]
thus $\Out(\Gamma) \cong \Aut(E)$.
\end{corollary}

\begin{proof}
Let $T$ be the presenting triple of $\Gamma$ given by $E$. It follows from Theorem~\ref{thm:groupoid_exact_sequence} that there is an exact sequence
\[
\hom_{\Inner}([\Gamma,T],[\Gamma,T]) \to \hom_{\SingTrip}([\Gamma,T],[\Gamma,T]) \to \hom_{\DiffMat}(\bar{R}.E, \bar{R}.E)\text{.}
\]
The first two are naturally isomorphic to $\Inn(\Gamma)$ and $\Aut(\Gamma)$ by construction. The last one is naturally isomorphic to $\Aut(E)$ by the discussion above.\qed
\end{proof}


\section{Census}
\label{sec:census}

The goal of this section is to estimate the number of Singer cyclic lattices up to isomorphism using Corollary~\ref{cor:count_classes_isom_equiv}. For this purpose it will be helpful to take non-based difference matrices into view. Let $\DiffMatSet$ be the set of all difference matrices (based or not) for our fixed parameter $q$. Dropping the requirement that the difference matrices be based allows us to extend the action of $C$ to the supergroup $\tilde{C} \defeq \Aff_1(\Z/\delta\Z)^3 \rtimes S_3$.

\begin{lemma}
\label{lem:diff_mats_to_based_diff_mats}
There is a well-defined map
\[
R \backslash \DiffMatSet \to R \backslash \DiffMatBSet
\]
that induces a bijection
\[
R \backslash \DiffMatSet / \tilde{C} \to R \backslash \DiffMatBSet  / C\text{.}
\]
\end{lemma}

\begin{proof}
The map takes the $R$-orbit of $E$ to the $R$-orbit of the based difference matrix $E_0$ that is obtained from $E$ by subtracting some row from all rows. Clearly this new difference matrix is based. The map is well-defined because it does not matter, up to the $R$-action, which row one chooses: if $e$ and $f$ are rows of $E$ then
\[
E - f = (E - e) - (f-e)
\]
and $(f-e)$ is a row of $E-e$ (here $f-e$ means the componentwise difference of $f$ and $e$ and $E-f$ means $E$ with $f$ subtracted from each row).

Clearly adding a constant to some column of $E$ has no effect on $E_0$ thus we get a bijection
\[
R \backslash \DiffMatSet / (\Z/\delta\Z)^3 \to R \backslash \DiffMatBSet
\]
and quotienting further the bijection claimed in the statement.\qed
\end{proof}

We say that two difference matrices are \emph{equivalent} if they lie in the same $R \times \tilde{C}$-orbit. Note that this is consistent with the notion of equivalence for based difference matrices by Lemma~\ref{lem:diff_mats_to_based_diff_mats}: if two based difference matrices are equivalent modulo $R \times \tilde{C}$ then they already lie in the same $R \times C$-orbit.

The advantage of of this approach lies in the fact that it is easier to list representatives. To obtain representatives of based difference matrices modulo equivalence we would need to list difference sets modulo $(\Z/\delta\Z)^\times$, which is not so clear how to do. By contrast in order to obtain difference matrices up to equivalence we need to list difference sets modulo $\Aff_1(\Z/\delta\Z)$. At least for Desarguesian ones we know from Theorem~\ref{thm:difference_sets_affine} that there is only one. We call a difference matrix \emph{Desarguesian} if all of its columns are Desarguesian difference sets and get:

\begin{lemma}
Let $D$ be any Desarguesian difference set and let $E$ be a Desarguesian difference matrix. Then $E$ is equivalent to a difference matrix where each column is equal to $D$ as a set.
\end{lemma}

So let us take $D$ to be any Desarguesian difference set and let $A_D < \Aff_1(\Z/\delta\Z)$ be the stabilizer of $D$. Denoting by $\DiffMatSet(D)$ the set of difference matrices whose columns setwise equal $D$, inclusion induces a map
\begin{equation}
\label{eq:enumerating_difference_sets}
\bar{R} \backslash \DiffMatSet(D) / (A_D^3 \rtimes S_3) \to R \backslash \DiffMatSet / \tilde{C}\text{.}
\end{equation}

\begin{lemma}
\label{lem:diff_mat_representatives}
The map \eqref{eq:enumerating_difference_sets} is injective.
\end{lemma}

\begin{proof}
It is not hard to see that an element of $\tilde{C}$ taking one difference matrix in $\DiffMatSet(D)$ to another actually has to lie in $A_D^3 \rtimes S_3$. That two difference matrices in $\DiffMatSet(D)$ that differ by an element of $R$ actually differ by an element of $\bar{R}$ follows from Lemma~\ref{lem:diff_set_difference}.\qed
\end{proof}

We want to use \eqref{lem:diff_mat_representatives} to bound the number of difference matrices up to equivalence. Lemmas~\ref{lem:diff_mats_to_based_diff_mats} and~\ref{lem:diff_mat_representatives} imply the last two of the inequalities
\begin{equation}
\label{eq:prepare_bound}
\frac{\abs{S_{q+1} \backslash \DiffMatSet(D) / S_3}}{\abs{A_D^3}} \le \abs{\bar{R} \backslash \DiffMatSet(D) / (A_D^3 \rtimes S_3)} \le \abs{R \backslash \DiffMatSet / \tilde{C}}= \abs{R \backslash \DiffMatBSet  / C}\text{.}
\end{equation}
The next two lemmas compute the numerator respectively denominator of the left hand side of the inequality.

\begin{lemma}
\label{lem:matrix_count}
Let $D$ be a set of cardinality $k$. Let $\calC$ be the set of $k \times 3$-matrices where each column setwise equals $D$. Then $S_k \backslash \calC /S_3$ has cardinality
\begin{equation}
\label{eq:count}
\frac{1}{6}\left ((k!)^2 + 3k! + 2(\tau_k + 1)\right)
\end{equation}
where $\tau_k$ is the number of elements of order (exactly) $3$ in $S_k$.
\end{lemma}

\begin{proof}
Let $G = S_k \times S_3$ acting on $\calC$. The Burnside Lemma gives
\[
\abs{\calC/G} = \frac{\abs{\calC}}{\abs{G}} + \sum_{cG \in \calC/G} \left(1 - \frac{1}{\abs{G_c}} \right)
\]
where $G_c$ is the stabilizer of an element $c \in \calC$. We have to determine which of the $(k!)^3$ elements (or more precisely: which of their equivalence classes) have a non-trivial stabilizer.

We start by observing that the projection $\rho \colon G \to S_3$ always restricts to an injective morphism on $G_c$. Indeed, given two column indices $i$ and $j$ there is a unique permutation $\pi_{i,j} \in S_k$ that takes the $i$th column to the $j$th. We distinguish cases by what $\rho(G_c)$ is.

If $\rho(G_c)$ is all of $S_3$ then all columns of $c$ have to coincide. This is because there are $3$-cycles $(\pi, (i\,j\,k)) \in G_c$ which satisfy $\pi = \pi_{i,j} = \pi_{j,k} = \pi_{k,i}$ and $\pi^3 = \id$, and there are elements $(\pi_{i,j},(i\,j))$ satisfying $\pi_{i,j}^2 = \id$ from which we see that $\pi_{i,j} = \id$ for all $i,j$. Conversely if all columns coincide it is obvious that $\rho(G_c) = S_3$. There are $k!$ such elements in $\calC$, which form a single equivalence class.

If $\rho(G_c) = C_3$ then there is a permutation $\pi$ ($=\pi_{1,2} = \pi_{2,3} = \pi_{3,1}$) of order $3$ in $S_k$ such that $G_c$ is generated by $(\pi, (1\,2\,3))$. The number of these clements is $k! \cdot \tau_k$, and the number of their equivalence classes is $1/2 \cdot \tau_k$.

If $\rho(G_c) = C_2$ then $G_c$ is generated by an element $(\pi, (i\,j))$ of order $2$. Let $k$ be the third index. We see that $\pi_{i,j} = \pi = \pi_{k,k} = \id$ and thus that the $i$th and the $j$th columns have to coincide (but the $k$th must not). There are ${3 \choose 2} \cdot k! \cdot (k! - 1)$ of these elements and $(k!-1)$ of their equivalence classes.

The only remaining case is where $\rho(G_c)$ is trivial so $G_c$ is trivial and we need not count these.

Putting everything together, we find that
\[
\abs{\calC/G} = \frac{(k!)^3}{6k!} + \frac{5}{6} + \frac{2}{3} \cdot \frac{1}{2} \tau_k + \frac{1}{2}(k! -1)
\]
which simplifies to the expression in the statement.\qed
\end{proof}

\begin{lemma}
\label{lem:frobenius_order}
Write $q = p^\eta$ with $p$ prime. If $D \subseteq \Z/\delta\Z$ is a Desarguesian difference set then its stabilizer in $\Aff_1(\Z/\delta\Z)$ has order $3\eta$.
\end{lemma}

\begin{proof}
The normalizer of the cyclic Singer group $\F_{q^3}^\times/\F_q^\times$ in the full group of colineations $\PGaL_3(\F_{q}) = \PGL_3(\F_{q}) \rtimes \Gal(\F_{q}/\F_p)$ is $N \defeq \F_{q^3}^\times/\F_q^\times \rtimes \Gal(\F_{q^3}/\F_p)$. This can be seen using \cite[II.7.3a]{huppert67} which implies that the normalizer in $\PGL_3(\F_q)$ is $\F_{q^3}^\times/\F_q^\times \rtimes \Gal(\F_{q^3}/\F_q)$.

Identifying the Singer group with $\Z/\delta\Z$, the group $N$ is identified with the subgroup $\Z/\delta\Z \rtimes \gen{p}$ of $\Aff_1(\Z/\delta\Z)$. Let $x$ be a point and $\ell$ be a line such that $\sigma^i.x \inc L$ iff $i \in D$. Since $N$ acts by colineations, there is a line $\ell'$ such that $\sigma^{pi}.x \inc \ell'$ iff $i \in D$. Say $\ell' = \sigma^a.\ell$. Thus $\sigma^{pi-a}.x \inc \ell$ iff $i \in D$, showing that the subgroup of $\Z/\delta\Z \rtimes \gen{p}$ generated by $i \mapsto pi-a$ preserves $D$.\qed
\end{proof}

Combining Lemmas~\ref{lem:matrix_count},~\ref{lem:frobenius_order}, and~Corollary~\ref{cor:count_classes_isom_equiv} with \eqref{eq:prepare_bound} we get:

\begin{theorem}
\label{thm:bound}
Let $q = p^\eta$ with $p$ prime. The number of Singer cyclic lattices with parameter $q$ up to isomorphism is bounded below by
\[
\frac{1}{6 \cdot 27\eta^3}\left (((q+1)!)^2 + 3(q+1)! + 2(\tau_{q+1} + 1)\right)
\]
where $\tau_{q+1}$ is the number of elements of order precisely $3$ in $S_{q+1}$. In particular, the number grows super-exponentially with $q$.
\end{theorem}

Table~\ref{tab:parameters} shows that this bound quickly becomes relatively good, and also that it grows very fast.

For $q \in \{2,3,4\}$ representatives for difference matrices up to equivalence are shown in Tables~\ref{tab:lattices_2}, \ref{tab:lattices_3},~and~\ref{tab:lattices_4} together with some invariants of the associated Singer cyclic lattice.

\begin{table}
\caption{Singer cyclic lattices for $q = 2$. The table shows a difference matrix with fixed difference set, a based difference matrix, the outer automorphism group (which is the automorphism group of the difference matrix), the abelianization, and the abelianization of the commutator subgroup.}
\label{tab:lattices_2}
\begin{gather*}
\begin{array}{cccccc}
\text{Name} & \text{Diff mat} & \text{Based DM} & \textrm{Out}(\Gamma) & H_1(\Gamma) & H_1([\Gamma,\Gamma])\\
\input{lattices_2.tex}
\end{array}
\end{gather*}
\end{table}

\begin{table}
\caption{Singer cyclic lattices for $q = 3$. The columns are the same as in Table~\ref{tab:lattices_2}.}
\label{tab:lattices_3}
\begin{gather*}
\begin{array}{cccccc}
\text{Name} & \text{Diff mat} & \text{Based DM} & \textrm{Out}(\Gamma) & H_1(\Gamma) & H_1([\Gamma,\Gamma])\\
\input{lattices_3.tex}
\end{array}
\end{gather*}
\end{table}


\section{Bruhat--Tits examples}
\label{sec:brutit_examples}

This section as well as Section~\ref{sec:linearity} are devoted to investigating which of the Singer cyclic lattices act on Bruhat--Tits buildings. The Bruhat--Tits buildings relevant for us arise as follows. If $K$ is a local field and $\calO$ is its ring of integers, then there is a Bruhat--Tits building $X$ of type $\tilde{A}_2$ on which the group $\PGL_3(K)$ acts strongly transitively. Its vertices are homothety classes of $\calO$-lattices in $K^3$, and in particular $\PGL_3(\calO)$ is the stabilizer of a vertex. The full automorphism group of $X$ is $(\PGL_3(K) \rtimes C_2) \rtimes \Aut(K)$ where the non-trivial element of $C_2$ takes a matrix to its transpose inverse.

In this section we will describe for each $q$ one Singer cyclic lattice that is an arithmetic subgroup of $\PGL_3(\F_q\lseries{t})$ and in particular acts on the associated Bruhat--Tits building. The existence of lattices these was first noted by Tits~\cite[Section~3.1]{tits86b}. In \cite{caprumtho12} they were recovered in a similar way to how we will describe them below. The main point will therefore be to work out precisely which difference matrices produce these arithmetic lattices.

To start with, we need to recall a construction by Cartwright, Mantero, Steger, and Zappa \cite{carmanstezap93a} (see also \cite{carste98}) to construct a vertex regular lattice in $\PGL_3(K)$. We begin with a setup as in Section~\ref{sec:diff_sets}: let $q$ be a prime power and consider the field extension $\F_{\smash{q^3}}/\F_q$. Let $\xi$ be a generator of $\F_{\smash{q^3}}^*/\F_q^*$. Now take $K = \F_{q}(Y)$ to be the field of rational functions over $\F_q$ and $L = \F_{q^3}(Y)$ the one over $\F_{q^3}$. CMSZ consider the automorphism $\varphi$ of $L$ that takes $x \in \F_{q^3}$ to $x^q$ and fixes $Y$ and define the cyclic algebra $\alg = L[\sigma]$ with relations
\begin{equation}
\label{eq:relations}
\sigma^3 = 1 + Y\quad \text{and}\quad\sigma x \sigma^{-1} = \varphi(x)\text{.}
\end{equation}
(this is a central algebra over $K$). They define an element $b_1 \in \alg$ and for $u \in \F_{\smash{q^3}}^*/\F_q^*$ put $b_u = u b_1 u^{-1}$ to show:

\begin{theorem}
\label{thm:cmsz}
The algebra $\alg$ splits giving rise to an isomorphism
\[
\alg^*/Z(\alg)^* \to \PGL_3(\F_q\lseries{Y})\text{.}
\]
Under this isomorphism, the group $\Gamma_{\text{CMSZ}}$ generated by all the $b_u$ acts regularly on vertices of the associated building. More specifically, there is a vertex $v_0$ such that the maps $u \mapsto b_u.v_0$, $u^\perp \mapsto b_{u^{-1}}.v_0$ induce an isomorphism of the projective plane $\Pi(\F_{\smash{q^3}})$ with $\lk(v_0)$.
\end{theorem}

A vertex regular lattice as above gives rise to the following data: a polarity $\lambda \colon \lk(v_0) \to \lk(v_0)$ defined by $\lambda(b_u.v_0) = b_u^{-1}.v_0$ and a \emph{triangle presentation} $\mathcal{T}$ which is just a set of triples $(u,v,w)$ such that $b_ub_vb_w = 1$. The lattice $\Gamma$ then has a presentation
\[
\gen{b_u, u \in \F_{\smash{q^3}}^*/\F_q^* \mid b_ub_vb_w = 1, (u,v,w) \in \mathcal{T}}\text{.}
\]
Let $\Tr \defeq \Tr_{\F_{\smash{q^3}}/\F_q}$ denote the relative trace and $(x,y) = \Tr(xy)$ the associated bilinear form. In Theorem~\ref{thm:cmsz} the involution (seen as a map of $\Pi(\F_{\smash{q^3}})$) is $\lambda(u\F_q) = u^\perp\F_q$ and the triangle presentation consists of triples $(u,u\zeta,u\zeta^{q+1})$ where $u,\zeta \in \F_{\smash{q^3}}^*$ and $\Tr(\zeta) = 0$. In other words, $(u,v,w) \in \mathcal{T}$ if and only if $\Tr(v/u) = 0$ and $w u^q = v^{q+1}$.

Note that from the theorem it is obvious that $\xi$ acts a Singer cycle on $\lk(v_0)$. Thus it is natural to look at the supergroup $\hat{\Gamma} \defeq \gen{\xi, b_1}$ of $\Gamma_{\text{CMSZ}}$. One expects that the normal closure $\Gamma \defeq \gen{\xi^{\hat{\Gamma}}}$ is a Singer cyclic lattice. This is indeed the case, and not surprisingly, their difference matrices are the most symmetric ones:

\begin{corollary}
\label{cor:cartwright_to_panel}
The group $\Gamma$ is a Singer cyclic lattice with difference matrix $E = (d\ d\ d)_{d \in D}$ where $D$ is a Desarguesian difference set.
\end{corollary}

\begin{proof}
First we note that
\begin{equation}
\label{eq:xi_action}
\xi^k b_u \xi^{-k} = b_{\xi^k u}\text{.}
\end{equation}
Thus $\xi^k$ takes $u.v_0$ to $(\xi^ku).v_0$ and in particular acts as a Singer cycle on $\lk v_0$. Since $\Gamma_{\text{CMSZ}}$ is vertex-transitive, it follows that $\Gamma$ contains a Singer cycle around every vertex.

To check that it is a Singer cyclic lattice and to determine the difference matrix we define the vertices $w_1 \defeq b_1.v_0$ and let $v_2 \defeq b_1^{-1}.v_0$. Note that $\xi^k.w_1$ is adjacent to $v_2$ if and only if $\xi^k\F_q$ is adjacent to $\F_q^\perp$ if and only if $\Tr(\xi^k) = 0$. The set $D'$ of those $k$ is the classical Singer difference set of $\F_{\smash{q^3}}/\F_q$.

Let $\zeta \in \F_{\smash{q^3}}^*/\F_q^*$ be such that $\Tr(\zeta) = 0$ and put $v_1 \defeq \zeta.w_1 = b_{\zeta}.v_0$. Note that $v_0$, $v_1$, and $v_2$ are pairwise adjacent. Thus $\Gamma$ is generated by
\[
t_1 \defeq \xi \quad t_2\defeq b_\zeta \xi b_\zeta^{-1} \quad \text{and} \quad t_3 \defeq b_1^{-1} \xi b_1\text{.}
\]
We have to determine the tuples $(k,\ell,m)$ such that $t_1^k t_2^\ell t_3^m = 1$. Expanding and using the triangle presentation we find that
\[
t_1^kt_2^\ell t_3^m = \xi^k b_\zeta \xi^\ell b_\zeta^{-1} b_1^{-1} \xi^m b_1 = \xi^k b_\zeta \xi^\ell b_{\zeta^{q+1}} \xi^m b_1\text{.}
\]
We permute cyclically and apply \eqref{eq:xi_action} to get
\[
b_1 \xi^k b_\zeta \xi^\ell b_{\zeta^{q+1}} \xi^m = b_1 b_{\xi^k\zeta}b_{\xi^{\ell+k}\zeta^{q+1}} \xi^{m+\ell+k}\text{.}
\]
The power of $\xi$ on the right fixes $v_0$, so for the whole expression to be trivial, the left part also has to fix $v_0$. But the left part lies in $\Gamma$, so if it fixes $v_0$ it is trivial. Thus the whole expression is trivial if and only if the two factors are trivial individually. This is equivalent to the following relations:
\begin{align*}
\xi^{m+\ell+k} &\in \F_q\\
(1,\xi^k\zeta,\xi^{\ell+k}\zeta^{q+1}) & \in \mathcal{T}\text{.}
\end{align*}
If we let $r \in D'$ be such that $\zeta = \xi^r$ we we can write these as
\begin{align*}
m+\ell+k &\equiv 0 \mod q^2+q+1\\
k + r &\in D'\\
(q+1)k &\equiv (\ell+k) \mod q^2 + q + 1\text{.}
\end{align*}
The solutions are triples $(k,\ell,m) = (k, qk, (q + 1)k)$ with $k \in -r + D'$. Now $q$ and $q + 1$ are relatively prime to $q^2 + q + 1$ (since $q^2 + q + 1 = q(q+1) + 1$) so we may replace the above Singer cycles by
\[
s_1 \defeq t_1 \quad s_2 \defeq t_2^{q} \quad \text{and} \quad s_3 \defeq t_3^{q+1}
\]
and the difference set by $D \defeq -r + D'$.\qed
\end{proof}

Note that any two difference matrices of the form $E = (d\ d\ d)_{d \in D}$ with $D$ a Desarguesian difference set are equivalent. So the corollary says that the unique equivalence class of difference matrices that has representatives with all rows constant (and the difference set Desarguesian) parametrizes Bruhat--Tits lattices.

\section{Relation to chamber regular lattices}
\label{sec:ronan}

This section is a slight deviation from the general train of thought. After having made the connection with the vertex regular lattices studied by CMSZ in the last section, we study a similar relationship in this section. Ronan \cite{ronan84}, Tits~\cite{tits85,tits86b} and Köhler--Meixner--Wester \cite{koemeiwes84a} (see also \cite[Theorem~C.3.6]{kantor86}) have introduced four lattices that act chamber regularly (preserving types) on buildings of type $\tilde{A}_2$ and of order $2$. They are:

\begin{align*}
\Gamma_{R1} & = \gen{a,b,c \mid a^3 = b^3 = c^3 = 1, (ab)^2 = ba, (bc)^2 = cb, (ca)^2 = ac}\\
\Gamma_{R2} & = \gen{a,b,c \mid a^3 = b^3 = c^3 = 1, (ab)^2 = ba, (bc)^2 = cb, (ac)^2 = ca}\\
\Gamma_{R3} & = \gen{a,b,c \mid a^3 = b^3 = c^3 = 1, (ab)^2 = ba, (c^2b)^2 = bc^2, (ac)^2 = ca}\\
\Gamma_{R4} & = \gen{a,b,c \mid a^3 = b^3 = c^3 = 1, (ab)^2 = ba, (bc^2)^2 = c^2b, (ac)^2 = ca}\\
\end{align*}

Besides these there is one other known family of chamber-regular lattices on $\tilde{A}_2$-buildings. Tits \cite[Section~3.2]{tits85} describes $44$ such lattices for $q = 8$. Timmesfeld \cite{timmesfeld89} classified chamber-transitive lattices and in particular showed that these are the only possibilities for type $\tilde{A}_2$:

\begin{theorem}
\label{thm:chamber_trans}
Let $X$ be a locally finite building of type $\tilde{A}_2$ and order $q$. If $X$ admits a type-preserving chamber-transitive lattice $\Gamma$ then $q = 2$ or $q = 8$ and $\Gamma$ is chamber-regular.
\end{theorem}

Tits \cite{tits90b} has shown that the buildings that the lattices $\Gamma_{Ri}$ act on are distinct and can be distinguished by the Hjelmslev planes of level $2$ (we will look at Hjelmslev planes in more detail in Section~\ref{sec:hjelmslev}; for now we may just think of them as balls of combinatorial radius $2$). He assigns an invariant in $\Z/2\Z$ to the Hjelmslev planes of level $2$, that evaluates to $0$ on the plane of $\F_2\pseries{t}/t^2$ and to $1$ on the plane of $\Z_2/4$ (he also shows that this are the two only possible planes of level $2$ for $q = 2$). Evaluating the invariants on the three types of vertices he obtains a triple $(\Z/2\Z)^3$ (up to permutation).

Tits also defines an algebraic invariant in $\Z/2\Z$ for every generator as follows: first pick generators $\sigma_{ab}$, $\sigma_{bc}$ and $\sigma_{ca}$ for the normal subgroups of order $7$ in $\gen{a,b}$, $\gen{b,c}$, $\gen{a,c}$ (each vertex stabilizer is isomorphic to $C_7 \rtimes C_3$). Now the generator $a$ will carry invariant $0$ if it conjugates both, $\sigma_{ab}$ and $\sigma_{ca}$ to their square or both to their fourth power, and will carry invariant $1$ if it conjugates one to its square and the other to its fourth power. The other invariants for $b$ and $c$ are assigned analogously. He then shows that the invariant of the lattice coincides with the one of the building it acts on. Tits denotes the lattice with invariants $(c_0, c_1, c_2)$ by $G(c_0,c_1,c_2)$.

\begin{lemma}
The translation between Ronan's presentation and Tits's invariant is:
$\Gamma_{R1} = G(1,1,1)$, $\Gamma_{R2} = G(0,0,1)$, $\Gamma_{R3} = G(0,0,0)$, $\Gamma_{R4} = G(0,1,1)$.
\end{lemma}

\begin{proof}
For $\Gamma_{R1}$ we take elements of order $7$ to be $ab$, $bc$, $ab$. We find that $(ab)^a = ba = (ab)^2$ while $(ac)^a = ca = (ac)^4$ giving $a$ invariant $1$. The other generators are similar. We skip the other verifications except for $\Gamma_{R4}$. Here we take the elements of order $7$ to be $ab$, $bc^2$, and $ac$. We see that $a$ takes $ab$ as well as $ac$ to a square giving invariant $0$. On the other hand $b$ takes $bc^2$ to a square but $ab$ to its fourth power, giving invariant $1$. Similarly $c$ takes $bc^2$ to its square but takes $ac$ to its fourth power, giving invariant $1$.\qed
\end{proof}

Both Singer cyclic lattices for $q = 2$ embed into one of the chamber regular lattices. That $\Gamma_{2,1}$ embeds into $\Gamma_{R3}$ was already noted by Essert \cite[Remark, p. 1559]{essert13}.

Now we consider the lattice
\[
\Gamma_{2,2} = \gen{x,y,z \mid x^7=y^7=z^7, xyz^3=x^3y^3z = 1}\text{.}
\]
We will show that it has an extension that is isomorphic to one of the chamber regular lattices.
Note that $z = (xy)^2$, $z^3 = y^{-1}x^{-1}$, $z^4 = xy$, $z^6 = x^3y^3$. As a consequence, one sees $\Gamma_{2,2}$ admits the following automorphism of order three: $x \mapsto x^2$, $y \mapsto y^2$, $z \mapsto yx$. If we take the semidirect product with the group generated by this automorphism we get the lattice
\begin{equation}
\label{eq:22_ronan}
\hat{\Gamma}_{2,2} = \gen{x,y,z,a \mid \text{the above}, a^3, xa=ax^2, ya=ay^2, za=ayx}
\end{equation}
which is chamber regular because $a$ permutes the three $\Gamma_{2,2}$-orbits of chambers. Indeed, it is immediate from the relations that $\gen{a,x}$ and $\gen{a,y}$ generate Frobenius groups of order $21$ ($C_7 \rtimes C_3$) which are the building blocks of the chamber regular lattices. To recover Ronan's presentation, we need to find two further elements of order $3$ that normalize $x$ and $z$ respectively $y$ and $z$. That is, we are looking for conjugates of $a$ by $x$ and $y$ respectively that normalize $z$. We find
\[
z^{x^iax^{-i}} = z^{ax^{2i}x^{-i}} = (yx)^{x^i}\text{.}
\]
Taking $i = -1$ we get
\[
z^{x^{-1}ax} = xy = z^4\text{.}
\]
Similarly one computes that
\[
z^{yay^{-1}} = xy = z^4\text{.}
\]
So $b \defeq x^{-1}ax$ and $c \defeq ya^{-1}y^{-1}$ work. We note in passing that $ba = x$ and $ca = y^{-1}$.

\begin{lemma}
\label{lem:22_ronan_tits}
The lattice $\hat{\Gamma}_{2,2}$ has the presentation
\begin{equation}
\label{eq:ronan_4}
\hat{\Gamma}_{2,2} = \gen{a,b,c \mid a^3 = b^3 = c^3 = 1, (ab)^2 = ba, (ac)^2 = ca, c^2b = (bc^2)^2}
\end{equation}
and in particular is isomorphic to $\Gamma_{R4}$.
\end{lemma}

\begin{proof}
We compute
\[
ab = ax^{-1}a^{-1}x = x^{-4}x = x^4 = (ba)^4\text{.}
\]
Taking squares we obtain the first relation. The second relation is obtained completely analogously by evaluating $ac$. Finally
\[
bc^2 = x^{-1}a^{-1}xyay^{-1} = xy = z^4 \quad \text{and} \quad c^2b = yay^{-1}x^{-1}ax = y^4x^4 = z
\]
so $(bc^2)^2 = c^2b$.

Since we only needed to verify, which of the chamber regular lattices $\hat{\Gamma}_{2,2}$ is, we skip the verification that conversely the relations in \eqref{eq:22_ronan} follow from those in \eqref{eq:ronan_4}.\qed
\end{proof}

\begin{corollary}
The lattice $\Gamma_{2,2}$ acts on a building with non-isomorphic Hjelmslev planes and, in particular, not on a Bruhat-Tits building.
\end{corollary}

\section{Relation to vertex regular lattices}
\label{sec:vertex}

In this section we briefly discuss when Singer cyclic lattices give rise to vertex regular lattices. By a vertex regular lattice one could mean two things: a type-preserving lattice acting regularly on one type of vertex or a lattice acting regularly on all types of vertices. Both kinds will be discussed.

Let $\Gamma$ be a Singer cyclic lattice and let $(\sigma_0, \sigma_1, \sigma_2)$ be a presenting triple such that $\sigma_i$ fixes a vertex of type $i$ for each $i$. Suppose $\sigma_i$ maps non-trivially to the abelianization. Then (and only then) there is a homomorphism $\Gamma \to \Z/\delta\Z$ with $\sigma_i \mapsto 1$. Its kernel $\Lambda$ will obviously act freely on vertices of type $i$. Since $\Gamma$ acts transitively on these vertices with stabilizers of order $\delta$ and since $\Lambda$ has index $\delta$ in $\Gamma$ and acts freely, one sees that $\Lambda$ in fact is regular on vertices of type $i$.

\begin{observation}
\label{obs:normal_vertex_regular}
The following are equivalent for a Singer cyclic lattice $\Gamma$.
\begin{enumerate}
\item $\Gamma$ contains a normal subgroup $\Lambda$ acting regularly on vertices of type $i$,\label{item:v_regular_contains}
\item some (every) vertex stabilizer of type $i$ maps isomorphically to the abelianization,\label{item:v_abelianization}
\item $\Gamma = \Gamma_{v_i} \ltimes \Lambda$.
\end{enumerate}
\end{observation}

In the situation of Observation~\ref{obs:normal_vertex_regular} the lattice $\Lambda$ is the unique subgroup of $\Gamma$ that acts regularly on vertices of type $i$. Using Theorem~\ref{thm:aut} one can see that it also unique with this property acting on the building. We do not know if there are subgroups acting transitively on vertices of type $i$ that are not normal in $\Gamma$.

The relationship with lattices acting regularly (or in fact transitively) on all vertices is quickly explained. For the Bruhat--Tits examples it was discussed in Section~\ref{sec:brutit_examples}, for the others we have:

\begin{proposition}
Besides the Bruhat--Tits examples no Singer cyclic lattice with $q \le 5$ is quasi-isometric to a lattice acting transitively on all vertices.
\end{proposition}

\begin{proof}
A lattice acting transitively on all vertices in particular acts transitively on types. By Theorem~\ref{thm:aut} no exotic Singer cyclic lattice with $q \le 5$ acts on a building with type-transitive automorphism group. The statement now follows from Theorem~\ref{thm:qi_rigidity}.\qed
\end{proof}

\section{Linearity}
\label{sec:linearity}

Starting with Section~\ref{sec:brutit_examples} we have become interested in when a Singer cyclic lattice acts on the Bruhat--Tits building $X$ associated to $\PGL_3(K)$ where $K$ is a finite-dimensional division algebra over a local field. Since $\Aut(X) = (\PGL_3(K) \rtimes C_2) \rtimes \Aut(K)$, a natural first step is to wonder whether $\Gamma$ is a lattice in $\PGL_3(K)$. In this section we will discuss how this problem can be systematically approached by a kind of Hensel lifting, which can be (and has been) implemented on a computer.

The method has originally been used to find out which Singer cyclic lattices for $q \le 3$ embed into $\PGL_3(K)$. However, at the time of this writing, we can answer the question for lattices with $q \le 5$ using different methods: using the classification from Section~\ref{sec:singer_lattices} we will see in Section~\ref{sec:hjelmslev} using geometric methods that none of them can be Bruhat--Tits except for those constructed in Section~\ref{sec:brutit_examples}. Nonetheless the method was useful before we had obtained all of these results and we expect that it (or variants) will be useful in future investigations. The reader who is most interested in learning about concrete examples may want to directly jump to the examples at the end of the section (starting with paragraph \ref{sec:embedding_21}).

\subsection{Field automorphisms}

From now on we restrict to the case where $K$ is commutative. How closely related the problem of embedding $\Gamma$ in $\PGL_3(K)$ is to the problem of making $\Gamma$ act on the building of $\PGL_3(K)$ depends on the field $K$:

\begin{lemma}
\label{lem:aut_pos}
Let $K$ be a local field of positive characteristic $p$ whose residue field is $\F_q$. Then $K$ is isomorphic to $\F_q\lseries{t}$ with $q = p^\eta$. The automorphism group of $K$ is
\[
\Gal(\F_q/\F_p) \ltimes \F_q^\times \ltimes P \cong C_\eta \ltimes C_{q-1} \ltimes P
\]
where $P$ is a pro-$p$-group.
\end{lemma}

\begin{proof}
The inclusion and projection $\F_q \to \F_q\lseries{t} \to \F_q$ which compose to the identity show that $\Aut(\F_q\lseries{t})$ is a semidirect product of $\Aut(\F_q) = \Gal(\F_q/\F_p)$ and the group of automorphisms of $\F_q\lseries{t}$ that fix $\F_q$. This group is described in \cite[Theorem~2]{schilling44} to be pro-$p$-by-$\F_q^\times$. There is an automorphism for each power series of valuation $1$ (see \cite[Corollary~VII.1.1]{zarsam75}) and multiplication is evaluation. A splitting of the projection to $\F_q^\times$ is given by taking $u$ to the automorphism $t \mapsto u\cdot t$.\qed
\end{proof}

\begin{lemma}
\label{lem:aut_0}
Let $K$ be a local field of characteristic $0$ whose residue field is $\F_q$, $q = p^\eta$. Assume that $K/\Q_p$ is Galois with Galois group $G \defeq \Gal(K/\Q_p)$. There is a maximal unramified subextension $L/\Q_p$ with $I \defeq \Gal(K/L)$ (inertia group) and a maximal tamely ramified subextension $V/\Q_p$ with $R \defeq \Gal(K/V)$ (ramification group). The automorphism group of $K$ is $\Aut(K) = G$ and the Galois groups satisfy $R \lhd I \lhd G$ with $G/I \cong \Gal(\F_q/\F_p)$ and $I/R \cong \chi(K/\Q_p)$, the quotient of value groups. Here $\Gal(\F_q/\F_p)$ is cyclic of order $\eta$, the group $\chi(K/\Q_p)$ is cyclic, and $R$ is a $p$-group.
\end{lemma}

\begin{proof}
An automorphism of a local field leaves its ring of integers invariant, and an automorphism of the ring of integers leaves the maximal ideal invariant. Therefore every field automorphism is automatically continuous. Any automorphism $\alpha$ of $\Q_p$ leaves the prime field $\Q$ invariant and $\alpha|_\Q$ is trivial. Since $\Q$ is dense in $\Q_p$, we see that $\alpha$ is trivial. Similarly, any automorphism of $K$ leaves $\Q_p$, the closure of the prime field, invariant. This shows that $\Aut(K) = \Gal(K/\Q_p)$.

For the remaining statements see for example \cite[Section~II.9]{neukirch99}.\qed
\end{proof}

Lemmas~\ref{lem:aut_pos} and~\ref{lem:aut_0} impose strong restrictions on what the image of a Singer cyclic lattice in $\Aut(K)$ could be. In particular we see:

\begin{corollary}
\label{cor:linear_lattice}
Let $q = p^\eta$ and let $K$ be a local field whose residue field is $\F_q$ and assume that $\gcd(\eta,\delta) = 1$. Assume further that
\begin{enumerate}
\item $q \not\equiv 1 \mmod 3$ if $\chr K = p > 0$
\item $\gcd(\delta,\abs{\chi(K/\Q_p)}) = 1$ if $\chr K = 0$.
\end{enumerate}
If $\Gamma$ is a Singer cyclic lattice with parameter $q$ then any morphism $\Gamma \to \Aut(K)$ is trivial.
\end{corollary}

\begin{proof}
Let $\alpha \colon \Gamma \to \Aut(K)$ be a homomorphism.  Since $\Gamma$ is generated by elements of order $\delta = q^2+q+1$, its abelianization is a $\Z/\delta\Z$-module. The morphism to $\Gal(\F_q/\F_p) \cong C_\eta$ induced by $\alpha$ has to factor through the abelianization and since $\gcd(\eta,\delta) = 1$ it has to be trivial.

If $\chr K = p$ Lemma~\ref{lem:aut_pos} implies that there is an induced homomorphism to $\F_q^\times \cong C_{q-1}$. Since $q^2+q+1 - (q+2)(q-1) = 3$ the gcd of $q-1$ and $\delta$ can only be $1$ or $3$ and the latter is ruled out by assumption. Hence this morphism is trivial as well. Finally the morphism to a pro-$p$-group has to be trivial.

If $\chr K = 0$ Lemma~\ref{lem:aut_0} tell us that there is an induced morphism to $\chi(K/\Q_p)$ which has to be trivial by assumption. Again we are left with a morphism to a $p$-group which has to be trivial.\qed
\end{proof}

We do not know any Singer cyclic lattice on a Bruhat--Tits building that is not contained in $\PGL_3(K)$. Note however, that a Singer lattice with non-linear vertex stabilizer (cf.\ Lemma~\ref{lem:singer_group} and the discussion following it) cannot be contained in $\PGL_3(K)$.

\subsection{Projective groups and splittings}
Next we look at the circumstances under which $\PGL_3(K)$ may be replaced by $\SL_3(K)$ which can be more tractable in practice.

\begin{lemma}
\label{lem:mod_3}
Let $K$ be a local field whose residue field $\kappa$ has order $q$. Let $\calO$ be the ring of integers. The following are equivalent:
\begin{enumerate}
\item $x \mapsto x^3$, $\calO^\times \to \calO^\times$ is injective;\label{item:3rd_inj}
\item $K$ contains no non-trivial $3$rd roots of unity;\label{item:K_roots_unity}
\item $\kappa$ contains no non-trivial $3$rd roots of unity;\label{item:kappa_roots_unity}
\item $3 \nmid q - 1$;\label{item:3_div_q1}
\item $3 \nmid q^2 + q + 1$;\label{item:3_div_q2q1}
\end{enumerate}
\end{lemma}

\begin{proof}
A root of unity must have valuation $0$ so \eqref{item:3rd_inj} and \eqref{item:K_roots_unity} are equivalent. The equivalence of \eqref{item:K_roots_unity} and \eqref{item:kappa_roots_unity} is an application of Hensel's lemma. Since $\kappa^\times$ is cyclic of order $q-1$ we have \eqref{item:kappa_roots_unity} $\Leftrightarrow$ \eqref{item:3_div_q1}. The equivalence \eqref{item:3_div_q1} $\Leftrightarrow$ \eqref{item:3_div_q2q1} follows from the fact that $q^2+q+1 \equiv (q-1)^2 \mod 3$.\qed
\end{proof}

\begin{lemma}
\label{lem:vertex_stab}
Assume $q \not\equiv 1 \mmod 3$. If $\sigma \in \PGL_3(K)$ is of order $q^2+q+1$ then $\sigma$ fixes a vertex.
\end{lemma}

\begin{proof}
The Bruhat--Tits fixed point theorem \cite[Lemme 3.2.3]{brutit72} (see also \cite[Corollary~II.2.8]{brihae}) implies that $\sigma$ fixes a point of $X$, thus it stabilizes a cell $A$ (the carrier of the point). If the cell is an edge or a triangle then $\sigma$ acts on it via a morphism to $C_2$ or $D_3$ respectively. But $q^2+q+1$ is odd, so the action on an edge has to be trivial and on a triangle it can only be by rotation. If in addition $q^2+q+1$ is not divisible by $3$, the action on a triangle has to be trivial as well. Hence $\sigma$ fixes all vertices of $A$.\qed
\end{proof}

\begin{lemma}
\label{lem:q_congruence_factor}
Assume $q \not\equiv 1 \mmod 3$. If $\Gamma$ is a group generated by elements of order $q^2+q+1$ then any morphism $\Gamma \to \PGL_3(K)$ factors through $\SL_3(K)$.
\end{lemma}

\begin{proof}
Let $\sigma$ be an element of $\PGL_3(K)$ of order $q^2+q+1$. By Lemma~\ref{lem:vertex_stab} $\sigma$ stabilizes a vertex and by conjugating we may assume that it lies in $\PGL_3(\calO)$. We consider the following commuting diagram with exact rows and columns:
\begin{diagram}
 && 1 && 1 && 1 &\\
 && \dTo && \dTo && \dTo &\\
1 & \rTo & Z & \rTo & \SL_3(\calO) & \rTo & \PSL_3(\calO) & \rTo & 1\\
 && \dTo && \dTo && \dTo &\\
1 & \rTo & \calO^\times & \rTo & \GL_3(\calO) & \rTo & \PGL_3(\calO) & \rTo & 1\\
 && \dTo && \dTo && \dTo &\\
1 & \rTo & (\calO^\times)^3 & \rTo & \calO^\times & \rTo & \calO^\times/(\calO^\times)^3 & \rTo & 1\\
 && \dTo && \dTo && \dTo &\\
 && 1 && 1 && 1 &
\end{diagram}
The left column consists of homotheties and $Z$ is the group of $3$rd roots of unity in $\calO^\times$, which in our case is trivial by assumption.

The group $\calO^\times/(\calO^\times)^3$ in the lower right corner is cyclic of order $3$. So since $q^2+q+1$ is not divisible by $3$, the image of $\sigma$ in $\calO^\times/(\calO^\times)^3$ is trivial. This means that $\sigma \in \PSL_3(\calO)$.

If $\Gamma \to \PGL_3(K)$ is as assumed, we conclude that every generator is mapped into $\PSL_3(K) \cong \SL_3(K)$.\qed
\end{proof}

\subsection{Finding embeddings}
If $\Gamma$ is a Singer cyclic lattice and $(\sigma_i)_{0 \le i \le 2}$ is a presenting triple, a homomorphism $\iota \colon \Gamma \to G \defeq \PGL_3(K)$ is determined by the images $\iota(\sigma_i)$. Thus to find such a $\iota$ it suffices to find three $3 \times 3$ matrices over $K$ that satisfy the defining relations of $\Gamma$. To ensure that no additional relations are satisfied, we use:

\begin{lemma}
\label{lem:injectivity}
Let $\Gamma$ be a Singer cyclic lattice, let $(\sigma_i)_{0 \le i \le 2}$ be a presenting triple and let $\iota \colon \Gamma \to \PGL_3(K)$. Assume that in the building of $\PGL_3(K)$ there is a chamber with vertices $(w_i)_{0 \le i \le 2}$ such that $\iota(\sigma_i)$ fixes $w_i$ and acts as a Singer cycle on $\lk v_i$.\\
Then $\iota$ is injective. Moreover, it defines a unique $\Gamma$-equivariant isomorphism between the buildings of $\Gamma$ and $\PGL_3(K)$.
\end{lemma}

\begin{proof}
Let $X$ denote the building of $\Gamma$ and $Y$ the building of $\PGL_3(K)$ and let $v_i$ be the vertex fixed by $\sigma_i$. We claim that $\iota$ induces a $\Gamma$-equivariant simplicial map $\bar{\iota} \colon X \to Y$. On vertices $\bar{\iota}$ is defined by $\bar{\iota}(v_i) = w_i$. Since $\Gamma$ is transitive on each type of vertices, this determines $\bar{\iota}$ on the vertices of $X$. Similarly it is uniquely determined on edges by taking $\{v_i,v_j\}$ to $\{w_i,w_j\}$. To extend to triangles, note that the relations in $\Gamma$ (that also hold in $\iota(\Gamma)$) imply that $\{w_0,w_1,\iota(\sigma_0)^{e_0} w_2\}$ is a chamber when $e_0$ is an entry of the $0$th column of the difference matrix of $\Gamma$ with respect to $(\sigma_i)_i$. Thus taking $\{v_0,v_1,\sigma_0^{e_0}\}$ to $\{w_0,w_1,\iota(\sigma_0)^{e_0}\}$ is well-defined and completely determines $\bar{\iota}$.

Next we verify that $\bar{\iota}$ is surjective: it is clear from the definition that it is an isomorphism on the star of every vertex and that it is surjective on vertices.

Now let $N$ be the kernel of $\iota$. Note that the action of $N$ on $X$ is free: the only non-trivial elements in $\Gamma$ that fix points are conjugate to elements of some $\gen{\sigma_i}$ and are mapped non-trivially by $\iota$.

We claim that $\bar{\iota}$ is the quotient map $X \mapsto N \backslash X$. Clearly two points in the same $N$-orbit are identified under $\bar{\iota}$. Conversely if $\bar{\iota}(x) = \bar{\iota}(g.x) = \iota(g)\bar{\iota}(x)$ then $g \in N$.

This shows that $N \backslash X \cong Y$ so that $N = \pi_1(Y)$ is trivial.\qed
\end{proof}

\begin{remark}
In unpublished work, Yehuda Shalom and Tim Steger prove that any proper quotient of an $\tilde{A}_2$-lattice must be finite. Using that fact, the injectivity of $\iota$ in Lemma~\ref{lem:injectivity} could be verified by just checking that $\iota(\sigma_i)^e\iota(\sigma_j)^f$ has infinite order for appropriate exponents $e, f$ (that are not part of a row of the difference matrix).
\end{remark}

From now on we will make the embedding $\iota$ implicit and think of the $\sigma_i$ as elements of $\PGL_3(K)$ (also we identify $X$ with $Y$, $v_i$ with $w_i$, etc.). Let $\calO$ be the ring of integers in the local field $K$ and let $\pi$ be a uniformizing element, so that the residue field is $\kappa = \calO/(\pi)$ (of order $q$).

We take $P \defeq \PGL_3(\calO)$, which is the stabilizer of a vertex and we put
\[
\rho \defeq
\begin{pmatrix}
0&0&\pi^{-1}\\
1&0&0\\
0&1&0
\end{pmatrix}\text{.}
\]
Then for each $i$ the group $P_i \defeq P^{\rho^{-i}}$ stabilizes a vertex $w_i$ and $C \defeq \{w_0,w_1,w_2\}$ is a chamber. Since the automorphism group of the building of $\PGL_3(K)$ acts transitively on ordered vertices of this form, we may and do assume that $\sigma_i \in P_i$ for $0 \le i \le 2$ if they exist. Since in practice, it will be easier to work with elements of $P$ rather than $P_i$, we introduce
\[
s_i \defeq \sigma_i^{\rho^{i}} \in P, 0 \le i \le 2\text{.}
\]

With these preparations in place our problem of deciding whether the Singer cyclic lattice with difference matrix $E$ embeds into $\PGL_3(K)$ reduces to the problem of finding elements $s_i \in P$ that project to Singer-cycles in $\PGL_3(\kappa)$ and satisfy the relations
\begin{align}
s_i^\delta &= 1, 0 \le i \le 2\label{eq:order_new}\\
\pi s_0^{E_{j,0}}\rho s_1^{E_{j,1}}\rho s_2^{E_{j,2}}\rho & = 1, 1 \le j \le q+1 \text{.}\label{eq:product_new}
\end{align}

We will generally think about matrices in $\GL_3(K)$ but need to keep in mind that relations above have to be understood up to homotheties if $q \equiv 1 \mmod 3$ by Lemma~\ref{lem:q_congruence_factor}. The plan to decide whether such $s_i$ exist is to regard \eqref{eq:order_new} and \eqref{eq:product_new} as polynomial equations in the matrix entries of the $s_i$ and look for solutions in $\calO$. Since this is a local ring, we start by looking for solutions in $\kappa$ and then successively lift them to $\calO/(\pi^\ell)$ for increasing $\ell$. Geometrically this corresponds to looking for partially defined $\sigma_i$ that satisfy the relations whenever defined.  Of course an implementation of this method will only either return that no solution exists (which in practice happens very quickly), or it will return an approximate solution up to a finite precision. However, in practice it was always possible to guess an exact solution from the approximate one.

Without further preparations this approach is hopeless for the following reason. Let $B \defeq P_0 \cap P_1 \cap P_2$ be the pointwise stabilizer of our base chamber $C$. If $\sigma_0, \sigma_1, \sigma_2$ correspond to a solution of the above equations then so do the conjugates $\sigma_0^b, \sigma_1^b, \sigma_2^b$ for any $b \in B$. Thus each solution mod $\pi^\ell$ will lift to many solutions mod $\pi^{\ell+1}$, making it impossible to keep track of all potential solutions. For this reason, the main part of this section will be devoted to finding additional conditions to impose on the $s_i$ in order to make solutions to \eqref{eq:order_new} and \eqref{eq:product_new} unique.

\begin{figure}[!htb]
\centering
\includegraphics{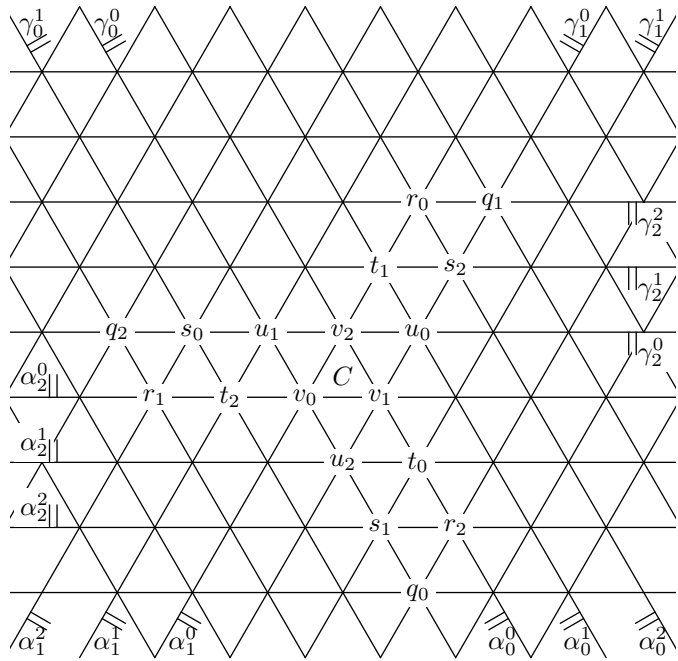}
\caption{Some named points and roots inside the apartment $\Sigma$.}
\label{fig:roots}
\end{figure}

To do so we will use root groups. Let $\Sigma$ be the apartment stabilized by the torus of diagonal matrices and let $T \cong (\calO^\times)^3$ be the torus of diagonal matrices over $\calO$. Note that this is the pointwise stabilizer of $\Sigma$. The root groups corresponding to the roots $\alpha_i^n, \gamma_i^n$, $0 \le i \le 2$, $n \in \N$ of $\Sigma$ indicated in Figure~\ref{fig:roots} are
\begin{align*}
U_{\alpha_0^n} &=
\Bigg\{
\begin{pmatrix}
1 &  & \\
& 1 &\\
x & & 1
\end{pmatrix}
\Bigg\vert
\nu(x) \ge n + 1
\Bigg\}
&
U_{\gamma_0^n} &=
\Bigg\{
\begin{pmatrix}
1 &  & x\\
& 1 & \\
& & 1
\end{pmatrix}
\Bigg\vert
\nu(x) \ge n
\Bigg\}\\
U_{\alpha_1^n} &=
\Bigg\{
\begin{pmatrix}
1 & x & \\
& 1 & \\
& & 1
\end{pmatrix}
\Bigg\vert
\nu(x) \ge n
\Bigg\}
&
U_{\gamma_1^n} &=
\Bigg\{
\begin{pmatrix}
1 &  & \\
x & 1 &\\
& & 1
\end{pmatrix}
\Bigg\vert
\nu(x) \ge n + 1
\Bigg\}\\
U_{\alpha_2^n} &=
\Bigg\{
\begin{pmatrix}
1 &  & \\
& 1 &x \\
& & 1
\end{pmatrix}
\Bigg\vert
\nu(x) \ge n
\Bigg\}
&
U_{\gamma_2^n} &=
\Bigg\{
\begin{pmatrix}
1 &  & \\
& 1 & \\
& x & 1
\end{pmatrix}
\Bigg\vert
\nu(x) \ge n + 1
\Bigg\}\text{.}
\end{align*}
By the ball of radius $r$ around $C$ in $\Sigma$, denoted $B_r(C)$, we mean the intersection of the $\beta_i^r, \beta \in \{\alpha,\gamma\}, 0 \le i \le 2$. Thus $B_0(C) = C$ and $B_1(C)$ consists of all chambers that meet $C$ in at least a vertex.

We will denote by $\bar{U}_{\beta_i^n}$ the quotient $U_{\beta_i^n}/U_{\beta_i^{n+1}}$ for $\beta \in \{\alpha, \gamma\}$, $0 \le i \le 2$, $n \ge 0$. It is isomorphic to (the additive group of) $\kappa$.

\begin{proposition}
\label{prop:little_projective}
An element $g \in B$ can be written as
\[
g = \big(\lim u_{\alpha_0^0}u_{\alpha_1^0}u_{\alpha_2^0}u_{\gamma_0^0}u_{\gamma_1^0}u_{\gamma_2^0}u_{\alpha_0^1}u_{\alpha_1^1}u_{\alpha_2^1}\cdots \big) \cdot t
\]
with $u_{\beta_i^n} \in U_{\beta_i^n}$ and $t \in T$ and this writing is unique in the following sense. If the first $k$ elements of the right hand side coincide then the $(k+1)$st defines a unique element in $\bar{U}_{\beta_i^n}$.
\end{proposition}

Thus we can get unique expressions by fixing a set-theoretic lift $\kappa \to \calO$.

\begin{proof}
Let $g \in B^+$ be arbitrary. We successively multiply $g$ by root group elements to make it coincide with the identity on larger and larger neighborhoods of $C$ in $\Sigma$. In the process we refer to the vertices named in Figure~\ref{fig:roots}. Since $g$ is in $B^+$ it fixes $C$. Multiplying by an element of $U_{\alpha_0^0}$ we get an element $g'$ that in addition takes $u_0$ to itself. Similarly we multiply by elements of $U_{\alpha_1^0}$ and of $U_{\alpha_2^0}$ to get an element $g''$ that takes $u_1$ and $u_2$ to themselves. Note that the elements by which we multiplied defined unique elements in the respective quotients $\bar{U}_{\alpha_i^0}$.

In a similar fashion we can use $U_{\gamma_0^0}$, $U_{\gamma_1^0}$, and $U_{\gamma_2^0}$ to get an element $g_1$ that fixes $t_2$, $t_1$, and $t_0$. Note that $g_1$ fixes $B_1(C)$. Again we multiplied by elements that were unique in $\bar{U}_{\gamma_i^0}$.

Continuing in this way with $U_{\alpha_i^n}, U_{\gamma_i^n}$, $0 \le i \le 2$ for increasing $n$, we get elements $g_n$ that fix $B_n(C)$.

The limit exists because $B$ is compact and because $U_{\beta_i^r}$ is contained in a small identity neighborhood for large $r$. Since the limit fixes all of $\Sigma$ we conclude that $t^{-1} \defeq \lim g_n I \in T$.\qed
\end{proof}

With these facts established we return to our problem of embedding $\Gamma$. We need to impose a condition on the difference matrix $E$. Namely we call $E$ \emph{normalized} if it is based and each column of $E$ contains the entry $1$. Not every difference matrix is equivalent to a normalized one: there are difference sets that are not equivalent to a difference set containing $0$ and $1$, however the first time this happens is for $q = 101$ (which is far beyond the computational scope of this section).

\begin{proposition}
Let $\Gamma$ be a Singer cyclic lattice and let $(\sigma_i)_i$ be a presenting triple such that the associated difference matrix is normalized. If $\Gamma \le \PGL_3(K)$ with $\sigma_i$ fixing $v_i$ then
\begin{equation}
s_i^{e_i^j} \cong
\begin{pmatrix}
* & * & *\\
* & * & *\\
0 & * & *
\end{pmatrix}
\mathrel{\mathrm{mod}} \pi\text{ for all }j\text{.}
\label{eq:local_constraint}
\end{equation}
Up to conjugating by an element of $B$ we can assume that the $s_i$ are of the form
\begin{equation}
\label{eq:global_constraint}
s_i = 
\begin{pmatrix}
0 & 0 & *\\
* & * & *\\
* & 1 & *
\end{pmatrix}
\end{equation}
where the stars indicate elements of $\calO$.
\end{proposition}

\begin{proof}
We start with some preliminary observations. Consider a relation
\[
\sigma_0^{e_0}\sigma_1^{e_1}\sigma_2^{e_2} = 1
\]
(we suppress the row index for readability). A consequence is that $\sigma_0^{-e_0}.v_2 = \sigma_1^{e_1}\sigma_2^{e_2}.v_2 = \sigma_1^{e_1}.v_2$. By cyclically permuting the relation and acting on the vertices $v_0$ and $v_1$ we find more generally (indices modulo $3$) that
\[
\sigma_i^{e_i}.v_{i+1} = \sigma_{i-1}^{-e_{i-1}}.v_{i+1}\text{.}
\]
Since $v_{i+1}$ is adjacent to $v_{i-1}$ it follows that the same is true of $\sigma_i^{e_i}.v_{i+1}$:
\begin{equation}
\sigma_i^{e_i}.v_{i+1} \sim v_{i-1} \quad \text{which means} \quad s_i^{e_i}.v_1 \sim v_2\text{.}\label{eq:incidence}
\end{equation}
In fact, $v_{i+1}$ and the vertices $\sigma_i^{e_i^j}.v_{i+1}$ are all the neighbors of the edge $v_{i-1}$ in the link of $v_i$, recovering the defining difference set.

The natural identification of the link of $v_0$ with $\P^2\kappa$ is such that $v_1$ gets identified with $[e_1]$ and $v_2$ gets identified with $[e_1] + [e_2]$. Thus the relation \eqref{eq:incidence} means that $s_i^{e_i}.e_1 \in \gen{e_1, e_2}$ which is precisely \eqref{eq:local_constraint}.

The rest of the proof is similar to that of Proposition~\ref{prop:little_projective}.
Since the difference matrix is normalized, for every $i$ there is some $j$ such that $E_{i,j} = 1$. Thus relation \eqref{eq:incidence} states that $\sigma_i.v_{i+1} \sim v_{i-1}$ for all $i$. We may now conjugate (all of $\Gamma$) by an element of $U_{\alpha_1^0}$ to achieve that $\sigma_0.v_1 = u_1$. Similarly conjugating by elements of $U_{\alpha_2^0}$ and $U_{\alpha_0^0}$ we get $\sigma_1.v_2 = u_2$ and $\sigma_2.v_0 = u_0$. Note that all of these can be expressed as $s_i.v_1 = u_1$ for $0 \le i \le 2$.

Conjugating by further root groups, we can achieve that $s_i.v_2 = t_2$, that $s_i.u_0 = s_0$, that $s_i.t_1 = r_1$, etc. Note that each of the root groups preserves the progress made so far by the root group property. Taking the limit, we find that the geodesic ray from $v_0$ to $[e_1] + [e_2] \in \P^2 K \cong \partial X$ is taken to the geodesic ray from $v_0$ to $[e_2] + [e_3] \in \P^2 K$. This explains the two $0$s in \eqref{eq:global_constraint}.

Now consider the ray $\omega$ from $v_0$ to $[e_2] \in \P^2 K$. Note that $s_i.u_1$ is adjacent to $t_2$ but certainly does not lie in $\Sigma$ because $s_i$ is a Singer cycle. Thus $s_i.\omega$ immediately leaves $\Sigma$. As above we can use elements of $U_{\alpha_2^0}$ and $U_{\gamma_0^0}$ to take $s_i.\omega$ to the ray from $v_0$ to $[e_3]$. Conjugating by these elements would destroy the progress made so far but it shows that $s_i.[e_2]$ lies in $U_{\alpha_2^0}U_{\gamma_0^0}.[e_3]$. In fact, we already know that $s_i.[e_2] \in [e_2] + [e_3]$ so it lies in $U_{\alpha_2^0}.[e_3]$. This means that $s_i[e_2]$ is of the form $[y_i e_2 + z_i e_3]$ with $y_i \in \calO$ and $z_i \in \calO^*$. Conjugating everything by the diagonal matrix with entries $(z_1^{-1}, z_2^{-1}, z_0^{-1})$ (which lies in $T$ and thus stabilizes $\Sigma$ and preserves everything we have done so far) we get that $s_i$ takes $[e_2]$ to a point of the form $[y_i' e_2 + e_3]$. This justifies the $1$ in \eqref{eq:global_constraint}.\qed
\end{proof}

We summarize the content of this section:

\begin{theorem}
Let $E$ be a normalized difference matrix and let $\Gamma$ be the Singer cyclic lattice defined by $E$.
There is an embedding of $\Gamma$ into $G$ if and only if there is a solution to the system or relations $X^i_{r,s}$
\begin{align*}
s_i^\delta \approx I, 0 \le i \le 2\\
\pi s_0^{E_{0,j}} \rho s_1^{E_{1,j}}\rho s_2^{E_{2,j}} \approx I, 1 \le j \le q
\end{align*}
with
\[
s_i =
\begin{pmatrix}
0 & 0 & X^i_{0,2}\\
X^i_{1,0} & X^i_{1,1} & X^i_{1,2}\\
\pi X^i_{2,0} & 1 & X^i_{2,2}\text{.}\\
\end{pmatrix}
\]
If $p \equiv 1 \mmod p$ then ``$\approx$'' means ``up to $\calO^*$'', otherwise it means ``$=$''.
\end{theorem}

We end the section by giving examples of embeddings found using the method described above.


\subsection{Embedding $\Gamma_{2,1}$}
\label{sec:embedding_21}

Let $A$ be the set of all natural numbers whose binary expansion does not contain the string $00$ (including $0$). The sequence of elements of $A$ is A003754 in \cite{oeis}. Let $B$ be the set of odd elements in $A$. The corresponding sequence in \cite{oeis} is A247648. We will need the following:

\begin{lemma}
\label{lem:sequence_bijections}
\begin{enumerate}
\item $B = \{2i+1 \mid i \in A\}$ \label{item:bijection_alpha_beta}
\item $A = B \cup 2 B \cup \{0\}$ (disjoint union). \label{item:bijection_alpha_beta_complement}
\end{enumerate}
\end{lemma}

\begin{proof}
In terms of the binary expansion taking $i$ to $2i+1$ means just extending by $1$ on the right. So if $i$ is in $A$ then $2i+1$ is in $\Gamma$ and every element of $\Gamma$ (uniquely) arises in this way.

Similarly, taking $i$ to $2i$ means extending by $0$ on the right. So if $i$ is in $B$ then $2i$ is in $B$ and is even. Every element of $A$ but $0$ arises in this way.\qed
\end{proof}

Now we consider the elements $\alpha = \sum_{e \in A} t^e$ and $\beta = \sum_{e \in B} t^e$ of $\F_2\lls{}t\rrs$. Thus
\begin{align*}
\alpha &= 1 + t + t^2 + t^3 + t^5 + t^6 + t^7 + t^{10} + t^{11} + t^{13} + t^{14} + t^{15} + t^{21} + \ldots\\
\beta &= t + t^3 + t^5 + t^7 + t^{11} + t^{13} + t^{15} + t^{21} + \ldots\text{.}
\end{align*}
We find

\begin{lemma}
\label{lem:relations}
The elements $\alpha$ and $\beta$ satisfy the relations $t\alpha^2 = \beta$ and $\alpha + \beta + \beta^2 = 1$.
\end{lemma}

\begin{proof}
First note that
\[
\alpha^2 = \sum_{(e,f) \in A \times A} t^{e+f} = \sum_{e \in A} t^{2e}
\]
since we are working in characteristic $2$ so that the non-diagonal terms cancel. Thus
\[
t\alpha^2 = \sum_{e \in A} t^{2e + 1} = \sum_{e' \in B} t^{e'} = \beta
\]
by Lemma~\ref{lem:sequence_bijections}\eqref{item:bijection_alpha_beta}.

For the second relation we compute
\[
\alpha = \sum_{e \in A} t^e = \sum_{e \in \Gamma} t^e + \sum_{e \in 2\Gamma} t^e + t^0 = \beta + \beta^2 + 1
\]
using Lemma~\ref{lem:sequence_bijections}\eqref{item:bijection_alpha_beta_complement}. This is the desired relations thanks to characteristic $2$.\qed
\end{proof}

We consider the matrices ($\rho$ is just the special case $K = \F_2\lseries{t}$ of the $\rho$ defined before)
\[
s =
\begin{pmatrix}
0 & 0 & \alpha \\
1 + t \alpha & \beta & \alpha\\
t\alpha & 1 & \beta
\end{pmatrix}
\in \SL_3(\F_2\lseries{t})
\quad
\text{and}
\quad
\rho =
\begin{pmatrix}
0 & 0 & t^{-1}\\
1 & 0 & 0\\
0 & 1 & 0
\end{pmatrix}\text{.}
\]

\begin{proposition}
The matrices $s$ and $\rho$ satisfy the relations
\[
t(s\rho)^3 = 1 \quad \text{and} \quad t(s^3\rho)^3 =1  \quad \text{and} \quad s^7 = 1\text{.}
\]
As a consequence $\Gamma_{2,1}$ embeds into $\PGL_3(\F_2\lseries{t})$ via $\sigma_i \mapsto s^{\rho^{-i}}$.
\end{proposition}

\begin{proof}
This is just a computation using Lemma~\ref{lem:relations}.\qed
\end{proof}

\subsection{Embedding $\Gamma_{3,1}$}
\label{sec:embedding_31}

To embed $\Gamma_{3,1}$ we use the alternative difference matrix
\[
\begin{pmatrix}
0 & 0 & 0\\
1 & 3 & 9\\
3 & 9 & 1\\
9 & 1 & 3
\end{pmatrix}
\]
which has the advantage of having smaller row sums.

The polynomial $A(X) = tX^2 + X + 1$ has a unique root in $\F_3\pseries{t}$:
\[
\alpha = -1 - t + t^2 + t^3 + t^4 + t^8 + t^9 + t^{10} - t^{11} - t^{12} - t^{13} + O(t^{20})
\]
From the defining equation it is clear that its inverse is $\alpha^{-1} = -(\alpha + 1) \cdot t$. Consider the matrices
\[
s =
\begin{pmatrix}
0 & 0 & \alpha\\
\alpha^{-1} & 1 & 0\\
0 & 1 & -1
\end{pmatrix}
\quad
\text{and}
\quad
\rho =
\begin{pmatrix}
0 & 0 & t^{-1}\\
1 & 0 & 0\\
0 & 1 & 0
\end{pmatrix}\text{.}
\]
To get an embedding of $\Gamma_{3,1}$ as presented by our standard difference matrix we put $s_0 = s$, $s_1 = s^9$, $s_2 = s^3$.

\begin{proposition}
The lattice $\Gamma_{3,1}$ embeds into $\SL_3(\F_3\lseries{t})$ via $\sigma_i \mapsto s_i^{\rho^{-i}}$.
\end{proposition}

\begin{proof}
This can be checked using the defining relation for $\alpha$.\qed
\end{proof}

\subsection{An alternative embedding of $\Gamma_{3,1}$.}

We will now embed the Singer cyclic lattice with difference matrix
\[
E = 
\begin{pmatrix}
0 & 0 & 0\\
1 & 1 & 1\\
4 & 4 & 4\\
6 & 6 & 6
\end{pmatrix}\text{.}
\]
By Corollary~\ref{cor:count_classes_isom_equiv} this is isomorphic to $\Gamma_{3,1}$, so we already have an embedding from the previous example. Originally we performed this computation when we were not yet aware of Corollary~\ref{cor:count_classes_isom_equiv} in full generality. Now it illustrates that how nice the embedding is depends strongly on the chosen difference matrix.

The polynomials
\begin{align*}
A(X) & = t^{17} X^{11} + (t^{15} + t^{16}) X^{10} + t^{15} X^{9} + t^{12} X^{8} + (t^{10} + t^{11}) X^{7} + t^{10} X^{6} + (t^{7} + t^{9}) X^{5}\\
&+ (2t^{5} + t^{6} + t^{7} + t^{8}) X^{4} + (t^{4} + t^{5} + t^{7}) X^{3} + (t^{2} + 2t^{4} + t^{6} + 2t^{8}) X^{2}\\
& + (2 + t + t^{2} + 2t^{3} + t^{4} + t^{5} + 2t^{6} + 2t^{7}) X + 1 + 2t + 2t^{2} + t^{4} + 2t^{6}\\
B(X) & = t^{16} X^{11} + t^{15} X^{10} + t^{14} X^{9} + t^{12} X^{8} + t^{11} X^{7} + (t^{9} + t^{10}) X^{6} + t^{9} X^{5} + (t^{6} + t^{8}) X^{4}\\
&+ (2t^{5} + t^{6} + t^{7}) X^{3} + (t + t^{2} + 2t^{4} + t^{6} + 2t^{7}) X^{2} + (1 + t + t^{3} + t^{5} + 2t^{6}) X\\
& + 1 + t^{3} + t^{4} + 2t^{5}\\
&  + t^{5} X^{3} + t^{5} X + t^{4} X^{3} - t^{4} X^{2} + t^{4} X + t^{4} - t^{3} X + t^{2} X^{2} + t^{2} X - t^{2} + t X - t - X + 1\\
C(X) & = t^{13} X^{11} + t^{12} X^{10} + t^{11} X^{9} + 2t^{10} X^{8} + 2t^{9} X^{7} + (t^{7} + 2t^{8}) X^{6} + (t^{6} + 2t^{7}) X^{5} + 2t^{6} X^{4}\\
&+ (t^{3} + 2t^{4} + 2t^{5}) X^{3} + (2t + 2t^{2} + 2t^{3} + 2t^{4}) X^{2} + (2 + t + 2t^{3}) X + 2 + t + 2t^{2}
\end{align*}
each have a unique root in $\F_3\pseries{t}$. They are
\begin{align*}
\alpha &= 1 + t^{2} + 2t^{4} + 2t^{5} + 2t^{6} + 2t^{7} + t^{10} + t^{11} + 2t^{12} + 2t^{14} + 2t^{15} + 2t^{16} + 2t^{17} + O(t^{20})\\
\beta &= 2 + 2t + t^{2} + 2t^{3} + 2t^{4} + 2t^{5} + 2t^{7} + t^{10} + t^{11} + 2t^{12} + 2t^{13} + 2t^{14} + 2t^{16} + 2t^{17} + t^{18} + O(t^{20})\\
\gamma &= 2 + 2t^{2} + 2t^{3} + 2t^{5} + t^{8} + 2t^{9} + 2t^{11} + 2t^{12} + 2t^{14} + 2t^{17} + 2t^{18} + O(t^{20})\\
\end{align*}
With these series we set
\[
s =
\left(\begin{array}{rrr}
0 & 0 & t \alpha -1 \\
-t^{2} \alpha -1 + t & -t \beta -1 & t \gamma -1 \\
-t^{2} \gamma + t & 1 & t \beta
\end{array}\right).
\]

\begin{proposition}
The map $\sigma_i \mapsto s^{\rho^{-i}}$ defines an embedding of the Singer cyclic lattice with difference matrix $E$ into $\SL_3(\F_3\lseries{t})$.
\end{proposition}

\begin{proof}
If we replace $\alpha, \beta, \gamma$ by $X, Y, Z$ in the definition of $s$ above, the ideal of $\F_3(t)[X, Y, Z]$ generated by the equations $s^{13} = I$, $(s\rho)^3 = t^2 I$, $(s^4\rho)^3 = t^2I$ and $(s^6\rho)^3 = t^2I$ is generated by the three polynomials
\[
t^{6} X^{4} + (2 t^{5} + t^{4}) X^{3} + t^{3} Y Z^{2} + t^{3} X^{2} + t^{2} X Z + (2 t^{2}) Y Z + (2 t^{3} + t^{2} + 2 t) X + Y + (t + 1) Z + t^{2} + 2
\]
\[
t^{4} X^{2} Z + t^{3} Y Z^{2} + (2 t^{3}) X^{2} + (t^{3} + t^{2}) X Z + t^{2} Y Z + (2 t^{2}) Z^{2} + (2 t^{2} + 2 t) X + (t + 2) Y + (t^{2} + t + 1) Z + 2 t
\]
\[
t^{3} Z^{3} + (2 t^{2}) X^{2} + t^{2} X Z + (2 t) Y Z + (t + 2) X + Y + (2 t) Z + 2
\]
(this can be checked on a computer using Gröbner bases) and it contains the three polynomials $A(X)$, $B(Y)$ and $C(Z)$.

An application of Hensel's lemma \cite[Corollaire~III.4.2]{bourbaki_alg_comm_3-4fr} shows that the three generators of the ideal have a unique root in $\F_3\lseries{t}[X,Y,Z]$ and so do $A(X)$, $B(Y)$, $C(Z)$. We conclude that the tuple $(\alpha, \beta, \gamma)$ must be the unique root of the generators and therefore a solution to the equations.\qed
\end{proof}


\section{Hjelmslev planes}
\label{sec:hjelmslev}

From now on we will be interested in distinguishing the buildings that Singer cyclic lattices act on. Note that in view of Theorem~\ref{thm:qi_rigidity} this building is unique for a given lattice, and classifying the buildings amounts to classifying the lattices up to quasi-isometry. To do so we will compare combinatorial balls of a certain radius around a vertex. Such balls can be encoded by incidence geometries called \emph{Hjelmslev planes} \cite[Section~7.2]{dembowski68} (see also~\cite{klingenberg55,artmann69,bacon78}).

\subsection{Hjelmslev planes}
An incidence geometry $\hjelm = (\hpts, \hlns, \inc)$ is a \emph{Hjelmslev plane} if any two points lie on at least one common line and dually, and there is an epimorphism $\varphi \colon \hjelm \to \ProjPla$ to a projective plane $\ProjPla = (P,L,I)$ such that points $x$ and $x'$ lie on more than one common line if and only if $\varphi(x) = \varphi(x')$ and dually. Two elements with same image under $\varphi$ are called \emph{neighbors}. A \emph{Klingenberg plane} is defined just like a Hjelmslev plane with the difference that the last condition is that $x$ and $x'$ lie on more than one common line \emph{only if} $\varphi(x) = \varphi(y)$ and dually. The projection $\varphi$ gives rise to an equivalence relation $\sim$ on points and lines defined by $x \sim x'$ if and only if $\varphi(x) = \varphi(x')$. For Hjelmslev planes this can be recovered from the incidence relation, for Klingenberg planes it is part of the structure.

\begin{example}
Let $R$ be a local ring with maximal ideal $\frakm$. Call a vector $(x_1,x_2,x_3) \in R$ \emph{unimodular} if some $x_i$ is a unit (i.e.\ not in $\frakm$). We define a sets of points as $\hpts \defeq \{R^\times (x_1,x_2,x_3) \mid (x_1,x_2,x_3) \text{ unimodular}\}$ and a set of lines as $\hlns \defeq \{(y_1,y_2,y_3) R^\times \mid (y_1,y_2,y_3) \text{ unimodular}\}$. Incidence is defined by $R^\times (x_1,x_2,x_3) \inc (y_1,y_2,y_3) R^\times$ if and only if $x_1y_1 + x_2y_2 + x_3y_3 = 0$ (which is clearly well-defined). The incidence structure $\hjelm = (\hpts, \hlns, \inc)$ is a Klingenberg plane and any plane that arises in this way is called \emph{Desarguesian}.

If in addition every element of $\frakm$ is both a left- and a right zero divisor and for any two elements $x, y \in \frakm$ the left (and right) ideals generated by $x$ and $y$ are contained in each other then $R$ is an $H$-ring. The Klingenberg plane associated to an $H$-ring is a Hjelmslev plane.

In particular, if $\calO$ is a discrete valuation ring then for any $k \in \N \setminus \{0\}$ the quotient $\calO/\frakm^k$ is a commutative local principal ideal ring and thus an $H$-ring. We will be interested in Hjelmslev planes associated to rings of this form.
\label{exmpl:hjelmslev_moufang}
\end{example}

\subsection{Hjelmslev planes for Singer cyclic lattices}

Hjelmslev planes encode neighborhoods in $\tilde{A}_2$ buildings and the inverse system of Hjelmslev planes encodes the whole building, see \cite{vanmaldeghem87,vanmaldeghem88,hanvma89}. We consider the building on which a Singer cyclic lattice acts. For the purpose of exposition we fix the vertex $v_0$ as our center of interest. But since a Singer cyclic lattice has three vertex orbits, everything we do around $v_0$ can be done around $v_1$ and $v_2$ and give different answers. The analogous results are obtained by cyclically permuting the indices.

The elements of the Hjelmslev plane of level $k$ are vertices $v$ that are a geodesic edge path of length $k$ away from $v_0$ (which is unique, being a geodesic). We denote the Hjelmslev plane of level $k$ (including additional structure to be introduced) by $\hjelm^k$ or by $\hjelm^k(v_0)$ if we want to specify the center. Note that $\hjelm^1$ consists just of the vertices in $\lk v_0$. For $k \le \ell$ there is an obvious projection $\hjelm^\ell \to \hjelm^k$ that takes the end-vertex of an $\ell$-edge geodesic edge path to the end-vertex of its $k$-edge initial subpath. We say that an element of $\hjelm^k$ is a \emph{point} or a \emph{line} respectively if its projection image in $\hjelm^1$ is a point or a line, namely if it is of type $1$ respectively $2$. Thus if $v_0 = v^0, \ldots, v^k=v$ are the vertices along a geodesic edge path so that $v \in \hjelm^k$ then $v$ is a point if $\typ v^1 = 1$ and is a line if $\typ v^1 = 2$. We write $\hpts^k$ for the points and $\hlns^k$ for the lines in $\hjelm^k$. A point $x \in \hpts^k$ and a line $y \in \hlns^k$ are incident if there is a regular triangle with side lengths $k$ and vertices $v_0$, $x$, $y$. Note that (non-)incidence is preserved under the projection $\hjelm^\ell \to \hjelm^k$.

\begin{fact}
\label{fact:hjelmslev_building_ring}
Let $K$ be a local field and let $\calO$ be its ring of integers. The Hjelmslev planes of level $k$ inside the building of $\PGL_3(K)$ are isomorphic to the Hjelmslev planes of $\calO/\frakm^k$ as in Example~\ref{exmpl:hjelmslev_moufang}.
\end{fact}

A useful feature of Singer cyclic lattices is that Hjelmslev planes can be described in an extremely efficient way. Let $\Gamma$ be a Singer cyclic lattice with presenting triple $(\sigma_i)_{0 \le i \le 2}$ and based difference matrix $E$. We let $D_i$ denote the difference set consisting of entries of the $i$th column of $E$. Let $X$ be the associated building and let $v_i$ be the vertex fixed by $\sigma_i$. Any vertex of the building is of the form $\sigma_{0}^{e^0} \sigma_{i_1}^{e^1} \cdots \sigma_{i_{\ell-1}}^{e^{\ell-1}} v_{i_\ell}$ for some $\ell$ and some $i_j \in \{0,1,2\}, 1 \le j \le \ell$ and $e^j \in \Z/\delta\Z, 0 \le j \le \ell-1$. This notation is a little cumbersome and since we will have to deal with many expressions of this form, we drop the $\sigma_{i_j}$ and make its index an index of the exponent, for instance we would write the previous expression as $e^0_0 e^1_{i_1} \ldots e^{\ell-1}_{i_{\ell-1}} v_{i_{\ell}}$. See the left hand side of Figure~\ref{fig:hjelmslev-paths} for an example. The subscripts are always understood modulo $3$ and we always assume that $i_j \ne i_{j+1}$. We call such an expression an \emph{edge path}. This is justified by the fact that it uniquely describes the edge path that has vertices $v_0, \sigma_0^{e^0}v_{i_1}, \sigma_0^{e^0}\sigma^{e^1} v_{i_2}, \sigma_0^{e^0}\sigma^{e^1}\sigma^{e^2} v_{i_3}, \ldots, \sigma_0^{e^0} \cdots \sigma_{i_{\ell-1}}^{e^{\ell-1}} v_{i_{\ell}}$. Pictorially we describe this edge path by labeling the edge from $\ldots \sigma_{i_{k-1}}^{e^{k-1}} v_k$ to $\ldots \sigma_{i_{k}}^{e^{k}} v_{k+1}$ by $e^k$. In the pictorial description the type $i_k$ is again clear from the context (it is just the type of the initial vertex of the edge labeled $e^k$) and may be omitted.

\begin{observation}
\label{obs:adjacent}
The two vertices $\sigma_i^b.v_{i+1}$ and $\sigma_i^a.v_{i+2}$ are adjacent if and only if $b - a \in D_i$.
\end{observation}

\begin{proof}
We assume $i = 0$. The vertices $\sigma_0^b.v_{1}$ and $\sigma_0^a.v_{2}$ are adjacent if and only if $\sigma_0^{b-a}.v_{1}$ and $v_{2}$ are adjacent. This happens if and only if $\sigma_0^{b-a}\sigma_{1}^f = \sigma_{2}^{-g}$ for some $f$ and $g$, which is equivalent to $(b-a, f, g)$ being a row of $E$.\qed
\end{proof}

\begin{observation}
\label{obs:geodesic_edge_paths}
The edge path $e^0_0 e^1_{i_1} \ldots e^\ell_{i_{\ell-1}} v_{i_\ell}$ is geodesic if for $2 \le j \le \ell-2$ either
\begin{enumerate}
\item $(i_{j-1}, i_j, i_{j+1}) = (i_j -1, i_j, i_j + 1)$ and $-e^j \not\in D_{i_j}$ or \label{item:geodesic_forward}
\item $(i_{j-1}, i_j, i_{j+1}) = (i_j + 1, i_j, i_j -1)$ and $e^j \not\in D_{i_j}$.\label{item:geodesic_backward}
\end{enumerate}
\end{observation}

\begin{proof}
The first condition (of both items) just says that $i_{j-1}$, $i_j$, and $i_{j+1}$ need to be pairwise distinct, which is clearly necessary. We only prove the \eqref{item:geodesic_forward} case, the other one is analogous. We fix some $j$, put $g = \sigma_0^{e^0} \ldots \sigma_{i_{j-1}}^{e^{j-1}}$, and check when the edge path is locally geodesic in $g.v_j$. The previous vertex is $g. v_{i_j-1}$ and the following one is $g\sigma_{i_j}^{e^j}. v_{i_j+1}$. Thus we see that the path is locally geodesic in $g.v_j$ if and only if $v_{i_j-1}$ and $\sigma_{i_j}^{e^j}. v_{i_j+1}$ are opposite in $\lk v_j$ if and only if they are not adjacent.
Using the previous observation we see that 
the condition for being locally geodesic is that $-e^j \not\in D_{i_j}$.\qed
\end{proof}

We have just seen that the types along a geodesic edge path can only either cycle forward ($0,1,2,0,1,2,\ldots$) or cycle backward ($0,2,1,0,2,1,\ldots$) so we when specifying an edge path we may even drop the type-indices (including the terminal vertex) as long as we specify whether it is forward- or backward-cycling. In particular, elements of $\hpts^k$ are forward cycling edge paths and elements of $\hlns^k$ are backward cycling edge paths which can both be described by a sequence of elements of $\Z/\delta\Z$ of length $k$. It remains to analyze incidence.

\begin{lemma}
\label{lem:el_hty}
The two expressions $\sigma_i^a \sigma_{i+2}^y \sigma_i^r$ and $\sigma_i^b \sigma_{i+1}^x \sigma_i^s$ are equal if and only if there are rows $(e_0,e_1,e_2)$ and $(f_0, f_1, f_2)$ of $E$ such that the following equations hold:
\begin{align*}
b- a &= e_i\\
x &= e_{i+1} - f_{i+1}\\
y &= f_{i+2} - e_{i+2}\\
r - s & = f_i\text{.}
\end{align*}
\end{lemma}

\begin{proof}
We assume $i=0$. Suppose the equations hold. Then we verify
\begin{align*}
\sigma_0^b \sigma_1^x \sigma_0^s &=  \sigma_0^a \sigma^{e_0} \sigma_1^x \sigma_0^{-f_0} \sigma_0^r\\
&= \sigma_0^a \sigma_0^{e_0} \sigma_1^{e_1} \sigma_1^{-f_1} \sigma_0^{-f_0} \sigma_0^r\\
&= \sigma_0^a \sigma_2^{-e_2} \sigma_2^{f_2} \sigma_0^r\\
&= \sigma_0^a \sigma_2^{y} \sigma_0^r\text{.}\\
\end{align*}
Conversely, if the two expressions are equal then $\sigma_0^a \sigma_2^y v_0 = \sigma_0^b \sigma_2^x v_0$ thus $\sigma_0^a v_2$ and $\sigma_0^b v_1$ have to be adjacent showing the existence of a row $(e_0, e_1, e_2)$ with $b-a = e_0$. Similarly we obtain a row $(f_0, f_1, f_2)$ with $r - s = f_0$. Now we compute
\[
\sigma_2^{-e_2} \sigma_1^{-e_1} \sigma_1^x = \sigma_0^{e_0} \sigma_1^x = \sigma_0^{b-a} \sigma_1^x = \sigma_2^y \sigma_0^{r-s} = \sigma_2^y \sigma_0^{f_0} = \sigma_2^y \sigma_2^{-f_2} \sigma_1^{-f_1}
\]
showing that
\[
\sigma_1^{x-e_1+f_1} = \sigma_2^{y+e_2-f_2}
\]
fixes both $v_1$ and $v_2$ and therefore has to be trivial.\qed
\end{proof}

\begin{figure}
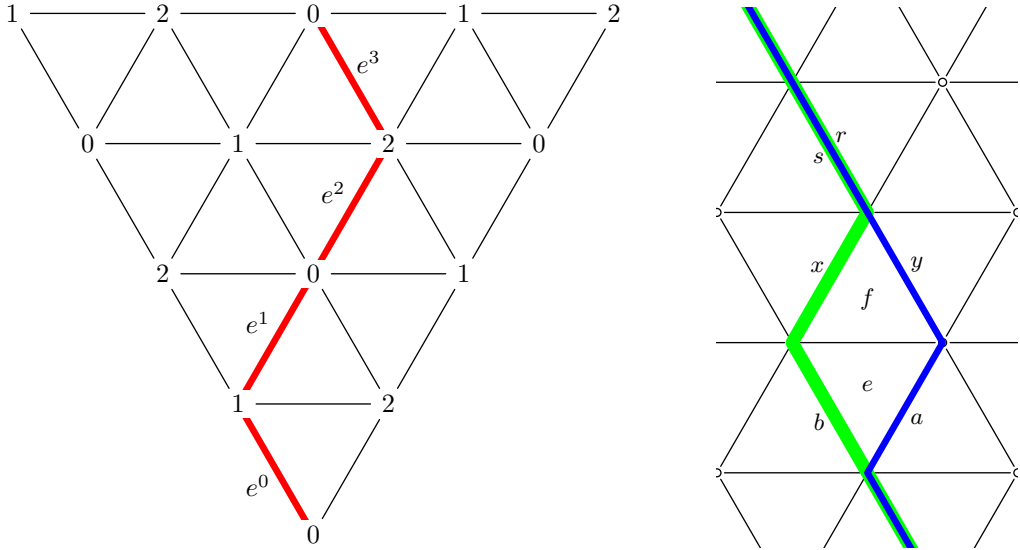

\centering
\includegraphics{hjelmslev-5}
\hspace{1cm}
\includegraphics{hjelmslev-3}
\caption{The left picture illustrates the graphical description of the path $e^0_0e^1_1e^2_0e^3_2v_0$. The right picture illustrates the effect of the elementary homotopy $b_ix_{i+1}s_i \leadsto a_i y_{i+2}r_i$.}
\label{fig:hjelmslev-paths}
\end{figure}

We call a substitution $\sigma_i^a \sigma_{i+2}^y \sigma_i^r \leadsto \sigma_i^b \sigma_{i+1}^x \sigma_i^s$ as in Lemma~\ref{lem:el_hty} (and its inverse) an \emph{elementary homotopy}, see the right hand side of Figure~\ref{fig:hjelmslev-paths}. This includes the possibility of substituting $\sigma_i^a \sigma_{i+2}^y v_i \leadsto \sigma_i^b \sigma_{i+1}^x v_i$. We call a substitution $\sigma_i^b v_{i+1} \leadsto \sigma_i^a v_{i+2}$ (and its inverse) a \emph{step move}.

We are now ready to give a criterion for incidence in Hjelmslev planes of arbitrary level. Let $b^0 \ldots b^{\ell-1}$ be a forward-cycling edge path (ending in an element of $\hpts^\ell$) and let $a^0 \ldots a^{\ell-1}$ be a backward-cycling edge path (ending in an element of $\hlns^\ell$). We consider a regular triangle of side length $\ell$ that is a conceivable filling triangle with base vertex $v_0$. We orient the edges along the two sides containing $v_0$ away from $v_0$ and label them by $b^0$ to $b^{\ell-1}$ and by $a^0$ to $a^{\ell-1}$ respectively. All other edges receive label $0$. The label of an edge $e$ will be denoted $\lambda(e)$. There is a caveat: the elements $b^i$ and $a^i$ do not actually correspond to the edges but rather to the turns. For practical reasons we still attach them to the edges but note that an elementary homotopy changes the label of the last edge without changing the edge.

Whether the point and the line are incident will depend on whether it is possible to label the chambers of the triangle by rows of $E$ subject to certain conditions. The label of a chamber $c$ will be denoted $\lambda(c)$. As a final preparation we need to introduce signs (see Figure~\ref{fig:hjelmslev-signs}). Each chamber $c$ carries a sign $\varepsilon_c$ such that the chamber containing $v_0$ has positive sign and adjacent chambers have opposite sign. Each oriented edge $e$ receives a sign $\varepsilon_e$ that is positive if $e$ is forward pointing (i.e.\ from type $i$ to $i+1$) and negative if it is backward pointing. For a chamber $c$, an oriented boundary edge $e$ and a vertex $v$ we write $e > v$ if $e$ contains $v$ and points away from it and $c > v$ if $c$ contains $v$.

\begin{figure}
\centering
\includegraphics{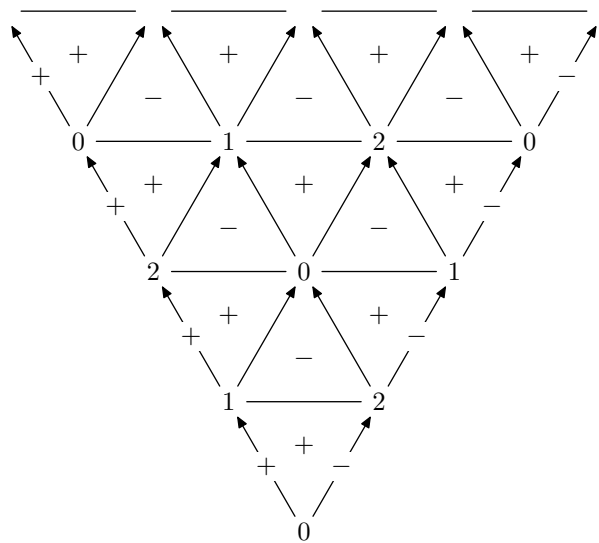}
\caption{Signs associated to edges and chambers. Each vertex is labeled by its type (this is relevant for the edge signs).}
\label{fig:hjelmslev-signs}
\end{figure}

\begin{figure}
\centering
\includegraphics{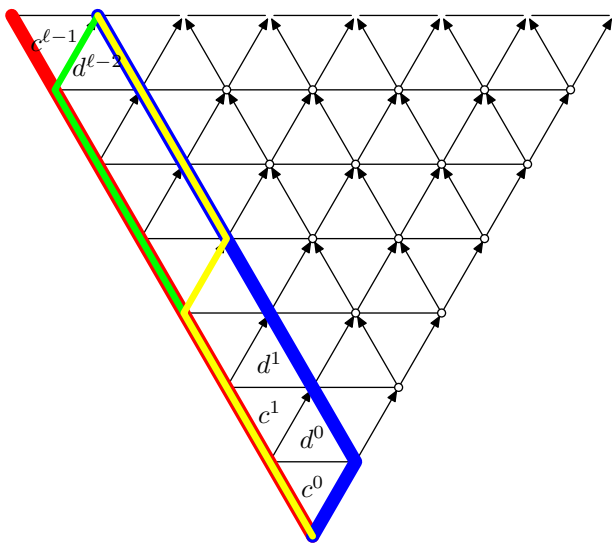}
\caption{Paths used in the proof of Theorem~\ref{thm:hjelmslev}.}
\label{fig:hjelmslev-signs}
\end{figure}

\begin{theorem}
\label{thm:hjelmslev}
The Hjelmslev plane of level $\ell$ is given by the following data. Its points are the end points of forward-cycling edge paths $b^0 \ldots b^{\ell-1}$ of length $\ell$. Its lines are the end points of backward-cycling edge paths $a^0 \ldots a^{\ell-1}$ of length $\ell$. The point given by $b^0\ldots b^{\ell-1}$ and the line given by $a^0 \ldots a^{\ell-1}$ are incident if there is an labeling of the chambers by rows of $E$ such that for every vertex $v$ at combinatorial distance $< \ell$ from $v_0$
\begin{equation}
\label{eq:edge_chamber_formula}
\sum_{c > v} \varepsilon_c \lambda(c)_{\typ(v)} = \sum_{e > v} \varepsilon_e \lambda(e)\text{.}
\end{equation}
\end{theorem}

\begin{proof}
The proof is by induction on $\ell$. Let $b^0 \ldots b^{\ell-1}$ and $a^0 \ldots a^{\ell-1}$ be forward- respectively backward-cycling paths as in the statement.

We assume that their endpoints are incident in the Hjelmslev plane and consider a triangle that establishes this incidence. The strategy is as follows, see the right-hand-side of Figure~\ref{fig:hjelmslev-signs}. Starting with the edge path $b^0 \ldots b^{\ell-1} v_\ell$ (red) we perform a step move to obtain a path $b^0_0 \ldots b^{\ell-2}_{\ell-2} x^{\ell-1}_{\ell-1} v_{\ell+1}$ (green) whose endpoint is closer to the endpoint of $a^0 \ldots a^{\ell-1}$ (recall that the subscripts are modulo $3$). We then apply elementary homotopies through paths $b^0_0 \ldots b^{k-1}_{k-1} x^{k-1}_{k-1} \dbar{a}^k_{k+1} \bar{a}^{k+1}_{k+2} \ldots \bar{a}^{\ell-1}_\ell v_{\ell+1}$ (yellow) until we end up with a path that starts with $a^0_0$. We then apply the induction hypothesis to the smaller triangle.

We use the chamber labeling indicated in Figure~\ref{fig:hjelmslev-signs}. By Observation~\ref{obs:adjacent} there is a row $e^{\ell-1}$ of $E$ such that $b^{l-1} - x^{l-1} = e^{\ell-1}_{l-1}$ and we put $\lambda(c^{\ell-1}) = e^{\ell-1}$. Now comes the first elementary homotopy from $\ldots b^{\ell-2}_{\ell-2}x^{\ell-1}_{\ell-1}v_{\ell+1}$ to $\ldots x^{\ell-2}_{\ell-2} y^{\ell-1}_{\ell} v_{\ell+1}$. By Lemma~\ref{lem:el_hty} there are rows $e^{\ell-2}$ and $f^{\ell-2}$ of $E$ such that $b^{\ell-2} - x^{\ell-2} = e^{\ell-2}_{\ell-2}$ and $x^{\ell-1} = f^{\ell-2}_{\ell-1} - e^{\ell-2}_{\ell-1}$ and $y^{\ell-2} = e^{\ell-2}_\ell - f^{\ell-2}_\ell$ (the last relation does not yet show up because the homotopy is so close to the vertex). We assign $\lambda(c^{\ell-2}) = e^{\ell-2}$ and $\lambda(d^{\ell-2}) = f^{\ell-2}$. From now on we perform a sequence of elementary homotopies that are all of the form
\[
\ldots b^{k-1}_{k-1} x^k_k y^{k+1}_{k+2} \ldots \quad \leadsto \quad \ldots x^{k-1}_{k-1} y^k_{k+1} \bar{b}^{k+1}_{k+2} \ldots\text{.}
\]
Using again Lemma~\ref{lem:el_hty} we find rows $e^{k-1}$ and $f^{k-1}$ such that
\begin{align}
\label{eq:e} b^{k-1} - x^{k-1} &= e^{k-1}_{k-1}\\
\label{eq:x} x^k &= e^{k-1}_k - f^{k-1}_k\\
\label{eq:y} y^k &= f^{k-1}_{k+1} - e^{k-1}_{k+1}\\
\label{eq:f} \bar{b}^{k+1} - y^{k+1} &= f^{k-1}_{k+2}
\end{align}
As before we put $\lambda(c^{k-1}) = e^{k-1}$ and $\lambda(d^{k-1}) = e^{k-1}$.

Eventually we get to the edge path
\[
x^0_0 y^1_2 \bar{b}^2_3 \ldots \bar{b}^{\ell-1}_{\ell} v_{\ell+2}
\]
which has $x^0 = a^0$. So we may regard the edge paths $y^1_2 \bar{b}^2 \ldots \bar{b}^{\ell-1}$ and $a^1 \ldots a^{\ell-1}$ as defining a point and a line in the Hjelmslev plane of the vertex $\sigma_0^{a^0}.v_2$ and a labeling of the remaining triangles from the induction hypothesis. The equation \eqref{eq:edge_chamber_formula} holds for all interior vertices of the smaller triangle by induction hypothesis, so we need to check it on the vertices on the edge path $b^0 \ldots b^{\ell-1}$ and we need to produce the correcting terms for vertices on the edge path $x^0_0 y^1_2 \bar{b}^2_3 \ldots \bar{b}^{\ell-1}_\ell$.

Let us first check the additional relations along the face. The relation \eqref{eq:edge_chamber_formula} for $v_0$ reads $b^0 - a^0 = \lambda(c^0)_0 = e^0_0$ which is just \eqref{eq:e} for $k = 1$. The other instances of \eqref{eq:edge_chamber_formula} along the face read $b^i = \lambda(c^i)_k - \lambda(d^{i-1})_k + \lambda(c^{i-1})_k = e^i_k - f^{i-1}_k + e^{i-1}_k$ which is obtained by summing together \eqref{eq:e} for $k = i+1$ and \eqref{eq:x} for $k = i$.

As for the correction terms, we first need to show that $- y^1 = -\lambda(d^0)_2 + \lambda(c^0)_2 = -f^0_2 + e^0_2$ which is \eqref{eq:y} for $k = 1$. For the remaining correction terms we need to show that $-\bar{b}^i = -\lambda(d^{i-1})_{i+1} + \lambda(c^{i-1})_{i+1} - \lambda(d^{i-2})_{i+1} = -f^{i-1}_{i+1} + e^{i-1}_{i+1} - d^{i-2}_{i+1}$, which is \eqref{eq:f} for $k = i-1$ plus \eqref{eq:y} for $k = i$.

Conversely suppose that the chambers of the triangle between $b^0 \ldots b^{\ell-1}$ and $a^0 \ldots a^{\ell-1}$ can be labeled so that they satisfy \eqref{eq:edge_chamber_formula}. Then there is a sequence of step moves and elementary homotopies transforming $b^0\ldots b^{\ell-1}$ into $a^0 \ldots a^{\ell-1}$. For each vertex of the triangle there is a subpath of an intermediate path that ends in this vertex. Thus we obtain an incidence preserving map of the vertices of the triangle to the vertices of $X$, showing that an isomorphic triangle exists in $X$.\qed
\end{proof}

For Hjelmslev planes up to level $3$ we give an explicit description that avoids the special notation of this section, see Figure~\ref{fig:hjelmslev-3} for the chamber labels used. This is the analogue of \cite[Lemmas~5.11, 5.12]{essert13}.

\begin{corollary}
\label{cor:hjelmslev_3}
The set $\calP^\ell$ of points and $\calL^\ell$ of lines of the Hjelmslev plane of radius $\ell$ around $v_0$ are
\begin{align*}
\hpts^\ell &= \{ \sigma_0^{b^0} \sigma_1^{b^1} \sigma_2^{b^2} \sigma_0^{b^3} \cdots \sigma_{\ell-1 \mmod 3}^{a^{\ell-1}}v_{\ell \mmod 3} \mid b_0 \in \Z/\delta\Z, b^k \in\Z/\delta\Z \setminus -D_{k \mmod 3}\} \text{ and}\\
\hlns^\ell &= \{ \sigma_0^{a^0} \sigma_2^{a^1} \sigma_1^{a^2} \sigma_0^{a^3} \cdots \sigma_{-\ell+1 \mmod 3}^{a^{\ell-1}}v_{-\ell \mmod 3} \mid a_0 \in \Z/\delta\Z, a^k \in\Z/\delta\Z \setminus D_{- k  \mmod 3}\}\text{.}
\end{align*}
The point $\sigma_0^{b^0}v_1$ and the line $\sigma_0^{a^0}v_2$are incident if and only if
\begin{enumerate}[label={(C\arabic*)}, ref={C\arabic*}]
\item there is a row $(e_0, e_1, e_2)$ of $E$ such that $b^0 - a^0 = e_0$.
\end{enumerate}
The point $\sigma_0^{b^0}\sigma_1^{b^1}v_2$ and the line $\sigma_0^{a^0}\sigma_2^{a^1}v_1$ are incident if in addition
\begin{enumerate}[label={(C\arabic*)}, ref={C\arabic*}]
\setcounter{enumi}{1}
\item there are rows $(f_0,f_1,f_2)$, $(g_0,g_1,g_2)$, and $(h_0,h_1,h_2)$ of $E$ such that $b^1 - e_1+f_1-g_1 = 0$ and $a^1 + e_2 - f_2 + h_2 = 0$.\label{item:hjelm_2}
\end{enumerate}
The point $\sigma_0^{b^0}\sigma_1^{b^1}\sigma_2^{b^2}v_0$ and the line $\sigma_0^{a^0}\sigma_2^{a^1}\sigma_1^{a^2}\sigma_1 v_0$  are incident if in addition
\begin{enumerate}[label={(C\arabic*)}, ref={C\arabic*}]
\setcounter{enumi}{2}
\item there are rows $(k_0,k_1,k_2)$, $(\ell_0,\ell_1,\ell_2)$, $(m_0, m_1,m_2)$, $(n_0,n_1,n_2)$, $(o_0,o_1,o_2)$ of $E$ such that $b^2_2-g_2+k_2-m_2 = 0$, $a^2_1+h_1-\ell_1+n_1 = 0$, $-f_0+g_0+h_0-k_0-\ell_0+o_0 = 0$.
\end{enumerate}
\end{corollary}

\begin{figure}
\centering
\includegraphics{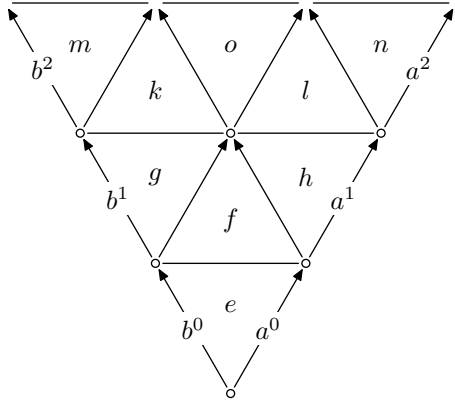}
\caption{Chamber labels used in Corollary~\ref{cor:hjelmslev_3}}
\label{fig:hjelmslev-3}
\end{figure}

Using Theorem~\ref{thm:hjelmslev} or Corollary~\ref{cor:hjelmslev_3} it is very easy to get explicit descriptions of Hjelmslev spheres of the three types of vertices in Singer cyclic lattices. This leaves us with the problem computing meaningful invariants for them. We will mostly concentrate on two invariants: the Moufang property, and the number of splittings of $\hjelm^2 \to \hjelm^1$.

\subsection{Moufang planes}

A Klingenberg plane $\hjelm^k$ is \emph{Moufang} if for every incident point-line pair $x \inc y$, $x \in \hpts, y \in \hlns$, for every point $x' \inc y$, $x' \not\sim x$ and for any two lines $z, z' \inc x'$, $y\not\sim z,z'$ there is a colineation that fixes all lines through $x$, all points on $y$, and takes $z$ to $z'$. The relevance of this property lies in \cite[Theorem~3.10]{baklanlor91} which implies:

\begin{theorem}
The Klingenberg plane associated to a commutative local ring as in Example~\ref{exmpl:hjelmslev_moufang} is Moufang.
\end{theorem}

Together with Fact~\ref{fact:hjelmslev_building_ring} this immediately gives us a way to show that the building associated to a Singer cyclic lattice is not Bruhat--Tits: we just need to find some Hjelmslev plane that is not Moufang. Tables~\ref{tab:buildings_2}, \ref{tab:buildings_3}, \ref{tab:buildings_4}, \ref{tab:buildings_5} show for which Singer cyclic lattices the Hjelmslev planes of level $2$ (and for $q \le 3$ also of level $3$) are Moufang.

\begin{theorem}
\label{thm:exotic_buildings}
For every prime power $q \le 5$ there is only one isomorphism class of Singer cyclic lattices that acts on a Bruhat--Tits building, all other act on exotic buildings.
\end{theorem}

\begin{proof}
In each table only one building has all Hjelmslev planes Moufang, all the others are exotic. We know that the first lattice acts on a Bruhat--Tits building from Corollary~\ref{cor:cartwright_to_panel}.\qed
\end{proof}

\begin{remark}
Bader, Caprace, and Lécureux \cite[Appendix~D]{badcaplec16} show that an exotic Singer cyclic lattice with parameter $q_0 \not\equiv 1 \mmod 3$ can be used to deduce that certain Singer cyclic lattices with parameter $q = q_0^3$, $e \not\equiv 0 \mmod 3$ are exotic as well. Namely, if $E_0$ is a difference matrix for $q_0$ and $E$ is a difference matrix for $q$ such that $E$ contains the rows of $\delta/\delta_0 \cdot E_0$ then there is an obvious morphism $\Gamma_0 \to \Gamma$ taking each generator $\sigma_i$ to $\sigma_i^{\delta/\delta_0}$. The main result of \cite{badcaplec16} implies that if $\Gamma$ is Bruhat--Tits then $\Gamma_0$ needs to have infinite image in the automorphism group of the field, and in particular have an infinite residually finite quotient. Thus if $\Gamma_0$ is (virtually) perfect in addition to being exotic, $\Gamma$ needs to be exotic as well.

Via this construction Theorem~\ref{thm:exotic_buildings} gives rise to further exotic lattices, providing some evidence toward the conjecture formulated in the introduction. However, for a given parameter $q$ the proportion of lattices that can be seen to be exotic among all Singer cyclic lattices for this parameter, tends to $0$ (at least when the strategy is applied and estimated naively).

Applying the method for $q_0 = 2$ and $q = 4$ and using that $\Gamma_{2,2}$ is exotic, only one lattice in Table~\ref{tab:lattices_4} can be shown to be exotic, namely $\Gamma_{4,2}$. One of its based difference matrices is
\[
\begin{psmallmatrix}
0 & 0 & 0\\3&3&9\\9&9&3\\4&4&4\\11&11&11
\end{psmallmatrix}
\]
and all possibilities to swap entries between the last two rows lead to equivalent difference matrices.
\end{remark}

Besides identifying exotic buildings, the Moufang property also allows us to distinguish between non-isomorphic buildings. For example we can see that there are at least three non-isomorphic buildings for $q = 3$ by just taking into account how many Hjelmslev planes are Moufang.

\begin{table}
\caption{Buildings of Singer cyclic lattices for $q = 2$. The numbering is the same as that of Table~\ref{tab:lattices_2}. The table shows a based difference matrix, whether the Hjelmslev planes of level $2$ respectively $3$ are Moufang, in how many ways the projection $\hjelm^2 \to \hjelm^1$ splits, some subquotients $Q_1^{k}$ of the vertex stabilizers, and the normalizer of the Singer cycle on level $2$. Each of these latter columns has three entries, one for Hjelmslev planes around $v_0$, $v_1$, and $v_2$.}
\label{tab:buildings_2}
\begin{gather*}
\begin{array}{cccccccc}
\text{Name} & \text{Based DM} & \substack{\hjelm^2\\\text{Moufang}} & \substack{\hjelm^3\\\text{Moufang}} & \substack{\#\,\text{splits of}\\\hjelm^2 \to \hjelm^1} & \abs{\stab{1}[][2]} & \abs{\stab{1}[2][3]} & \abs{N_{\stab{}[1][2]}(\gen{\sigma})/\gen{\sigma}}\\[.5em]
\input{buildings_2.tex}
\end{array}
\end{gather*}
\end{table}

\begin{table}
\caption{Buildings of Singer cyclic lattices for $q = 3$. The columns are the same as in Table~\ref{tab:buildings_2}. One can check that two Hjelmslev planes of level $2$ are isomorphic if and only if their data in this table coincide, including $\abs{N_{\stab{}[1][2]}(\gen{\sigma})/\gen{\sigma}}$ which is not a priori an invariant of $\hjelm^2$. All the Hjelmslev planes of level $2$ are also self-dual. In particular, only four isomorphism types of Hjelmslev planes of level $2$ appear. The only building that has three different planes of level $2$ is $X_{3,6}$. The only exotic building that has all planes of level $2$ isomorphic is $X_{3,7}$.}
\label{tab:buildings_3}
\begin{gather*}
\begin{array}{cccccccc}
\text{Name} & \text{Based DM} & \substack{\hjelm^2\\\text{Moufang}} & \substack{\hjelm^3\\\text{Moufang}} & \substack{\#\,\text{splits of}\\\hjelm^2 \to \hjelm^1} & \abs{\stab{1}[][2]} & \abs{\stab{1}[][3]} & \abs{N_{\stab{}[1][2]}(\gen{\sigma})/\gen{\sigma}}\\[.5em]
\input{buildings_3.tex}
\end{array}
\end{gather*}
\end{table}

\subsection{Splittings}

Another invariant to distinguish between different Hjelmslev planes is very simple. We know that for $k > \ell$ there is a projection $\hjelm^k \to \hjelm^\ell$. Thus we may wonder whether and in how many ways this projection splits. This is of course computationally most feasible if $k = 2$, $\ell = 1$.

Essert formulates a sufficient condition for the projection $\hjelm^2 \to \hjelm^1$ to split. The statement is not quite correct as stated. In the last sentence of \cite[Lemma~5.12]{essert13} the difference sets have to be equal as ordered difference sets, not as unordered difference sets. If it applied to unordered difference sets then the argument in \cite[Proposition~5.14]{essert13}, saying that the Hjelmslev planes split would apply to all Hjelmslev planes of $\Gamma_{2,2}$. But we know from Lemma~\ref{lem:22_ronan_tits} and Tits's invariant that only one of the Hjelmslev spheres of level $2$ of $\Gamma_{2,2}$ is that of $\F_2\pseries{t}/t^2$ while the other two are those of $\Z_2/4$ which do not split (which is also confirmed by Table~\ref{tab:buildings_2}).

On the other hand, looking at Hjelmslev planes around a single vertex separately, the assumption in Essert's arguments can be weakened to only depend on two difference sets. We formulate the correct version of \cite[Proposition~5.14]{essert13} only in our specific context of Singer cyclic lattices.

\begin{corollary}
\label{cor:equal_columns}
If the $1$st and the $2$nd column of $E$ coincide (as \emph{ordered} difference sets) then $\sigma_0^{a^0} \sigma_2^{a^1}. v_1$ is incident with $\sigma_0^{b^0} \sigma_1^{b^1} v_2$ if and only if there is a row $(e_0,e_1,e_1)$ of $E$ such that $b^0 - a^0 = e_0$ and there is an $n \in (b^1 - D_1) \cap (-a^1 - D_1) \cap (e_1 - D_1)$.

In particular, for every $m \in \Z/\delta\Z \setminus D_1$ the map
\begin{align*}
\sigma_0^{a^0}.v_2 &\mapsto \sigma_0^{a^0}\sigma_2^{-m}.v_1\\
\sigma_0^{b^0}.v_1 &\mapsto \sigma_0^{b^0}\sigma_1^m.v_2
\end{align*}
is a splitting of $\hjelm^2(v_0) \to \hjelm^1(v_0)$ in that case.
\end{corollary}

\begin{proof}
Given $n$ as in the statement, the relations for \eqref{item:hjelm_2} are satisfied with $f_1 = e_1 - n$, $g_1 = b^1 - n$, $h_2 = -a^1 - n$.

Given $m$ as in the statement, the target points of the map are actually points and lines of $\hjelm^2$. Incidence is preserved because $(m - D_1) \cap (e_1 - D_1) \ne \emptyset$ by the difference set property: to find $g_1, f_1 \in D_1$ with $m - g_1 = e_1 - f_1$ amounts to writing $m - e_1 = g_1 - f_1$, which is always possible.\qed
\end{proof}

\begin{remark}
\begin{enumerate}
\item Looking at Table~\ref{tab:buildings_3} it is tempting to think that the isomorphism type of the Hjelmslev sphere of level two around $v_i$ might depend only on the $(i-1)$st and $(i+1)$st column of $E$. This is not the case and a counterexample (the only one for $q = 3$) is given by the Hjelmslev plane around $v_0$ in $X_{3,2}$ and the Hjelmslev plane around $v_2$ in $X_{3,5}$. The obstruction is the following: a Singer cycle gives rise to a partition of the edges of the building into orbits. For two different cycles there is generally not an isomorphism taking one partition to the other.
\item A more modest guess would be that in the situation of Corollary~\ref{cor:equal_columns}, that is, if two columns of $E$ coincide, the Hjelmslev plane around the third vertex is Moufang. There is no counterexample for $q \le 4$ but there are counterexamples for $q = 5$, see the third row of Table~\ref{tab:buildings_5}.
\end{enumerate}
\end{remark}

The number of splittings of $\hjelm^2 \to \hjelm^1$ has been computed for all Singer cyclic lattices with $q \le 4$ and all vertices. The results are shown in Tables~\ref{tab:buildings_2},~\ref{tab:buildings_3},~\ref{tab:buildings_4}. It turns out that in the situation of Corollary~\ref{cor:equal_columns} there are in fact always $q^8$ splittings. The only splittings apart from these exist for $q = 3$ and are unique. An interesting example is $\Gamma_{3,7}$ where the projection has a unique splitting for each vertex. By taking the number of splittings into account, we see that for $q = 3$ there are at least five quasi-isometry classes of Singer cyclic lattices: only $\Gamma_{3,2}$ and $\Gamma_{3,3}$ as well as $\Gamma_{3,5}$ and $\Gamma_{3,6}$ could still be quasi-isometric.

\section{Building automorphisms}
\label{sec:building_auts}

Let $X$ be the building associated to a Singer cyclic lattice $\Gamma$. The goal of this section is to determine the full automorphism group $G \defeq \Aut(X)$ of $X$. We start by investigating the subgroup $G_0 \defeq \Aut_0(X)$ of type preserving automorphisms.

\begin{observation}
If $v \in X$ is any vertex then $G_0 = (G_0)_v \Gamma$.
\end{observation}

\begin{proof}
This is because $\Gamma$ is transitive on each type of vertices.\qed
\end{proof}

The observation allows us to reduce our problem to studying vertex stabilizers. So we fix a vertex $v$ and let $P = (G_0)_v$ denote the stabilizer of $v$ in $G_0$. We introduce the following subquotients. We let $\stab{r}$ be the pointwise stabilizer in $P$ of the combinatorial ball of radius $r$ around $v$ and we let $\stab{}[k]$ be the image of $P$ in the automorphism group of the ball of radius $k$ around $v$. Combining both constructions we define the group $\stab{r}[k]$ to be the image of $\stab{r}$ in $\stab{}[k]$. That is, if $r \le r'$ then $\stab{r'}[k]$ is a subgroup of $\stab{r}[k]$ and if $k \le k'$ then $\stab{r}[k]$ is a quotient of $\stab{r}[k']$. Then for $0 \le r \le \ell \le k \le \infty$ there is an exact sequence
\begin{equation}
\label{eq:stabilizer_ses}
1 \to \stab{\ell}[k] \to \stab{r}[k]  \to \stab{r}[\ell] \to 1
\end{equation}
(where the indices $r = 0$ or $k = \infty$ are understood to be omitted).

The phenomenon that we are after is the following.

\begin{lemma}
\label{lem:stabilizer_faithful}
Assume either that $\stab{1}[3]$ is trivial for some vertex or that $\stab{1}[2]$ is trivial for vertices of two different types .

Then $\stab{1}$ is trivial and so the projection $\stab{} \to \stab{}[1]$ is an isomorphism. That is, $\stab{}$ acts faithfully on $\lk v$.
\end{lemma}

\begin{proof}
Assume that $\stab{1}[3]$ is trivial and let $\alpha \in \stab{1}$. By assumption $\alpha$ fixes the combinatorial $3$-ball around $v$, in particular it fixes the full star of the vertices $v'$ that are at combinatorial distance $2$ from $v$ and of the same type. Since the automorphism group of $X$ is transitive on vertices of each type, the Lemma also holds with $v$ replaced by $v'$ and so $\alpha$ fixes the $3$-balls around all of these vertices. Proceeding inductively one gets that $\alpha = 1$.

For the second statement assume that $\stab{1}[2]$ is trivial for vertices of types $0$ and $1$. The argument proceeds in the same way as before and we need to show that the sequence of subcomplexes defined by $K_0 \defeq B_2(v)$,
\[
K_{i+1} \defeq \bigcup \{B_2(v') \mid v' \text{ interior vertex of type $0$ or $1$ of }K_i\}
\]
exhausts the whole building. This is saying that for every vertex $w$ of $X$ there is an edge path through vertices $v = v_0, v_1, \ldots v_{\ell-1}, v_\ell = w$ such that $v_0, \ldots v_{\ell-1}$ are of type $0$ or $1$. To see that this statement is true, take a gallery from a chamber containing $v$ to a chamber containing $w$ and note that every panel contains a vertex of type $0$ or a vertex of type $1$.

The remaining statement is just the sequence \eqref{eq:stabilizer_ses} in case $r = 0, \ell = 1, k = \infty$.\qed
\end{proof}

\begin{corollary}
If the assumptions of Lemma~\ref{lem:stabilizer_faithful} hold then $G_0$ is discrete.
\end{corollary}

\begin{proof}
Discreteness means that the stabilizer of a large enough ball is trivial. Lemma~\ref{lem:stabilizer_faithful} asserts that combinatorial balls of radius $3$ around vertices are trivial.\qed
\end{proof}

The group $\stab{}$ is not accessible to computer experiments but it can be approximated by the group $\stab{}[][r]$ of all automorphisms of the ball of radius $k$ around $v$:
\[
\stab{} = \bigcap_{r \in \N} \stab{}[][r]\text{.}
\]
In particular $\stab{}[k] < \stab{}[][k]$. Similarly, letting $\stab{r}[][k]$ denote the stabilizer of the ball of radius $r$ in $\stab{}[][k]$ we get that $\stab{r}[k][] < \stab{r}[][k]$. Thus the hypothesis of Lemma~\ref{lem:stabilizer_faithful} can be checked by verifying that $\stab{1}[][k]$ is trivial for $k = 2$ or $k = 3$. It happens, however, that $\stab{1}[2]$ is trivial while $\stab{1}[][2]$ as well as $\stab{2}[][3]$ are non-trivial. In these cases we look at the group $\stab{1}[2][3]$ of isomorphisms on level $2$ that fix level $1$ and extend to level $3$. This is still a supergroup of $\stab{1}[2]$, so it is enough to prove it trivial around two types of vertices. The orders of these groups are shown for $q \le 5$ in Tables~\ref{tab:buildings_2},~\ref{tab:buildings_3},~\ref{tab:buildings_4}, and~\ref{tab:buildings_5}.

In cases where Lemma~\ref{lem:stabilizer_faithful} applies, it remains to study $\stab{}[1][]$. Thanks to the existence of a Singer cyclic lattice, we already know that $\stab{}[1][]$ contains a Singer cycle. The remaining possible groups are limited by the following lemma.

\begin{lemma}
\label{lem:normalizer}
The subgroup $M \defeq \Gal(\F_{q^3}/\F_p) \ltimes \F_{q^3}^\times/\F_q^\times < \PGL_3(\F_q)$ is maximal.
\end{lemma}

\begin{proof}
By \cite{kantor80} any proper subgroup of $\PGL_3(\F_q)$ that contains the Singer cycle has to normalize it and by \cite[II.7.3a]{huppert67} $M$ is the full normalizer.

Alternatively, one can use Main~Theorem~2.1.1 together with Table~8.3 in \cite{braholron13}. One then has to read Sections~1.3.1 and~1.7.1 of that reference to see that the maximal subgroup of $\SL_3(\F_q)$ gives rise to a maximal subgroup of $\PGL_3(\F_q)$ by moding out the center and then extending by $\delta$.\qed
\end{proof}

Lemma~\ref{lem:normalizer} tells us that $\stab{}[1][]$ is either all of $\PGL_3(\F_q)$ or it lies between the Singer group $\Gamma_v$ and its normalizer. The former possibility is ruled out in most cases by Theorem~\ref{thm:chamber_trans}: if $G_0$ is panel transitive and $\stab{}[1][] = \PGL_3(\F_q)$ then $G$ is chamber-transitive, which can only happen in case $q = 2$ (where it does happen, as we have seen in Section~\ref{sec:ronan}) or $q = 8$, which is beyond the scope of this article.

Let $\sigma \in \Gamma$ denote a Singer cycle around our base vertex $v$. The discussion so far has shown that if $\stab{}[][] \cong \stab{}[1][]$ and if $q \ne 2, 8$ then $\stab{}[1][]$ normalizes $\gen{\sigma}$. A lower bound for $\stab{}[1][]$ is given by $\Aut(\Gamma)_v$. Let $\varphi$ be a generator of $N_{\stab{}[1][]}(\gen{\sigma})/\gen{\sigma}$ (the Frobenius automorphism in the correspondence of the Lemma). The procedure to produce an upper bound is similar to the stabilizer computations before. We check whether some subgroup $\gen{\varphi^k}$ of $\gen{\varphi}$ lifts to an element of $\stab{}[][2]$. The normalizer of $\gen{\sigma}$ in $\stab{}[2][1]$ is then $\gen{\sigma} \rtimes \gen{\varphi^k}$ for the maximal $\gen{\varphi^k}$. If $q$ is prime (in particular for $p \in \{2,3,5\}$) then $\varphi$ has order $3$ and it suffices to check whether $\varphi$ lifts. If $q = p^\eta$ with $\eta > 1$ then the order of $\varphi$ is $3 \eta$ and we need to check elements of the various prime powers separately. For $q = 4$ the order is $6$ and we need to check whether $\varphi^2$ and $\varphi^3$ lift. The results are again listed in Tables~\ref{tab:buildings_2},~\ref{tab:buildings_3},~\ref{tab:buildings_4}, and~\ref{tab:buildings_5}. 

Comparing them with Tables~\ref{tab:lattices_2},~\ref{tab:lattices_3},~\ref{tab:lattices_4}, and~\ref{tab:lattices_5} we find:
\begin{theorem}
\label{thm:aut}
Let $\Gamma$ be a Singer cyclic lattice and let $X$ be its building. If $q \le 5$ and $\Gamma$ is not one of the Bruhat--Tits lattices from Section~\ref{sec:brutit_examples} then $\Aut(\Gamma) = \Aut(X)$, which is never type-transitive.
\end{theorem}

\begin{proof}
Let $\Aut_0(\Gamma)$ and $\Aut_0(X)$ denote the groups of type preserving automorphisms of $\Gamma$ and $X$ respectively.

For $q = 2$ we only need to consider $\Gamma_{2,2}$. Looking at Table~\ref{tab:lattices_2} we see that $P_1^{3}$ is trivial for vertices of type $0$ so we can apply Lemma~\ref{lem:stabilizer_faithful} to conclude that the stabilizer of $v_0$ acts faithfully on $\lk v_0$. We have already seen in Section~\ref{sec:ronan} that $\Aut_0(\Gamma_{2,2})$ is chamber regular, so $P_1 = \PGL_3(\F_2)$ and $\Aut_0(\Gamma_{2,1}) = \Aut_0(X)$. Looking at Table~\ref{tab:lattices_2} we see that $\Aut(\Gamma_{2,2})$ acts by exchanging two types. From Table~\ref{tab:buildings_2} (or the discussion in Section~\ref{sec:ronan}) we see that $\Aut(X_{2,2})$ cannot permute all three types since the Hjelmslev planes of level $2$ are non-isomorphic.

For $q = 3$, we need to consider the lattices $\Gamma_{3,2}$ to $\Gamma_{3,7}$. We see in Table~\ref{tab:lattices_3} that for these lattices $\stab{1}[][3]$ is trivial for all vertices so that the stabilizer of any vertex $v$ acts faithfully on $\lk v$. We know from Theorem~\ref{thm:chamber_trans} that the vertex stabilizer cannot be all of $\PGL_3(\F_3)$ and so by Lemma~\ref{lem:normalizer} it has to normalize the Singer cycle in $\Gamma$. Since $q$ is prime the order of $N_{P}(\gen{\sigma})/\gen{\sigma}$ can only be $3$ or $1$. We see in Table~\ref{tab:lattices_3} that the order is $3$ already within $\Aut_0(\Gamma_{3,2})$ so that $\Aut_0(\Gamma_{3,2}) = \Aut_0(X_{3,2})$ is an extension of $\Gamma_{3,2}$ by $C_3$. In all the other cases Table~\ref{tab:buildings_3} shows at least one vertex where the Singer cycle has trivial normalizer at level $2$ and thus $\Aut_0(\Gamma_{3,i}) = \Gamma = \Aut_0(X_{3,i})$ for $i > 2$.

Toward the full automorphism group, for $2 \le i \le 4$ we see in Table~\ref{tab:lattices_3} that $\Aut(\Gamma)$ exchanges two types and from Table~\ref{tab:buildings_3} we see that it cannot exchange more (for example because $\hjelm^2(v_2)$ is Moufang while $\hjelm^2(v_0)$ and $\hjelm^2(v_1)$ are not). The remaining groups satisfy
\begin{equation}
\label{eq:aut0_gamma_aut}
\Aut_0(X) = \Gamma = \Aut(\Gamma)\text{.}
\end{equation}
We claim that these equations already imply that $\Aut(X) = \Aut_0(X)$. Indeed suppose $\theta \in \Aut(X) \setminus \Aut_0(X)$ and consider the group of type-preserving automorphisms $\gen{\Gamma^{\gen{\theta}}}$. If this group equals $\Gamma$ then $\theta$ normalizes $\Gamma$ contradicting the right equation of \eqref{eq:aut0_gamma_aut}. If it is strictly larger than $\Gamma$ the left equation of \eqref{eq:aut0_gamma_aut} is violated.

For $q =4$, we look at the lattices $\Gamma_{4,2}$ to $\Gamma_{4,17}$. Again we see in Table~\ref{tab:lattices_4} that in each building either $\stab{1}[2][3]$ is trivial for two vertex types, or that $\stab{1}[][3]$ is trivial for some vertex so the vertex stabilizers act faithfully on the corresponding vertex links. As before Theorem~\ref{thm:chamber_trans} and Lemma~\ref{lem:normalizer} imply that every vertex stabilizer has to normalize the Singer cycle in $\Gamma$. This time the possible values for the order of $N_{P}(\gen{\sigma})/\gen{\sigma}$ are $1$, $3$, and $6$, since $4 = 2^2$. Table~\ref{tab:lattices_4} shows that the order is $6$ in $\Aut_0(\Gamma_{4,2})$ and Table~\ref{tab:buildings_4} shows that the order is $1$ in all other cases.

For the full automorphism group the argument is just as for $q = 3$: the lattices $\Gamma_{4,i}$, $2 \le i \le 5$ exchange two types by Table~\ref{tab:lattices_4} and cannot exchange all types by the number of splits in Table~\ref{tab:buildings_4}. The remaining lattices satisfy \eqref{eq:aut0_gamma_aut}.

For $q = 5$ (excluding the Bruhat--Tits example in the first row) the same reasoning as before applies to see that $\Aut_0(\Gamma) = \Aut_0(X)$ from Tables~\ref{tab:lattices_5} and~\ref{tab:buildings_5}. Whenever $\Aut(\Gamma)$ strictly contains $\Aut_0(\Gamma)$ only two types are exchanged and it can be seen from the column $\abs{\stab{1}[][2]}$ of Table~\ref{tab:buildings_5} that the third type could not be exchanged with the other two, so that $\Aut(\Gamma) = \Aut(X)$. When \eqref{eq:aut0_gamma_aut} holds, we conclude as before. This leaves us with lattices where $\Out(\Gamma) = C_3$. We tweak the previous argument a bit to see that $\Aut(X) = \Aut_0(X)$ even in these cases: suppose $\theta$ was a building automorphism that did not preserve types. We know that $\theta$ cannot normalize $\Gamma$ so $\Gamma$ and $\Gamma^\theta$ are distinct. Since $\Gamma$ is generated by Singer cycles, there has to be one that is not contained in $\Gamma^\theta$. But then the index of $\Gamma$ in $\gen{\Gamma, \Gamma^\theta}$ is at least $5$ contradicting that $\Gamma$ has index $3$ in $\Aut_0(X)$.\qed
\end{proof}

\begin{corollary}
Let $\Gamma \curvearrowright X$ and $\Gamma' \curvearrowright X'$ be two Singer cyclic lattices with $q \le 5$ acting on their associated buildings. If $\Gamma$ is not isomorphic to $\Gamma'$ then $X$ is not isomorphic to $X'$. In particular, $\Gamma$ and $\Gamma'$ are then not quasi-isometric.
\end{corollary}

\begin{proof}
We assume $X = X'$ and need to show $\Gamma \cong \Gamma'$. The statement is clear if $\Gamma = \Aut_0(X)$ or (using Tables~\ref{tab:buildings_2},~\ref{tab:buildings_3},~\ref{tab:buildings_4} and~\ref{tab:buildings_5}) if $X$ is Bruhat--Tits. For $q \le 4$ there is only at most one exotic building left in the tables with $\Aut_0(X) \gneq \Gamma$. For $q = 5$ we conclude as at the end of the last proof: if $\Gamma$ and $\Gamma'$ are distinct Singer cyclic lattices on the same building, then the index of $\Gamma$ in $\gen{\Gamma,\Gamma'}$ is at least $5$ but the index in $\Aut_0(X)$ is at most $3$.

The last sentence follows from Theorem~\ref{thm:qi_rigidity}.\qed
\end{proof}

\begin{longtable}{>{$}c<{$}>{$}c<{$}>{$}c<{$}>{$}c<{$}>{$}c<{$}>{$}c<{$}}
\caption{Singer cyclic lattices for $q = 4$. The columns are the same as in Table~\ref{tab:lattices_2}.}\endfirsthead
\caption{continued}\endhead
\label{tab:lattices_4}
\text{Name} & \text{Diff mat} & \text{Based DM} & \textrm{Out}(\Gamma) &H_1(\Gamma) & H_1([\Gamma,\Gamma])\\
\input{lattices_4.tex}
\end{longtable}

\clearpage

\begin{longtable}{>{$}c<{$}>{$}c<{$}>{$}c<{$}>{$}c<{$}>{$}c<{$}>{$}c<{$}>{$}c<{$}>{$}c<{$}}
\caption{Buildings of Singer cyclic lattices for $q = 4$. The columns are the same as in Table~\ref{tab:buildings_2}.}\endfirsthead
\caption{continued}\endhead
\label{tab:buildings_4}
\text{Name} & \text{Based DM} & \substack{\hjelm^2\\\text{Moufang}} & \substack{\#\,\text{splits of}\\\hjelm^2 \to \hjelm^1} & \abs{\stab{1}[][2]} & \abs{\stab{1}[3][2]} & \abs{N_{\stab{}[][2]}(\gen{\sigma})/\gen{\sigma}}\\
\input{buildings_4.tex}
\end{longtable}

\clearpage

\begin{longtable}{>{$}c<{$}>{$}c<{$}>{$}c<{$}>{$}c<{$}>{$}c<{$}>{$}c<{$}>{$}c<{$}}
\caption{Summary of Singer cyclic lattices for $q = 5$. The leftmost column shows how many lattices the row represents. The based difference matrix is one example that realizes the indicated properties.  The other columns are the same as in Table~\ref{tab:lattices_2}.}\endfirsthead
\caption{continued}\endhead
\label{tab:lattices_5}
\# & \text{Based DM} & \textrm{Out}(\Gamma) &H_1(\Gamma) & H_1([\Gamma,\Gamma])\\
\input{lattices_5.tex}
\end{longtable}

\clearpage

\begin{longtable}{>{$}c<{$}>{$}c<{$}>{$}c<{$}>{$}c<{$}>{$}c<{$}>{$}c<{$}>{$}c<{$}}
\caption{Summary of buildings of Singer cyclic lattices for $q = 5$. The leftmost column shows how many lattices the row represents. The based difference matrix is one example that realizes the indicated properties. The other columns are the same as in Table~\ref{tab:buildings_2}.}\endfirsthead
\caption{continued}\endhead
\label{tab:buildings_5}
\# & \text{Based DM} & \substack{\hjelm^2\\\text{Moufang}} & \substack{\#\,\text{splits of}\\\hjelm^2 \to \hjelm^1} & \abs{\stab{1}[][2]} & \abs{N_{\stab{}[][2]}(\gen{\sigma})/\gen{\sigma}}\\
\input{buildings_5.tex}
\end{longtable}

\bibliographystyle{amsalpha}
\bibliography{singer_lattice.bib}

\end{document}

%% file: lattices_2.tex
\Gamma_{2, 1}& \left(\begin{smallarray}{ccc}0&0&0\\1&1&1\\3&3&3\end{smallarray}\right)& \left(\begin{smallarray}{ccc}0&0&0\\1&1&1\\3&3&3\end{smallarray}\right)& C_3 \rtimes S_3& (\Z/7\Z)^{2}& (\Z/2\Z)^{6}\\
\Gamma_{2, 2}& \left(\begin{smallarray}{ccc}0&0&0\\1&1&3\\3&3&1\end{smallarray}\right)& \left(\begin{smallarray}{ccc}0&0&0\\1&1&3\\3&3&1\end{smallarray}\right)& C_3 \rtimes (0, 1)& \Z/7\Z& 0

%% file: lattices_3.tex
\Gamma_{3, 1}& \left(\begin{smallarray}{ccc}0&0&0\\1&1&1\\3&3&3\\9&9&9\end{smallarray}\right)& \left(\begin{smallarray}{ccc}0&0&0\\1&1&1\\3&3&3\\9&9&9\end{smallarray}\right)& C_3 \rtimes S_3& (\Z/13\Z)^{2}& (\Z/3\Z)^{6}\\
\Gamma_{3, 2}& \left(\begin{smallarray}{ccc}0&0&0\\1&1&1\\3&3&9\\9&9&3\end{smallarray}\right)& \left(\begin{smallarray}{ccc}0&0&0\\1&1&1\\3&3&9\\9&9&3\end{smallarray}\right)& C_3 \rtimes (0, 1)& \Z/13\Z& 0\\
\Gamma_{3, 3}& \left(\begin{smallarray}{ccc}0&0&1\\1&1&0\\3&3&3\\9&9&9\end{smallarray}\right)& \left(\begin{smallarray}{ccc}0&0&0\\1&1&1\\3&3&4\\9&9&6\end{smallarray}\right)& (0, 1)& \Z/13\Z& 0\\
\Gamma_{3, 4}& \left(\begin{smallarray}{ccc}0&0&1\\1&1&0\\3&3&9\\9&9&3\end{smallarray}\right)& \left(\begin{smallarray}{ccc}0&0&0\\1&1&1\\3&3&6\\9&9&4\end{smallarray}\right)& (0, 1)& \Z/13\Z& 0\\
\Gamma_{3, 5}& \left(\begin{smallarray}{ccc}0&0&1\\1&1&0\\3&9&3\\9&3&9\end{smallarray}\right)& \left(\begin{smallarray}{ccc}0&0&0\\1&1&1\\3&9&4\\9&3&6\end{smallarray}\right)& 1& 0& 0\\
\Gamma_{3, 6}& \left(\begin{smallarray}{ccc}0&1&1\\1&0&3\\3&3&0\\9&9&9\end{smallarray}\right)& \left(\begin{smallarray}{ccc}0&0&0\\1&1&6\\3&4&4\\9&6&1\end{smallarray}\right)& 1& 0& 0\\
\Gamma_{3, 7}& \left(\begin{smallarray}{ccc}0&1&1\\1&0&3\\3&9&0\\9&3&9\end{smallarray}\right)& \left(\begin{smallarray}{ccc}0&0&0\\1&1&4\\3&6&1\\9&4&6\end{smallarray}\right)& 1& 0& 0

%% file: buildings_2.tex
X_{2, 1}& \left(\begin{smallarray}{ccc}0&0&0\\1&1&1\\3&3&3\end{smallarray}\right)& \makebox[1.0em][c]{$+$} \makebox[1.0em][c]{$+$} \makebox[1.0em][c]{$+$}& \makebox[1.0em][c]{$+$} \makebox[1.0em][c]{$+$} \makebox[1.0em][c]{$+$}& \raisebox{0.5em}{\makebox[1.5em][c]{$256$}} \raisebox{-0.5em}{\makebox[1.5em][c]{$256$}} \raisebox{0.5em}{\makebox[1.5em][c]{$256$}}& \raisebox{-0.5em}{\makebox[1.9em][c]{$256$}} \raisebox{0.5em}{\makebox[1.9em][c]{$256$}} \raisebox{-0.5em}{\makebox[1.9em][c]{$256$}}& \raisebox{0.5em}{\makebox[2.4em][c]{$256$}} \raisebox{-0.5em}{\makebox[2.4em][c]{$256$}} \raisebox{0.5em}{\makebox[2.4em][c]{$256$}}& \makebox[1.0em][c]{$3$} \makebox[1.0em][c]{$3$} \makebox[1.0em][c]{$3$}\\
X_{2, 2}& \left(\begin{smallarray}{ccc}0&0&0\\1&1&3\\3&3&1\end{smallarray}\right)& \makebox[1.0em][c]{$+$} \makebox[1.0em][c]{$+$} \makebox[1.0em][c]{$+$}& \makebox[1.0em][c]{$-$} \makebox[1.0em][c]{$-$} \makebox[1.0em][c]{$-$}& \raisebox{0.5em}{\makebox[1.5em][c]{$0$}} \raisebox{-0.5em}{\makebox[1.5em][c]{$0$}} \raisebox{0.5em}{\makebox[1.5em][c]{$256$}}& \raisebox{-0.5em}{\makebox[1.9em][c]{$256$}} \raisebox{0.5em}{\makebox[1.9em][c]{$256$}} \raisebox{-0.5em}{\makebox[1.9em][c]{$256$}}& \raisebox{0.5em}{\makebox[2.4em][c]{$1$}} \raisebox{-0.5em}{\makebox[2.4em][c]{$1$}} \raisebox{0.5em}{\makebox[2.4em][c]{$256$}}& \makebox[1.0em][c]{$3$} \makebox[1.0em][c]{$3$} \makebox[1.0em][c]{$3$}

%% file: buildings_3.tex
X_{3, 1}& \left(\begin{smallarray}{ccc}0&0&0\\1&1&1\\3&3&3\\9&9&9\end{smallarray}\right)& \makebox[1.0em][c]{$+$} \makebox[1.0em][c]{$+$} \makebox[1.0em][c]{$+$}& \makebox[1.0em][c]{$+$} \makebox[1.0em][c]{$+$} \makebox[1.0em][c]{$+$}& \raisebox{0.5em}{\makebox[1.5em][c]{$6561$}} \raisebox{-0.5em}{\makebox[1.5em][c]{$6561$}} \raisebox{0.5em}{\makebox[1.5em][c]{$6561$}}& \raisebox{-0.5em}{\makebox[1.9em][c]{$13122$}} \raisebox{0.5em}{\makebox[1.9em][c]{$13122$}} \raisebox{-0.5em}{\makebox[1.9em][c]{$13122$}}& \raisebox{0.5em}{\makebox[2.4em][c]{$258280326$}} \raisebox{-0.5em}{\makebox[2.4em][c]{$258280326$}} \raisebox{0.5em}{\makebox[2.4em][c]{$258280326$}}& \makebox[1.0em][c]{$3$} \makebox[1.0em][c]{$3$} \makebox[1.0em][c]{$3$}\\
X_{3, 2}& \left(\begin{smallarray}{ccc}0&0&0\\1&1&1\\3&3&9\\9&9&3\end{smallarray}\right)& \makebox[1.0em][c]{$-$} \makebox[1.0em][c]{$-$} \makebox[1.0em][c]{$+$}& \makebox[1.0em][c]{$-$} \makebox[1.0em][c]{$-$} \makebox[1.0em][c]{$-$}& \raisebox{0.5em}{\makebox[1.5em][c]{$0$}} \raisebox{-0.5em}{\makebox[1.5em][c]{$0$}} \raisebox{0.5em}{\makebox[1.5em][c]{$6561$}}& \raisebox{-0.5em}{\makebox[1.9em][c]{$3$}} \raisebox{0.5em}{\makebox[1.9em][c]{$3$}} \raisebox{-0.5em}{\makebox[1.9em][c]{$13122$}}& \raisebox{0.5em}{\makebox[2.4em][c]{$1$}} \raisebox{-0.5em}{\makebox[2.4em][c]{$1$}} \raisebox{0.5em}{\makebox[2.4em][c]{$1$}}& \makebox[1.0em][c]{$3$} \makebox[1.0em][c]{$3$} \makebox[1.0em][c]{$3$}\\
X_{3, 3}& \left(\begin{smallarray}{ccc}0&0&0\\1&1&1\\3&3&4\\9&9&6\end{smallarray}\right)& \makebox[1.0em][c]{$-$} \makebox[1.0em][c]{$-$} \makebox[1.0em][c]{$+$}& \makebox[1.0em][c]{$-$} \makebox[1.0em][c]{$-$} \makebox[1.0em][c]{$-$}& \raisebox{0.5em}{\makebox[1.5em][c]{$0$}} \raisebox{-0.5em}{\makebox[1.5em][c]{$0$}} \raisebox{0.5em}{\makebox[1.5em][c]{$6561$}}& \raisebox{-0.5em}{\makebox[1.9em][c]{$3$}} \raisebox{0.5em}{\makebox[1.9em][c]{$3$}} \raisebox{-0.5em}{\makebox[1.9em][c]{$13122$}}& \raisebox{0.5em}{\makebox[2.4em][c]{$1$}} \raisebox{-0.5em}{\makebox[2.4em][c]{$1$}} \raisebox{0.5em}{\makebox[2.4em][c]{$1$}}& \makebox[1.0em][c]{$1$} \makebox[1.0em][c]{$1$} \makebox[1.0em][c]{$3$}\\
X_{3, 4}& \left(\begin{smallarray}{ccc}0&0&0\\1&1&1\\3&3&6\\9&9&4\end{smallarray}\right)& \makebox[1.0em][c]{$-$} \makebox[1.0em][c]{$-$} \makebox[1.0em][c]{$+$}& \makebox[1.0em][c]{$-$} \makebox[1.0em][c]{$-$} \makebox[1.0em][c]{$-$}& \raisebox{0.5em}{\makebox[1.5em][c]{$1$}} \raisebox{-0.5em}{\makebox[1.5em][c]{$1$}} \raisebox{0.5em}{\makebox[1.5em][c]{$6561$}}& \raisebox{-0.5em}{\makebox[1.9em][c]{$2$}} \raisebox{0.5em}{\makebox[1.9em][c]{$2$}} \raisebox{-0.5em}{\makebox[1.9em][c]{$13122$}}& \raisebox{0.5em}{\makebox[2.4em][c]{$1$}} \raisebox{-0.5em}{\makebox[2.4em][c]{$1$}} \raisebox{0.5em}{\makebox[2.4em][c]{$1$}}& \makebox[1.0em][c]{$1$} \makebox[1.0em][c]{$1$} \makebox[1.0em][c]{$3$}\\
X_{3, 5}& \left(\begin{smallarray}{ccc}0&0&0\\1&1&1\\3&9&4\\9&3&6\end{smallarray}\right)& \makebox[1.0em][c]{$-$} \makebox[1.0em][c]{$-$} \makebox[1.0em][c]{$-$}& \makebox[1.0em][c]{$-$} \makebox[1.0em][c]{$-$} \makebox[1.0em][c]{$-$}& \raisebox{0.5em}{\makebox[1.5em][c]{$1$}} \raisebox{-0.5em}{\makebox[1.5em][c]{$0$}} \raisebox{0.5em}{\makebox[1.5em][c]{$0$}}& \raisebox{-0.5em}{\makebox[1.9em][c]{$2$}} \raisebox{0.5em}{\makebox[1.9em][c]{$3$}} \raisebox{-0.5em}{\makebox[1.9em][c]{$3$}}& \raisebox{0.5em}{\makebox[2.4em][c]{$1$}} \raisebox{-0.5em}{\makebox[2.4em][c]{$1$}} \raisebox{0.5em}{\makebox[2.4em][c]{$1$}}& \makebox[1.0em][c]{$1$} \makebox[1.0em][c]{$1$} \makebox[1.0em][c]{$1$}\\
X_{3, 6}& \left(\begin{smallarray}{ccc}0&0&0\\1&1&6\\3&4&4\\9&6&1\end{smallarray}\right)& \makebox[1.0em][c]{$-$} \makebox[1.0em][c]{$-$} \makebox[1.0em][c]{$-$}& \makebox[1.0em][c]{$-$} \makebox[1.0em][c]{$-$} \makebox[1.0em][c]{$-$}& \raisebox{0.5em}{\makebox[1.5em][c]{$0$}} \raisebox{-0.5em}{\makebox[1.5em][c]{$1$}} \raisebox{0.5em}{\makebox[1.5em][c]{$0$}}& \raisebox{-0.5em}{\makebox[1.9em][c]{$3$}} \raisebox{0.5em}{\makebox[1.9em][c]{$2$}} \raisebox{-0.5em}{\makebox[1.9em][c]{$3$}}& \raisebox{0.5em}{\makebox[2.4em][c]{$1$}} \raisebox{-0.5em}{\makebox[2.4em][c]{$1$}} \raisebox{0.5em}{\makebox[2.4em][c]{$1$}}& \makebox[1.0em][c]{$3$} \makebox[1.0em][c]{$1$} \makebox[1.0em][c]{$1$}\\
X_{3, 7}& \left(\begin{smallarray}{ccc}0&0&0\\1&1&4\\3&6&1\\9&4&6\end{smallarray}\right)& \makebox[1.0em][c]{$-$} \makebox[1.0em][c]{$-$} \makebox[1.0em][c]{$-$}& \makebox[1.0em][c]{$-$} \makebox[1.0em][c]{$-$} \makebox[1.0em][c]{$-$}& \raisebox{0.5em}{\makebox[1.5em][c]{$1$}} \raisebox{-0.5em}{\makebox[1.5em][c]{$1$}} \raisebox{0.5em}{\makebox[1.5em][c]{$1$}}& \raisebox{-0.5em}{\makebox[1.9em][c]{$2$}} \raisebox{0.5em}{\makebox[1.9em][c]{$2$}} \raisebox{-0.5em}{\makebox[1.9em][c]{$2$}}& \raisebox{0.5em}{\makebox[2.4em][c]{$1$}} \raisebox{-0.5em}{\makebox[2.4em][c]{$1$}} \raisebox{0.5em}{\makebox[2.4em][c]{$1$}}& \makebox[1.0em][c]{$1$} \makebox[1.0em][c]{$1$} \makebox[1.0em][c]{$1$}

%% file: lattices_4.tex
\Gamma_{4, 1}& \left(\begin{smallarray}{ccc}0&0&0\\1&1&1\\4&4&4\\14&14&14\\16&16&16\end{smallarray}\right)& \left(\begin{smallarray}{ccc}0&0&0\\1&1&1\\4&4&4\\14&14&14\\16&16&16\end{smallarray}\right)& C_6 \rtimes S_3& (\Z/3\Z)^{2}\oplus(\Z/7\Z)^{2}& (\Z/2\Z)^{12}\\
\Gamma_{4, 2}& \left(\begin{smallarray}{ccc}0&0&0\\1&1&1\\4&4&16\\14&14&14\\16&16&4\end{smallarray}\right)& \left(\begin{smallarray}{ccc}0&0&0\\1&1&1\\4&4&16\\14&14&14\\16&16&4\end{smallarray}\right)& C_6 \rtimes (0, 1)& (\Z/3\Z)^{2}\oplus\Z/7\Z& 0\\
\Gamma_{4, 3}& \left(\begin{smallarray}{ccc}0&0&0\\1&1&1\\4&4&4\\14&14&16\\16&16&14\end{smallarray}\right)& \left(\begin{smallarray}{ccc}0&0&0\\1&1&1\\4&4&4\\14&14&16\\16&16&14\end{smallarray}\right)& (0, 1)& \Z/3\Z\oplus\Z/7\Z& 0\\
\Gamma_{4, 4}& \left(\begin{smallarray}{ccc}0&0&0\\1&1&1\\4&4&14\\14&14&16\\16&16&4\end{smallarray}\right)& \left(\begin{smallarray}{ccc}0&0&0\\1&1&1\\4&4&14\\14&14&16\\16&16&4\end{smallarray}\right)& (0, 1)& \Z/3\Z\oplus\Z/7\Z& 0\\
\Gamma_{4, 5}& \left(\begin{smallarray}{ccc}0&0&1\\1&1&0\\4&4&4\\14&14&16\\16&16&14\end{smallarray}\right)& \left(\begin{smallarray}{ccc}0&0&0\\1&1&1\\4&4&8\\14&14&6\\16&16&18\end{smallarray}\right)& (0, 1)& \Z/3\Z\oplus\Z/7\Z& 0\\
\Gamma_{4, 6}& \left(\begin{smallarray}{ccc}0&0&1\\1&1&0\\4&14&4\\14&16&16\\16&4&14\end{smallarray}\right)& \left(\begin{smallarray}{ccc}0&0&0\\1&1&1\\4&14&18\\14&16&6\\16&4&8\end{smallarray}\right)& 1& 0& 0\\
\Gamma_{4, 7}& \left(\begin{smallarray}{ccc}0&0&1\\1&1&0\\4&4&14\\14&16&4\\16&14&16\end{smallarray}\right)& \left(\begin{smallarray}{ccc}0&0&0\\1&1&1\\4&4&8\\14&16&18\\16&14&6\end{smallarray}\right)& 1& 0& 0\\
\Gamma_{4, 8}& \left(\begin{smallarray}{ccc}0&0&1\\1&1&4\\4&14&0\\14&4&14\\16&16&16\end{smallarray}\right)& \left(\begin{smallarray}{ccc}0&0&0\\1&1&6\\4&4&18\\14&16&1\\16&14&8\end{smallarray}\right)& 1& \Z/3\Z& 0\\
\Gamma_{4, 9}& \left(\begin{smallarray}{ccc}0&0&0\\1&1&1\\4&4&16\\14&16&4\\16&14&14\end{smallarray}\right)& \left(\begin{smallarray}{ccc}0&0&0\\1&1&1\\4&4&16\\14&16&4\\16&14&14\end{smallarray}\right)& 1& \Z/3\Z& 0\\
\Gamma_{4, 10}& \left(\begin{smallarray}{ccc}0&0&1\\1&1&0\\4&16&4\\14&14&16\\16&4&14\end{smallarray}\right)& \left(\begin{smallarray}{ccc}0&0&0\\1&1&1\\4&16&8\\14&14&6\\16&4&18\end{smallarray}\right)& 1& \Z/3\Z& 0\\
\Gamma_{4, 11}& \left(\begin{smallarray}{ccc}0&0&1\\1&1&4\\4&4&16\\14&16&0\\16&14&14\end{smallarray}\right)& \left(\begin{smallarray}{ccc}0&0&0\\1&1&6\\4&14&1\\14&4&8\\16&16&18\end{smallarray}\right)& 1& \Z/3\Z& 0\\
\Gamma_{4, 12}& \left(\begin{smallarray}{ccc}0&0&0\\1&1&1\\4&4&14\\14&16&4\\16&14&16\end{smallarray}\right)& \left(\begin{smallarray}{ccc}0&0&0\\1&1&1\\4&4&14\\14&16&4\\16&14&16\end{smallarray}\right)& 1& 0& 0\\
\Gamma_{4, 13}& \left(\begin{smallarray}{ccc}0&0&0\\1&1&1\\4&14&16\\14&16&4\\16&4&14\end{smallarray}\right)& \left(\begin{smallarray}{ccc}0&0&0\\1&1&1\\4&14&16\\14&16&4\\16&4&14\end{smallarray}\right)& 1& 0& 0\\
\Gamma_{4, 14}& \left(\begin{smallarray}{ccc}0&0&1\\1&1&0\\4&4&4\\14&16&14\\16&14&16\end{smallarray}\right)& \left(\begin{smallarray}{ccc}0&0&0\\1&1&1\\4&4&18\\14&16&8\\16&14&6\end{smallarray}\right)& 1& 0& 0\\
\Gamma_{4, 15}& \left(\begin{smallarray}{ccc}0&0&1\\1&1&0\\4&4&16\\14&16&4\\16&14&14\end{smallarray}\right)& \left(\begin{smallarray}{ccc}0&0&0\\1&1&1\\4&4&6\\14&16&18\\16&14&8\end{smallarray}\right)& 1& 0& 0\\
\Gamma_{4, 16}& \left(\begin{smallarray}{ccc}0&0&1\\1&1&0\\4&14&14\\14&4&16\\16&16&4\end{smallarray}\right)& \left(\begin{smallarray}{ccc}0&0&0\\1&1&1\\4&14&8\\14&4&6\\16&16&18\end{smallarray}\right)& 1& 0& 0\\
\Gamma_{4, 17}& \left(\begin{smallarray}{ccc}0&0&1\\1&1&0\\4&14&14\\14&16&4\\16&4&16\end{smallarray}\right)& \left(\begin{smallarray}{ccc}0&0&0\\1&1&1\\4&14&8\\14&16&18\\16&4&6\end{smallarray}\right)& 1& 0& 0

%% file: buildings_4.tex
X_{4, 1}& \left(\begin{smallarray}{ccc}0&0&0\\1&1&1\\4&4&4\\14&14&14\\16&16&16\end{smallarray}\right)& \makebox[1.0em][c]{$+$} \makebox[1.0em][c]{$+$} \makebox[1.0em][c]{$+$}& \raisebox{0.5em}{\makebox[1.5em][c]{$65536$}} \raisebox{-0.5em}{\makebox[1.5em][c]{$65536$}} \raisebox{0.5em}{\makebox[1.5em][c]{$65536$}}& \raisebox{-0.5em}{\makebox[1.9em][c]{$196608$}} \raisebox{0.5em}{\makebox[1.9em][c]{$196608$}} \raisebox{-0.5em}{\makebox[1.9em][c]{$196608$}}& \raisebox{0.5em}{\makebox[2.4em][c]{$$}} \raisebox{-0.5em}{\makebox[2.4em][c]{$$}} \raisebox{0.5em}{\makebox[2.4em][c]{$$}}& \raisebox{0.5em}{\makebox[2.4em][c]{$$}} \raisebox{-0.5em}{\makebox[2.4em][c]{$$}} \raisebox{0.5em}{\makebox[2.4em][c]{$$}}& \makebox[1.0em][c]{$6$} \makebox[1.0em][c]{$6$} \makebox[1.0em][c]{$6$}\\
X_{4, 2}& \left(\begin{smallarray}{ccc}0&0&0\\1&1&1\\4&4&16\\14&14&14\\16&16&4\end{smallarray}\right)& \makebox[1.0em][c]{$-$} \makebox[1.0em][c]{$-$} \makebox[1.0em][c]{$+$}& \raisebox{0.5em}{\makebox[1.5em][c]{$0$}} \raisebox{-0.5em}{\makebox[1.5em][c]{$0$}} \raisebox{0.5em}{\makebox[1.5em][c]{$65536$}}& \raisebox{-0.5em}{\makebox[1.9em][c]{$65536$}} \raisebox{0.5em}{\makebox[1.9em][c]{$65536$}} \raisebox{-0.5em}{\makebox[1.9em][c]{$196608$}}& \raisebox{0.5em}{\makebox[2.4em][c]{$1$}} \raisebox{-0.5em}{\makebox[2.4em][c]{$1$}} \raisebox{0.5em}{\makebox[2.4em][c]{$1$}}& \raisebox{0.5em}{\makebox[2.4em][c]{$65536$}} \raisebox{-0.5em}{\makebox[2.4em][c]{$65536$}} \raisebox{0.5em}{\makebox[2.4em][c]{$262144$}}& \makebox[1.0em][c]{$6$} \makebox[1.0em][c]{$6$} \makebox[1.0em][c]{$6$}\\
X_{4, 3}& \left(\begin{smallarray}{ccc}0&0&0\\1&1&1\\4&4&4\\14&14&16\\16&16&14\end{smallarray}\right)& \makebox[1.0em][c]{$-$} \makebox[1.0em][c]{$-$} \makebox[1.0em][c]{$+$}& \raisebox{0.5em}{\makebox[1.5em][c]{$0$}} \raisebox{-0.5em}{\makebox[1.5em][c]{$0$}} \raisebox{0.5em}{\makebox[1.5em][c]{$65536$}}& \raisebox{-0.5em}{\makebox[1.9em][c]{$1$}} \raisebox{0.5em}{\makebox[1.9em][c]{$1$}} \raisebox{-0.5em}{\makebox[1.9em][c]{$196608$}}& \raisebox{0.5em}{\makebox[2.4em][c]{$$}} \raisebox{-0.5em}{\makebox[2.4em][c]{$$}} \raisebox{0.5em}{\makebox[2.4em][c]{$$}}& \raisebox{0.5em}{\makebox[2.4em][c]{$$}} \raisebox{-0.5em}{\makebox[2.4em][c]{$$}} \raisebox{0.5em}{\makebox[2.4em][c]{$$}}& \makebox[1.0em][c]{$1$} \makebox[1.0em][c]{$1$} \makebox[1.0em][c]{$6$}\\
X_{4, 4}& \left(\begin{smallarray}{ccc}0&0&0\\1&1&1\\4&4&14\\14&14&16\\16&16&4\end{smallarray}\right)& \makebox[1.0em][c]{$-$} \makebox[1.0em][c]{$-$} \makebox[1.0em][c]{$+$}& \raisebox{0.5em}{\makebox[1.5em][c]{$0$}} \raisebox{-0.5em}{\makebox[1.5em][c]{$0$}} \raisebox{0.5em}{\makebox[1.5em][c]{$65536$}}& \raisebox{-0.5em}{\makebox[1.9em][c]{$1$}} \raisebox{0.5em}{\makebox[1.9em][c]{$1$}} \raisebox{-0.5em}{\makebox[1.9em][c]{$196608$}}& \raisebox{0.5em}{\makebox[2.4em][c]{$$}} \raisebox{-0.5em}{\makebox[2.4em][c]{$$}} \raisebox{0.5em}{\makebox[2.4em][c]{$$}}& \raisebox{0.5em}{\makebox[2.4em][c]{$$}} \raisebox{-0.5em}{\makebox[2.4em][c]{$$}} \raisebox{0.5em}{\makebox[2.4em][c]{$$}}& \makebox[1.0em][c]{$1$} \makebox[1.0em][c]{$1$} \makebox[1.0em][c]{$6$}\\
X_{4, 5}& \left(\begin{smallarray}{ccc}0&0&0\\1&1&1\\4&4&8\\14&14&6\\16&16&18\end{smallarray}\right)& \makebox[1.0em][c]{$-$} \makebox[1.0em][c]{$-$} \makebox[1.0em][c]{$+$}& \raisebox{0.5em}{\makebox[1.5em][c]{$0$}} \raisebox{-0.5em}{\makebox[1.5em][c]{$0$}} \raisebox{0.5em}{\makebox[1.5em][c]{$65536$}}& \raisebox{-0.5em}{\makebox[1.9em][c]{$1$}} \raisebox{0.5em}{\makebox[1.9em][c]{$1$}} \raisebox{-0.5em}{\makebox[1.9em][c]{$196608$}}& \raisebox{0.5em}{\makebox[2.4em][c]{$$}} \raisebox{-0.5em}{\makebox[2.4em][c]{$$}} \raisebox{0.5em}{\makebox[2.4em][c]{$$}}& \raisebox{0.5em}{\makebox[2.4em][c]{$$}} \raisebox{-0.5em}{\makebox[2.4em][c]{$$}} \raisebox{0.5em}{\makebox[2.4em][c]{$$}}& \makebox[1.0em][c]{$1$} \makebox[1.0em][c]{$1$} \makebox[1.0em][c]{$6$}\\
X_{4, 6}& \left(\begin{smallarray}{ccc}0&0&0\\1&1&1\\4&14&18\\14&16&6\\16&4&8\end{smallarray}\right)& \makebox[1.0em][c]{$-$} \makebox[1.0em][c]{$-$} \makebox[1.0em][c]{$-$}& \raisebox{0.5em}{\makebox[1.5em][c]{$0$}} \raisebox{-0.5em}{\makebox[1.5em][c]{$0$}} \raisebox{0.5em}{\makebox[1.5em][c]{$0$}}& \raisebox{-0.5em}{\makebox[1.9em][c]{$4$}} \raisebox{0.5em}{\makebox[1.9em][c]{$4$}} \raisebox{-0.5em}{\makebox[1.9em][c]{$1$}}& \raisebox{0.5em}{\makebox[2.4em][c]{$1$}} \raisebox{-0.5em}{\makebox[2.4em][c]{$1$}} \raisebox{0.5em}{\makebox[2.4em][c]{$1$}}& \raisebox{0.5em}{\makebox[2.4em][c]{$$}} \raisebox{-0.5em}{\makebox[2.4em][c]{$$}} \raisebox{0.5em}{\makebox[2.4em][c]{$$}}& \makebox[1.0em][c]{$1$} \makebox[1.0em][c]{$1$} \makebox[1.0em][c]{$1$}\\
X_{4, 7}& \left(\begin{smallarray}{ccc}0&0&0\\1&1&1\\4&4&8\\14&16&18\\16&14&6\end{smallarray}\right)& \makebox[1.0em][c]{$-$} \makebox[1.0em][c]{$-$} \makebox[1.0em][c]{$-$}& \raisebox{0.5em}{\makebox[1.5em][c]{$0$}} \raisebox{-0.5em}{\makebox[1.5em][c]{$0$}} \raisebox{0.5em}{\makebox[1.5em][c]{$0$}}& \raisebox{-0.5em}{\makebox[1.9em][c]{$4$}} \raisebox{0.5em}{\makebox[1.9em][c]{$4$}} \raisebox{-0.5em}{\makebox[1.9em][c]{$16$}}& \raisebox{0.5em}{\makebox[2.4em][c]{$1$}} \raisebox{-0.5em}{\makebox[2.4em][c]{$1$}} \raisebox{0.5em}{\makebox[2.4em][c]{$1$}}& \raisebox{0.5em}{\makebox[2.4em][c]{$$}} \raisebox{-0.5em}{\makebox[2.4em][c]{$$}} \raisebox{0.5em}{\makebox[2.4em][c]{$$}}& \makebox[1.0em][c]{$1$} \makebox[1.0em][c]{$1$} \makebox[1.0em][c]{$3$}\\
X_{4, 8}& \left(\begin{smallarray}{ccc}0&0&0\\1&1&6\\4&4&18\\14&16&1\\16&14&8\end{smallarray}\right)& \makebox[1.0em][c]{$-$} \makebox[1.0em][c]{$-$} \makebox[1.0em][c]{$-$}& \raisebox{0.5em}{\makebox[1.5em][c]{$0$}} \raisebox{-0.5em}{\makebox[1.5em][c]{$0$}} \raisebox{0.5em}{\makebox[1.5em][c]{$0$}}& \raisebox{-0.5em}{\makebox[1.9em][c]{$65536$}} \raisebox{0.5em}{\makebox[1.9em][c]{$65536$}} \raisebox{-0.5em}{\makebox[1.9em][c]{$1$}}& \raisebox{0.5em}{\makebox[2.4em][c]{$$}} \raisebox{-0.5em}{\makebox[2.4em][c]{$$}} \raisebox{0.5em}{\makebox[2.4em][c]{$1$}}& \raisebox{0.5em}{\makebox[2.4em][c]{$$}} \raisebox{-0.5em}{\makebox[2.4em][c]{$$}} \raisebox{0.5em}{\makebox[2.4em][c]{$1$}}& \makebox[1.0em][c]{$$} \makebox[1.0em][c]{$$} \makebox[1.0em][c]{$1$}\\
X_{4, 9}& \left(\begin{smallarray}{ccc}0&0&0\\1&1&1\\4&4&16\\14&16&4\\16&14&14\end{smallarray}\right)& \makebox[1.0em][c]{$-$} \makebox[1.0em][c]{$-$} \makebox[1.0em][c]{$-$}& \raisebox{0.5em}{\makebox[1.5em][c]{$0$}} \raisebox{-0.5em}{\makebox[1.5em][c]{$0$}} \raisebox{0.5em}{\makebox[1.5em][c]{$0$}}& \raisebox{-0.5em}{\makebox[1.9em][c]{$1$}} \raisebox{0.5em}{\makebox[1.9em][c]{$1$}} \raisebox{-0.5em}{\makebox[1.9em][c]{$1$}}& \raisebox{0.5em}{\makebox[2.4em][c]{$$}} \raisebox{-0.5em}{\makebox[2.4em][c]{$$}} \raisebox{0.5em}{\makebox[2.4em][c]{$$}}& \raisebox{0.5em}{\makebox[2.4em][c]{$$}} \raisebox{-0.5em}{\makebox[2.4em][c]{$$}} \raisebox{0.5em}{\makebox[2.4em][c]{$$}}& \makebox[1.0em][c]{$1$} \makebox[1.0em][c]{$1$} \makebox[1.0em][c]{$1$}\\
X_{4, 10}& \left(\begin{smallarray}{ccc}0&0&0\\1&1&1\\4&16&8\\14&14&6\\16&4&18\end{smallarray}\right)& \makebox[1.0em][c]{$-$} \makebox[1.0em][c]{$-$} \makebox[1.0em][c]{$-$}& \raisebox{0.5em}{\makebox[1.5em][c]{$0$}} \raisebox{-0.5em}{\makebox[1.5em][c]{$0$}} \raisebox{0.5em}{\makebox[1.5em][c]{$0$}}& \raisebox{-0.5em}{\makebox[1.9em][c]{$1$}} \raisebox{0.5em}{\makebox[1.9em][c]{$1$}} \raisebox{-0.5em}{\makebox[1.9em][c]{$1$}}& \raisebox{0.5em}{\makebox[2.4em][c]{$$}} \raisebox{-0.5em}{\makebox[2.4em][c]{$$}} \raisebox{0.5em}{\makebox[2.4em][c]{$$}}& \raisebox{0.5em}{\makebox[2.4em][c]{$$}} \raisebox{-0.5em}{\makebox[2.4em][c]{$$}} \raisebox{0.5em}{\makebox[2.4em][c]{$$}}& \makebox[1.0em][c]{$1$} \makebox[1.0em][c]{$1$} \makebox[1.0em][c]{$1$}\\
X_{4, 11}& \left(\begin{smallarray}{ccc}0&0&0\\1&1&6\\4&14&1\\14&4&8\\16&16&18\end{smallarray}\right)& \makebox[1.0em][c]{$-$} \makebox[1.0em][c]{$-$} \makebox[1.0em][c]{$-$}& \raisebox{0.5em}{\makebox[1.5em][c]{$0$}} \raisebox{-0.5em}{\makebox[1.5em][c]{$0$}} \raisebox{0.5em}{\makebox[1.5em][c]{$0$}}& \raisebox{-0.5em}{\makebox[1.9em][c]{$1$}} \raisebox{0.5em}{\makebox[1.9em][c]{$1$}} \raisebox{-0.5em}{\makebox[1.9em][c]{$1$}}& \raisebox{0.5em}{\makebox[2.4em][c]{$$}} \raisebox{-0.5em}{\makebox[2.4em][c]{$$}} \raisebox{0.5em}{\makebox[2.4em][c]{$$}}& \raisebox{0.5em}{\makebox[2.4em][c]{$$}} \raisebox{-0.5em}{\makebox[2.4em][c]{$$}} \raisebox{0.5em}{\makebox[2.4em][c]{$$}}& \makebox[1.0em][c]{$1$} \makebox[1.0em][c]{$1$} \makebox[1.0em][c]{$1$}\\
X_{4, 12}& \left(\begin{smallarray}{ccc}0&0&0\\1&1&1\\4&4&14\\14&16&4\\16&14&16\end{smallarray}\right)& \makebox[1.0em][c]{$-$} \makebox[1.0em][c]{$-$} \makebox[1.0em][c]{$-$}& \raisebox{0.5em}{\makebox[1.5em][c]{$0$}} \raisebox{-0.5em}{\makebox[1.5em][c]{$0$}} \raisebox{0.5em}{\makebox[1.5em][c]{$0$}}& \raisebox{-0.5em}{\makebox[1.9em][c]{$1$}} \raisebox{0.5em}{\makebox[1.9em][c]{$1$}} \raisebox{-0.5em}{\makebox[1.9em][c]{$1$}}& \raisebox{0.5em}{\makebox[2.4em][c]{$$}} \raisebox{-0.5em}{\makebox[2.4em][c]{$$}} \raisebox{0.5em}{\makebox[2.4em][c]{$$}}& \raisebox{0.5em}{\makebox[2.4em][c]{$$}} \raisebox{-0.5em}{\makebox[2.4em][c]{$$}} \raisebox{0.5em}{\makebox[2.4em][c]{$$}}& \makebox[1.0em][c]{$1$} \makebox[1.0em][c]{$1$} \makebox[1.0em][c]{$1$}\\
X_{4, 13}& \left(\begin{smallarray}{ccc}0&0&0\\1&1&1\\4&14&16\\14&16&4\\16&4&14\end{smallarray}\right)& \makebox[1.0em][c]{$-$} \makebox[1.0em][c]{$-$} \makebox[1.0em][c]{$-$}& \raisebox{0.5em}{\makebox[1.5em][c]{$0$}} \raisebox{-0.5em}{\makebox[1.5em][c]{$0$}} \raisebox{0.5em}{\makebox[1.5em][c]{$0$}}& \raisebox{-0.5em}{\makebox[1.9em][c]{$1$}} \raisebox{0.5em}{\makebox[1.9em][c]{$1$}} \raisebox{-0.5em}{\makebox[1.9em][c]{$1$}}& \raisebox{0.5em}{\makebox[2.4em][c]{$$}} \raisebox{-0.5em}{\makebox[2.4em][c]{$$}} \raisebox{0.5em}{\makebox[2.4em][c]{$$}}& \raisebox{0.5em}{\makebox[2.4em][c]{$$}} \raisebox{-0.5em}{\makebox[2.4em][c]{$$}} \raisebox{0.5em}{\makebox[2.4em][c]{$$}}& \makebox[1.0em][c]{$1$} \makebox[1.0em][c]{$1$} \makebox[1.0em][c]{$1$}\\
X_{4, 14}& \left(\begin{smallarray}{ccc}0&0&0\\1&1&1\\4&4&18\\14&16&8\\16&14&6\end{smallarray}\right)& \makebox[1.0em][c]{$-$} \makebox[1.0em][c]{$-$} \makebox[1.0em][c]{$-$}& \raisebox{0.5em}{\makebox[1.5em][c]{$0$}} \raisebox{-0.5em}{\makebox[1.5em][c]{$0$}} \raisebox{0.5em}{\makebox[1.5em][c]{$0$}}& \raisebox{-0.5em}{\makebox[1.9em][c]{$1$}} \raisebox{0.5em}{\makebox[1.9em][c]{$1$}} \raisebox{-0.5em}{\makebox[1.9em][c]{$1$}}& \raisebox{0.5em}{\makebox[2.4em][c]{$$}} \raisebox{-0.5em}{\makebox[2.4em][c]{$$}} \raisebox{0.5em}{\makebox[2.4em][c]{$$}}& \raisebox{0.5em}{\makebox[2.4em][c]{$$}} \raisebox{-0.5em}{\makebox[2.4em][c]{$$}} \raisebox{0.5em}{\makebox[2.4em][c]{$$}}& \makebox[1.0em][c]{$1$} \makebox[1.0em][c]{$1$} \makebox[1.0em][c]{$1$}\\
X_{4, 15}& \left(\begin{smallarray}{ccc}0&0&0\\1&1&1\\4&4&6\\14&16&18\\16&14&8\end{smallarray}\right)& \makebox[1.0em][c]{$-$} \makebox[1.0em][c]{$-$} \makebox[1.0em][c]{$-$}& \raisebox{0.5em}{\makebox[1.5em][c]{$0$}} \raisebox{-0.5em}{\makebox[1.5em][c]{$0$}} \raisebox{0.5em}{\makebox[1.5em][c]{$0$}}& \raisebox{-0.5em}{\makebox[1.9em][c]{$1$}} \raisebox{0.5em}{\makebox[1.9em][c]{$1$}} \raisebox{-0.5em}{\makebox[1.9em][c]{$1$}}& \raisebox{0.5em}{\makebox[2.4em][c]{$$}} \raisebox{-0.5em}{\makebox[2.4em][c]{$$}} \raisebox{0.5em}{\makebox[2.4em][c]{$$}}& \raisebox{0.5em}{\makebox[2.4em][c]{$$}} \raisebox{-0.5em}{\makebox[2.4em][c]{$$}} \raisebox{0.5em}{\makebox[2.4em][c]{$$}}& \makebox[1.0em][c]{$1$} \makebox[1.0em][c]{$1$} \makebox[1.0em][c]{$1$}\\
X_{4, 16}& \left(\begin{smallarray}{ccc}0&0&0\\1&1&1\\4&14&8\\14&4&6\\16&16&18\end{smallarray}\right)& \makebox[1.0em][c]{$-$} \makebox[1.0em][c]{$-$} \makebox[1.0em][c]{$-$}& \raisebox{0.5em}{\makebox[1.5em][c]{$0$}} \raisebox{-0.5em}{\makebox[1.5em][c]{$0$}} \raisebox{0.5em}{\makebox[1.5em][c]{$0$}}& \raisebox{-0.5em}{\makebox[1.9em][c]{$1$}} \raisebox{0.5em}{\makebox[1.9em][c]{$1$}} \raisebox{-0.5em}{\makebox[1.9em][c]{$1$}}& \raisebox{0.5em}{\makebox[2.4em][c]{$$}} \raisebox{-0.5em}{\makebox[2.4em][c]{$$}} \raisebox{0.5em}{\makebox[2.4em][c]{$$}}& \raisebox{0.5em}{\makebox[2.4em][c]{$$}} \raisebox{-0.5em}{\makebox[2.4em][c]{$$}} \raisebox{0.5em}{\makebox[2.4em][c]{$$}}& \makebox[1.0em][c]{$1$} \makebox[1.0em][c]{$1$} \makebox[1.0em][c]{$1$}\\
X_{4, 17}& \left(\begin{smallarray}{ccc}0&0&0\\1&1&1\\4&14&8\\14&16&18\\16&4&6\end{smallarray}\right)& \makebox[1.0em][c]{$-$} \makebox[1.0em][c]{$-$} \makebox[1.0em][c]{$-$}& \raisebox{0.5em}{\makebox[1.5em][c]{$0$}} \raisebox{-0.5em}{\makebox[1.5em][c]{$0$}} \raisebox{0.5em}{\makebox[1.5em][c]{$0$}}& \raisebox{-0.5em}{\makebox[1.9em][c]{$1$}} \raisebox{0.5em}{\makebox[1.9em][c]{$1$}} \raisebox{-0.5em}{\makebox[1.9em][c]{$1$}}& \raisebox{0.5em}{\makebox[2.4em][c]{$$}} \raisebox{-0.5em}{\makebox[2.4em][c]{$$}} \raisebox{0.5em}{\makebox[2.4em][c]{$$}}& \raisebox{0.5em}{\makebox[2.4em][c]{$$}} \raisebox{-0.5em}{\makebox[2.4em][c]{$$}} \raisebox{0.5em}{\makebox[2.4em][c]{$$}}& \makebox[1.0em][c]{$1$} \makebox[1.0em][c]{$1$} \makebox[1.0em][c]{$1$}

%% file: lattices_5.tex
1& \left(\begin{smallarray}{ccc}0&0&0\\1&1&1\\3&3&3\\8&8&8\\12&12&12\\18&18&18\end{smallarray}\right)& C_3 \rtimes S_3& (\Z/31\Z)^{2}& (\Z/5\Z)^{6}\\
1& \left(\begin{smallarray}{ccc}0&0&0\\1&1&18\\3&3&1\\8&8&8\\12&12&3\\18&18&12\end{smallarray}\right)& C_3 \rtimes (0, 1)& \Z/31\Z& 0\\
1& \left(\begin{smallarray}{ccc}0&0&0\\1&1&1\\3&3&12\\8&8&18\\12&12&3\\18&18&8\end{smallarray}\right)& C_3 \rtimes (0, 1)& \Z/31\Z& 0\\
1& \left(\begin{smallarray}{ccc}0&0&0\\1&1&12\\3&3&1\\8&8&8\\12&12&3\\18&18&18\end{smallarray}\right)& (0, 1)& \Z/31\Z& 0\\
4& \left(\begin{smallarray}{ccc}0&0&0\\1&1&1\\3&3&3\\8&8&12\\12&12&8\\18&18&18\end{smallarray}\right)& (0, 1)& \Z/31\Z& 0\\
10& \left(\begin{smallarray}{ccc}0&0&0\\1&1&1\\3&3&3\\8&8&8\\12&12&18\\18&18&12\end{smallarray}\right)& (0, 1)& \Z/31\Z& 0\\
5& \left(\begin{smallarray}{ccc}0&0&0\\1&1&1\\3&3&12\\8&8&18\\12&12&8\\18&18&3\end{smallarray}\right)& 1& \Z/31\Z& 0\\
3& \left(\begin{smallarray}{ccc}0&0&0\\1&1&3\\3&3&1\\8&8&18\\12&12&12\\18&18&8\end{smallarray}\right)& 1& \Z/31\Z& 0\\
11& \left(\begin{smallarray}{ccc}0&0&0\\1&1&1\\3&3&12\\8&8&8\\12&12&18\\18&18&3\end{smallarray}\right)& 1& \Z/31\Z& 0\\
55& \left(\begin{smallarray}{ccc}0&0&0\\1&1&1\\3&3&3\\8&8&12\\12&12&18\\18&18&8\end{smallarray}\right)& 1& \Z/31\Z& 0\\
2& \left(\begin{smallarray}{ccc}0&0&0\\1&1&18\\3&8&1\\8&12&12\\12&18&3\\18&3&8\end{smallarray}\right)& C_3 \rtimes (0, 1)& 0& 0\\
1& \left(\begin{smallarray}{ccc}0&0&0\\1&1&18\\3&12&12\\8&8&8\\12&18&3\\18&3&1\end{smallarray}\right)& C_3 \rtimes (0, 1)& 0& 0\\
6& \left(\begin{smallarray}{ccc}0&0&0\\1&1&3\\3&12&12\\8&18&1\\12&8&18\\18&3&8\end{smallarray}\right)& C_3 \rtimes (0, 1)& 0& 0\\
11& \left(\begin{smallarray}{ccc}0&0&0\\1&1&1\\3&3&3\\8&8&12\\12&18&8\\18&12&18\end{smallarray}\right)& (0, 1)& 0& 0\\
1& \left(\begin{smallarray}{ccc}0&0&0\\1&1&18\\3&8&1\\8&18&3\\12&3&12\\18&12&8\end{smallarray}\right)& (0, 1)& 0& 0\\
2& \left(\begin{smallarray}{ccc}0&0&0\\1&1&8\\3&18&1\\8&12&18\\12&8&12\\18&3&3\end{smallarray}\right)& (0, 1)& 0& 0\\
5& \left(\begin{smallarray}{ccc}0&0&0\\1&1&1\\3&8&18\\8&18&3\\12&3&12\\18&12&8\end{smallarray}\right)& (0, 1)& 0& 0\\
18& \left(\begin{smallarray}{ccc}0&0&0\\1&1&1\\3&3&8\\8&18&12\\12&12&3\\18&8&18\end{smallarray}\right)& (0, 1)& 0& 0\\
24& \left(\begin{smallarray}{ccc}0&0&0\\1&1&1\\3&8&8\\8&12&12\\12&3&18\\18&18&3\end{smallarray}\right)& (0, 1)& 0& 0\\
1& \left(\begin{smallarray}{ccc}0&0&0\\1&1&12\\3&18&3\\8&12&1\\12&8&8\\18&3&18\end{smallarray}\right)& C_3& 0& 0\\
8& \left(\begin{smallarray}{ccc}0&0&0\\1&1&3\\3&8&12\\8&12&8\\12&3&18\\18&18&1\end{smallarray}\right)& C_3& 0& 0\\
10& \left(\begin{smallarray}{ccc}0&0&0\\1&1&3\\3&3&8\\8&8&12\\12&18&1\\18&12&18\end{smallarray}\right)& C_3& 0& 0\\
2& \left(\begin{smallarray}{ccc}0&0&0\\1&1&3\\3&12&1\\8&18&8\\12&3&18\\18&8&12\end{smallarray}\right)& 1& 0& 0\\
2& \left(\begin{smallarray}{ccc}0&0&0\\1&1&8\\3&12&3\\8&18&18\\12&8&12\\18&3&1\end{smallarray}\right)& 1& 0& 0\\
8& \left(\begin{smallarray}{ccc}0&0&0\\1&1&8\\3&18&1\\8&8&18\\12&12&3\\18&3&12\end{smallarray}\right)& 1& 0& 0\\
59& \left(\begin{smallarray}{ccc}0&0&0\\1&1&3\\3&3&8\\8&18&1\\12&12&18\\18&8&12\end{smallarray}\right)& 1& 0& 0\\
1& \left(\begin{smallarray}{ccc}0&0&0\\1&1&3\\3&3&12\\8&12&8\\12&18&18\\18&8&1\end{smallarray}\right)& 1& 0& 0\\
19& \left(\begin{smallarray}{ccc}0&0&0\\1&1&3\\3&3&12\\8&8&18\\12&18&1\\18&12&8\end{smallarray}\right)& 1& 0& 0\\
93& \left(\begin{smallarray}{ccc}0&0&0\\1&1&1\\3&8&12\\8&12&3\\12&3&8\\18&18&18\end{smallarray}\right)& 1& 0& 0\\
4& \left(\begin{smallarray}{ccc}0&0&0\\1&1&18\\3&12&1\\8&18&12\\12&3&8\\18&8&3\end{smallarray}\right)& 1& 0& 0\\
60& \left(\begin{smallarray}{ccc}0&0&0\\1&1&1\\3&12&12\\8&8&8\\12&3&18\\18&18&3\end{smallarray}\right)& 1& 0& 0\\
661& \left(\begin{smallarray}{ccc}0&0&0\\1&1&1\\3&3&8\\8&12&3\\12&8&18\\18&18&12\end{smallarray}\right)& 1& 0& 0\\
2179& \left(\begin{smallarray}{ccc}0&0&0\\1&1&1\\3&3&3\\8&8&18\\12&18&8\\18&12&12\end{smallarray}\right)& 1& 0& 0

%% file: buildings_5.tex
1& \left(\begin{smallarray}{ccc}0&0&0\\1&1&1\\3&3&3\\8&8&8\\12&12&12\\18&18&18\end{smallarray}\right)& \makebox[1.0em][c]{$+$} \makebox[1.0em][c]{$+$} \makebox[1.0em][c]{$+$}& \raisebox{0.5em}{\makebox[1.5em][c]{$390625$}} \raisebox{-0.5em}{\makebox[1.5em][c]{$390625$}} \raisebox{0.5em}{\makebox[1.5em][c]{$390625$}}& \raisebox{-0.5em}{\makebox[1.9em][c]{$1562500$}} \raisebox{0.5em}{\makebox[1.9em][c]{$1562500$}} \raisebox{-0.5em}{\makebox[1.9em][c]{$1562500$}}& \makebox[1.0em][c]{$3$} \makebox[1.0em][c]{$3$} \makebox[1.0em][c]{$3$}\\
1& \left(\begin{smallarray}{ccc}0&0&0\\1&1&18\\3&3&1\\8&8&8\\12&12&3\\18&18&12\end{smallarray}\right)& \makebox[1.0em][c]{$-$} \makebox[1.0em][c]{$-$} \makebox[1.0em][c]{$+$}& \raisebox{0.5em}{\makebox[1.5em][c]{$0$}} \raisebox{-0.5em}{\makebox[1.5em][c]{$0$}} \raisebox{0.5em}{\makebox[1.5em][c]{$390625$}}& \raisebox{-0.5em}{\makebox[1.9em][c]{$1$}} \raisebox{0.5em}{\makebox[1.9em][c]{$1$}} \raisebox{-0.5em}{\makebox[1.9em][c]{$1562500$}}& \makebox[1.0em][c]{$3$} \makebox[1.0em][c]{$3$} \makebox[1.0em][c]{$3$}\\
1& \left(\begin{smallarray}{ccc}0&0&0\\1&1&1\\3&3&12\\8&8&18\\12&12&3\\18&18&8\end{smallarray}\right)& \makebox[1.0em][c]{$-$} \makebox[1.0em][c]{$-$} \makebox[1.0em][c]{$-$}& \raisebox{0.5em}{\makebox[1.5em][c]{$3$}} \raisebox{-0.5em}{\makebox[1.5em][c]{$3$}} \raisebox{0.5em}{\makebox[1.5em][c]{$25$}}& \raisebox{-0.5em}{\makebox[1.9em][c]{$1$}} \raisebox{0.5em}{\makebox[1.9em][c]{$1$}} \raisebox{-0.5em}{\makebox[1.9em][c]{$100$}}& \makebox[1.0em][c]{$3$} \makebox[1.0em][c]{$3$} \makebox[1.0em][c]{$3$}\\
1& \left(\begin{smallarray}{ccc}0&0&0\\1&1&12\\3&3&1\\8&8&8\\12&12&3\\18&18&18\end{smallarray}\right)& \makebox[1.0em][c]{$-$} \makebox[1.0em][c]{$-$} \makebox[1.0em][c]{$+$}& \raisebox{0.5em}{\makebox[1.5em][c]{$0$}} \raisebox{-0.5em}{\makebox[1.5em][c]{$0$}} \raisebox{0.5em}{\makebox[1.5em][c]{$390625$}}& \raisebox{-0.5em}{\makebox[1.9em][c]{$1$}} \raisebox{0.5em}{\makebox[1.9em][c]{$1$}} \raisebox{-0.5em}{\makebox[1.9em][c]{$1562500$}}& \makebox[1.0em][c]{$1$} \makebox[1.0em][c]{$1$} \makebox[1.0em][c]{$3$}\\
4& \left(\begin{smallarray}{ccc}0&0&0\\1&1&1\\3&3&3\\8&8&12\\12&12&8\\18&18&18\end{smallarray}\right)& \makebox[1.0em][c]{$-$} \makebox[1.0em][c]{$-$} \makebox[1.0em][c]{$-$}& \raisebox{0.5em}{\makebox[1.5em][c]{$0$}} \raisebox{-0.5em}{\makebox[1.5em][c]{$0$}} \raisebox{0.5em}{\makebox[1.5em][c]{$25$}}& \raisebox{-0.5em}{\makebox[1.9em][c]{$1$}} \raisebox{0.5em}{\makebox[1.9em][c]{$1$}} \raisebox{-0.5em}{\makebox[1.9em][c]{$100$}}& \makebox[1.0em][c]{$1$} \makebox[1.0em][c]{$1$} \makebox[1.0em][c]{$3$}\\
10& \left(\begin{smallarray}{ccc}0&0&0\\1&1&1\\3&3&3\\8&8&8\\12&12&18\\18&18&12\end{smallarray}\right)& \makebox[1.0em][c]{$-$} \makebox[1.0em][c]{$-$} \makebox[1.0em][c]{$-$}& \raisebox{0.5em}{\makebox[1.5em][c]{$0$}} \raisebox{-0.5em}{\makebox[1.5em][c]{$0$}} \raisebox{0.5em}{\makebox[1.5em][c]{$25$}}& \raisebox{-0.5em}{\makebox[1.9em][c]{$1$}} \raisebox{0.5em}{\makebox[1.9em][c]{$1$}} \raisebox{-0.5em}{\makebox[1.9em][c]{$100$}}& \makebox[1.0em][c]{$1$} \makebox[1.0em][c]{$1$} \makebox[1.0em][c]{$1$}\\
5& \left(\begin{smallarray}{ccc}0&0&0\\1&1&1\\3&3&12\\8&8&18\\12&12&8\\18&18&3\end{smallarray}\right)& \makebox[1.0em][c]{$-$} \makebox[1.0em][c]{$-$} \makebox[1.0em][c]{$-$}& \raisebox{0.5em}{\makebox[1.5em][c]{$0$}} \raisebox{-0.5em}{\makebox[1.5em][c]{$0$}} \raisebox{0.5em}{\makebox[1.5em][c]{$3$}}& \raisebox{-0.5em}{\makebox[1.9em][c]{$1$}} \raisebox{0.5em}{\makebox[1.9em][c]{$1$}} \raisebox{-0.5em}{\makebox[1.9em][c]{$1$}}& \makebox[1.0em][c]{$1$} \makebox[1.0em][c]{$1$} \makebox[1.0em][c]{$1$}\\
3& \left(\begin{smallarray}{ccc}0&0&0\\1&1&3\\3&3&1\\8&8&18\\12&12&12\\18&18&8\end{smallarray}\right)& \makebox[1.0em][c]{$-$} \makebox[1.0em][c]{$-$} \makebox[1.0em][c]{$-$}& \raisebox{0.5em}{\makebox[1.5em][c]{$0$}} \raisebox{-0.5em}{\makebox[1.5em][c]{$0$}} \raisebox{0.5em}{\makebox[1.5em][c]{$2$}}& \raisebox{-0.5em}{\makebox[1.9em][c]{$1$}} \raisebox{0.5em}{\makebox[1.9em][c]{$1$}} \raisebox{-0.5em}{\makebox[1.9em][c]{$1$}}& \makebox[1.0em][c]{$1$} \makebox[1.0em][c]{$1$} \makebox[1.0em][c]{$1$}\\
11& \left(\begin{smallarray}{ccc}0&0&0\\1&1&1\\3&3&12\\8&8&8\\12&12&18\\18&18&3\end{smallarray}\right)& \makebox[1.0em][c]{$-$} \makebox[1.0em][c]{$-$} \makebox[1.0em][c]{$-$}& \raisebox{0.5em}{\makebox[1.5em][c]{$0$}} \raisebox{-0.5em}{\makebox[1.5em][c]{$0$}} \raisebox{0.5em}{\makebox[1.5em][c]{$1$}}& \raisebox{-0.5em}{\makebox[1.9em][c]{$1$}} \raisebox{0.5em}{\makebox[1.9em][c]{$1$}} \raisebox{-0.5em}{\makebox[1.9em][c]{$1$}}& \makebox[1.0em][c]{$1$} \makebox[1.0em][c]{$1$} \makebox[1.0em][c]{$1$}\\
55& \left(\begin{smallarray}{ccc}0&0&0\\1&1&1\\3&3&3\\8&8&12\\12&12&18\\18&18&8\end{smallarray}\right)& \makebox[1.0em][c]{$-$} \makebox[1.0em][c]{$-$} \makebox[1.0em][c]{$-$}& \raisebox{0.5em}{\makebox[1.5em][c]{$0$}} \raisebox{-0.5em}{\makebox[1.5em][c]{$0$}} \raisebox{0.5em}{\makebox[1.5em][c]{$0$}}& \raisebox{-0.5em}{\makebox[1.9em][c]{$1$}} \raisebox{0.5em}{\makebox[1.9em][c]{$1$}} \raisebox{-0.5em}{\makebox[1.9em][c]{$1$}}& \makebox[1.0em][c]{$1$} \makebox[1.0em][c]{$1$} \makebox[1.0em][c]{$1$}\\
2& \left(\begin{smallarray}{ccc}0&0&0\\1&1&18\\3&8&1\\8&12&12\\12&18&3\\18&3&8\end{smallarray}\right)& \makebox[1.0em][c]{$-$} \makebox[1.0em][c]{$-$} \makebox[1.0em][c]{$+$}& \raisebox{0.5em}{\makebox[1.5em][c]{$0$}} \raisebox{-0.5em}{\makebox[1.5em][c]{$0$}} \raisebox{0.5em}{\makebox[1.5em][c]{$390625$}}& \raisebox{-0.5em}{\makebox[1.9em][c]{$1$}} \raisebox{0.5em}{\makebox[1.9em][c]{$1$}} \raisebox{-0.5em}{\makebox[1.9em][c]{$1562500$}}& \makebox[1.0em][c]{$3$} \makebox[1.0em][c]{$3$} \makebox[1.0em][c]{$3$}\\
1& \left(\begin{smallarray}{ccc}0&0&0\\1&1&18\\3&12&12\\8&8&8\\12&18&3\\18&3&1\end{smallarray}\right)& \makebox[1.0em][c]{$-$} \makebox[1.0em][c]{$-$} \makebox[1.0em][c]{$-$}& \raisebox{0.5em}{\makebox[1.5em][c]{$3$}} \raisebox{-0.5em}{\makebox[1.5em][c]{$3$}} \raisebox{0.5em}{\makebox[1.5em][c]{$25$}}& \raisebox{-0.5em}{\makebox[1.9em][c]{$1$}} \raisebox{0.5em}{\makebox[1.9em][c]{$1$}} \raisebox{-0.5em}{\makebox[1.9em][c]{$100$}}& \makebox[1.0em][c]{$3$} \makebox[1.0em][c]{$3$} \makebox[1.0em][c]{$3$}\\
6& \left(\begin{smallarray}{ccc}0&0&0\\1&1&3\\3&12&12\\8&18&1\\12&8&18\\18&3&8\end{smallarray}\right)& \makebox[1.0em][c]{$-$} \makebox[1.0em][c]{$-$} \makebox[1.0em][c]{$-$}& \raisebox{0.5em}{\makebox[1.5em][c]{$0$}} \raisebox{-0.5em}{\makebox[1.5em][c]{$0$}} \raisebox{0.5em}{\makebox[1.5em][c]{$25$}}& \raisebox{-0.5em}{\makebox[1.9em][c]{$1$}} \raisebox{0.5em}{\makebox[1.9em][c]{$1$}} \raisebox{-0.5em}{\makebox[1.9em][c]{$100$}}& \makebox[1.0em][c]{$3$} \makebox[1.0em][c]{$3$} \makebox[1.0em][c]{$3$}\\
11& \left(\begin{smallarray}{ccc}0&0&0\\1&1&1\\3&3&3\\8&8&12\\12&18&8\\18&12&18\end{smallarray}\right)& \makebox[1.0em][c]{$-$} \makebox[1.0em][c]{$-$} \makebox[1.0em][c]{$+$}& \raisebox{0.5em}{\makebox[1.5em][c]{$0$}} \raisebox{-0.5em}{\makebox[1.5em][c]{$0$}} \raisebox{0.5em}{\makebox[1.5em][c]{$390625$}}& \raisebox{-0.5em}{\makebox[1.9em][c]{$1$}} \raisebox{0.5em}{\makebox[1.9em][c]{$1$}} \raisebox{-0.5em}{\makebox[1.9em][c]{$1562500$}}& \makebox[1.0em][c]{$1$} \makebox[1.0em][c]{$1$} \makebox[1.0em][c]{$3$}\\
1& \left(\begin{smallarray}{ccc}0&0&0\\1&1&18\\3&8&1\\8&18&3\\12&3&12\\18&12&8\end{smallarray}\right)& \makebox[1.0em][c]{$-$} \makebox[1.0em][c]{$-$} \makebox[1.0em][c]{$-$}& \raisebox{0.5em}{\makebox[1.5em][c]{$2$}} \raisebox{-0.5em}{\makebox[1.5em][c]{$2$}} \raisebox{0.5em}{\makebox[1.5em][c]{$25$}}& \raisebox{-0.5em}{\makebox[1.9em][c]{$1$}} \raisebox{0.5em}{\makebox[1.9em][c]{$1$}} \raisebox{-0.5em}{\makebox[1.9em][c]{$100$}}& \makebox[1.0em][c]{$1$} \makebox[1.0em][c]{$1$} \makebox[1.0em][c]{$1$}\\
2& \left(\begin{smallarray}{ccc}0&0&0\\1&1&8\\3&18&1\\8&12&18\\12&8&12\\18&3&3\end{smallarray}\right)& \makebox[1.0em][c]{$-$} \makebox[1.0em][c]{$-$} \makebox[1.0em][c]{$-$}& \raisebox{0.5em}{\makebox[1.5em][c]{$1$}} \raisebox{-0.5em}{\makebox[1.5em][c]{$1$}} \raisebox{0.5em}{\makebox[1.5em][c]{$25$}}& \raisebox{-0.5em}{\makebox[1.9em][c]{$1$}} \raisebox{0.5em}{\makebox[1.9em][c]{$1$}} \raisebox{-0.5em}{\makebox[1.9em][c]{$100$}}& \makebox[1.0em][c]{$1$} \makebox[1.0em][c]{$1$} \makebox[1.0em][c]{$3$}\\
5& \left(\begin{smallarray}{ccc}0&0&0\\1&1&1\\3&8&18\\8&18&3\\12&3&12\\18&12&8\end{smallarray}\right)& \makebox[1.0em][c]{$-$} \makebox[1.0em][c]{$-$} \makebox[1.0em][c]{$-$}& \raisebox{0.5em}{\makebox[1.5em][c]{$1$}} \raisebox{-0.5em}{\makebox[1.5em][c]{$1$}} \raisebox{0.5em}{\makebox[1.5em][c]{$25$}}& \raisebox{-0.5em}{\makebox[1.9em][c]{$1$}} \raisebox{0.5em}{\makebox[1.9em][c]{$1$}} \raisebox{-0.5em}{\makebox[1.9em][c]{$100$}}& \makebox[1.0em][c]{$1$} \makebox[1.0em][c]{$1$} \makebox[1.0em][c]{$1$}\\
18& \left(\begin{smallarray}{ccc}0&0&0\\1&1&1\\3&3&8\\8&18&12\\12&12&3\\18&8&18\end{smallarray}\right)& \makebox[1.0em][c]{$-$} \makebox[1.0em][c]{$-$} \makebox[1.0em][c]{$-$}& \raisebox{0.5em}{\makebox[1.5em][c]{$0$}} \raisebox{-0.5em}{\makebox[1.5em][c]{$0$}} \raisebox{0.5em}{\makebox[1.5em][c]{$25$}}& \raisebox{-0.5em}{\makebox[1.9em][c]{$1$}} \raisebox{0.5em}{\makebox[1.9em][c]{$1$}} \raisebox{-0.5em}{\makebox[1.9em][c]{$100$}}& \makebox[1.0em][c]{$1$} \makebox[1.0em][c]{$1$} \makebox[1.0em][c]{$3$}\\
24& \left(\begin{smallarray}{ccc}0&0&0\\1&1&1\\3&8&8\\8&12&12\\12&3&18\\18&18&3\end{smallarray}\right)& \makebox[1.0em][c]{$-$} \makebox[1.0em][c]{$-$} \makebox[1.0em][c]{$-$}& \raisebox{0.5em}{\makebox[1.5em][c]{$0$}} \raisebox{-0.5em}{\makebox[1.5em][c]{$0$}} \raisebox{0.5em}{\makebox[1.5em][c]{$25$}}& \raisebox{-0.5em}{\makebox[1.9em][c]{$1$}} \raisebox{0.5em}{\makebox[1.9em][c]{$1$}} \raisebox{-0.5em}{\makebox[1.9em][c]{$100$}}& \makebox[1.0em][c]{$1$} \makebox[1.0em][c]{$1$} \makebox[1.0em][c]{$1$}\\
1& \left(\begin{smallarray}{ccc}0&0&0\\1&1&12\\3&18&3\\8&12&1\\12&8&8\\18&3&18\end{smallarray}\right)& \makebox[1.0em][c]{$-$} \makebox[1.0em][c]{$-$} \makebox[1.0em][c]{$-$}& \raisebox{0.5em}{\makebox[1.5em][c]{$0$}} \raisebox{-0.5em}{\makebox[1.5em][c]{$3$}} \raisebox{0.5em}{\makebox[1.5em][c]{$3$}}& \raisebox{-0.5em}{\makebox[1.9em][c]{$1$}} \raisebox{0.5em}{\makebox[1.9em][c]{$1$}} \raisebox{-0.5em}{\makebox[1.9em][c]{$1$}}& \makebox[1.0em][c]{$3$} \makebox[1.0em][c]{$3$} \makebox[1.0em][c]{$3$}\\
8& \left(\begin{smallarray}{ccc}0&0&0\\1&1&3\\3&8&12\\8&12&8\\12&3&18\\18&18&1\end{smallarray}\right)& \makebox[1.0em][c]{$-$} \makebox[1.0em][c]{$-$} \makebox[1.0em][c]{$-$}& \raisebox{0.5em}{\makebox[1.5em][c]{$0$}} \raisebox{-0.5em}{\makebox[1.5em][c]{$0$}} \raisebox{0.5em}{\makebox[1.5em][c]{$3$}}& \raisebox{-0.5em}{\makebox[1.9em][c]{$1$}} \raisebox{0.5em}{\makebox[1.9em][c]{$1$}} \raisebox{-0.5em}{\makebox[1.9em][c]{$1$}}& \makebox[1.0em][c]{$3$} \makebox[1.0em][c]{$3$} \makebox[1.0em][c]{$3$}\\
10& \left(\begin{smallarray}{ccc}0&0&0\\1&1&3\\3&3&8\\8&8&12\\12&18&1\\18&12&18\end{smallarray}\right)& \makebox[1.0em][c]{$-$} \makebox[1.0em][c]{$-$} \makebox[1.0em][c]{$-$}& \raisebox{0.5em}{\makebox[1.5em][c]{$0$}} \raisebox{-0.5em}{\makebox[1.5em][c]{$0$}} \raisebox{0.5em}{\makebox[1.5em][c]{$0$}}& \raisebox{-0.5em}{\makebox[1.9em][c]{$1$}} \raisebox{0.5em}{\makebox[1.9em][c]{$1$}} \raisebox{-0.5em}{\makebox[1.9em][c]{$1$}}& \makebox[1.0em][c]{$3$} \makebox[1.0em][c]{$3$} \makebox[1.0em][c]{$3$}\\
2& \left(\begin{smallarray}{ccc}0&0&0\\1&1&3\\3&12&1\\8&18&8\\12&3&18\\18&8&12\end{smallarray}\right)& \makebox[1.0em][c]{$-$} \makebox[1.0em][c]{$-$} \makebox[1.0em][c]{$-$}& \raisebox{0.5em}{\makebox[1.5em][c]{$1$}} \raisebox{-0.5em}{\makebox[1.5em][c]{$2$}} \raisebox{0.5em}{\makebox[1.5em][c]{$3$}}& \raisebox{-0.5em}{\makebox[1.9em][c]{$1$}} \raisebox{0.5em}{\makebox[1.9em][c]{$1$}} \raisebox{-0.5em}{\makebox[1.9em][c]{$1$}}& \makebox[1.0em][c]{$1$} \makebox[1.0em][c]{$1$} \makebox[1.0em][c]{$1$}\\
2& \left(\begin{smallarray}{ccc}0&0&0\\1&1&8\\3&12&3\\8&18&18\\12&8&12\\18&3&1\end{smallarray}\right)& \makebox[1.0em][c]{$-$} \makebox[1.0em][c]{$-$} \makebox[1.0em][c]{$-$}& \raisebox{0.5em}{\makebox[1.5em][c]{$1$}} \raisebox{-0.5em}{\makebox[1.5em][c]{$1$}} \raisebox{0.5em}{\makebox[1.5em][c]{$3$}}& \raisebox{-0.5em}{\makebox[1.9em][c]{$1$}} \raisebox{0.5em}{\makebox[1.9em][c]{$1$}} \raisebox{-0.5em}{\makebox[1.9em][c]{$1$}}& \makebox[1.0em][c]{$1$} \makebox[1.0em][c]{$1$} \makebox[1.0em][c]{$1$}\\
8& \left(\begin{smallarray}{ccc}0&0&0\\1&1&8\\3&18&1\\8&8&18\\12&12&3\\18&3&12\end{smallarray}\right)& \makebox[1.0em][c]{$-$} \makebox[1.0em][c]{$-$} \makebox[1.0em][c]{$-$}& \raisebox{0.5em}{\makebox[1.5em][c]{$0$}} \raisebox{-0.5em}{\makebox[1.5em][c]{$1$}} \raisebox{0.5em}{\makebox[1.5em][c]{$3$}}& \raisebox{-0.5em}{\makebox[1.9em][c]{$1$}} \raisebox{0.5em}{\makebox[1.9em][c]{$1$}} \raisebox{-0.5em}{\makebox[1.9em][c]{$1$}}& \makebox[1.0em][c]{$1$} \makebox[1.0em][c]{$1$} \makebox[1.0em][c]{$1$}\\
59& \left(\begin{smallarray}{ccc}0&0&0\\1&1&3\\3&3&8\\8&18&1\\12&12&18\\18&8&12\end{smallarray}\right)& \makebox[1.0em][c]{$-$} \makebox[1.0em][c]{$-$} \makebox[1.0em][c]{$-$}& \raisebox{0.5em}{\makebox[1.5em][c]{$0$}} \raisebox{-0.5em}{\makebox[1.5em][c]{$0$}} \raisebox{0.5em}{\makebox[1.5em][c]{$3$}}& \raisebox{-0.5em}{\makebox[1.9em][c]{$1$}} \raisebox{0.5em}{\makebox[1.9em][c]{$1$}} \raisebox{-0.5em}{\makebox[1.9em][c]{$1$}}& \makebox[1.0em][c]{$1$} \makebox[1.0em][c]{$1$} \makebox[1.0em][c]{$1$}\\
1& \left(\begin{smallarray}{ccc}0&0&0\\1&1&3\\3&3&12\\8&12&8\\12&18&18\\18&8&1\end{smallarray}\right)& \makebox[1.0em][c]{$-$} \makebox[1.0em][c]{$-$} \makebox[1.0em][c]{$-$}& \raisebox{0.5em}{\makebox[1.5em][c]{$0$}} \raisebox{-0.5em}{\makebox[1.5em][c]{$2$}} \raisebox{0.5em}{\makebox[1.5em][c]{$2$}}& \raisebox{-0.5em}{\makebox[1.9em][c]{$1$}} \raisebox{0.5em}{\makebox[1.9em][c]{$1$}} \raisebox{-0.5em}{\makebox[1.9em][c]{$1$}}& \makebox[1.0em][c]{$1$} \makebox[1.0em][c]{$1$} \makebox[1.0em][c]{$1$}\\
19& \left(\begin{smallarray}{ccc}0&0&0\\1&1&3\\3&3&12\\8&8&18\\12&18&1\\18&12&8\end{smallarray}\right)& \makebox[1.0em][c]{$-$} \makebox[1.0em][c]{$-$} \makebox[1.0em][c]{$-$}& \raisebox{0.5em}{\makebox[1.5em][c]{$0$}} \raisebox{-0.5em}{\makebox[1.5em][c]{$1$}} \raisebox{0.5em}{\makebox[1.5em][c]{$2$}}& \raisebox{-0.5em}{\makebox[1.9em][c]{$1$}} \raisebox{0.5em}{\makebox[1.9em][c]{$1$}} \raisebox{-0.5em}{\makebox[1.9em][c]{$1$}}& \makebox[1.0em][c]{$1$} \makebox[1.0em][c]{$1$} \makebox[1.0em][c]{$1$}\\
93& \left(\begin{smallarray}{ccc}0&0&0\\1&1&1\\3&8&12\\8&12&3\\12&3&8\\18&18&18\end{smallarray}\right)& \makebox[1.0em][c]{$-$} \makebox[1.0em][c]{$-$} \makebox[1.0em][c]{$-$}& \raisebox{0.5em}{\makebox[1.5em][c]{$0$}} \raisebox{-0.5em}{\makebox[1.5em][c]{$0$}} \raisebox{0.5em}{\makebox[1.5em][c]{$2$}}& \raisebox{-0.5em}{\makebox[1.9em][c]{$1$}} \raisebox{0.5em}{\makebox[1.9em][c]{$1$}} \raisebox{-0.5em}{\makebox[1.9em][c]{$1$}}& \makebox[1.0em][c]{$1$} \makebox[1.0em][c]{$1$} \makebox[1.0em][c]{$1$}\\
4& \left(\begin{smallarray}{ccc}0&0&0\\1&1&18\\3&12&1\\8&18&12\\12&3&8\\18&8&3\end{smallarray}\right)& \makebox[1.0em][c]{$-$} \makebox[1.0em][c]{$-$} \makebox[1.0em][c]{$-$}& \raisebox{0.5em}{\makebox[1.5em][c]{$1$}} \raisebox{-0.5em}{\makebox[1.5em][c]{$1$}} \raisebox{0.5em}{\makebox[1.5em][c]{$1$}}& \raisebox{-0.5em}{\makebox[1.9em][c]{$1$}} \raisebox{0.5em}{\makebox[1.9em][c]{$1$}} \raisebox{-0.5em}{\makebox[1.9em][c]{$1$}}& \makebox[1.0em][c]{$1$} \makebox[1.0em][c]{$1$} \makebox[1.0em][c]{$1$}\\
60& \left(\begin{smallarray}{ccc}0&0&0\\1&1&1\\3&12&12\\8&8&8\\12&3&18\\18&18&3\end{smallarray}\right)& \makebox[1.0em][c]{$-$} \makebox[1.0em][c]{$-$} \makebox[1.0em][c]{$-$}& \raisebox{0.5em}{\makebox[1.5em][c]{$0$}} \raisebox{-0.5em}{\makebox[1.5em][c]{$1$}} \raisebox{0.5em}{\makebox[1.5em][c]{$1$}}& \raisebox{-0.5em}{\makebox[1.9em][c]{$1$}} \raisebox{0.5em}{\makebox[1.9em][c]{$1$}} \raisebox{-0.5em}{\makebox[1.9em][c]{$1$}}& \makebox[1.0em][c]{$1$} \makebox[1.0em][c]{$1$} \makebox[1.0em][c]{$1$}\\
661& \left(\begin{smallarray}{ccc}0&0&0\\1&1&1\\3&3&8\\8&12&3\\12&8&18\\18&18&12\end{smallarray}\right)& \makebox[1.0em][c]{$-$} \makebox[1.0em][c]{$-$} \makebox[1.0em][c]{$-$}& \raisebox{0.5em}{\makebox[1.5em][c]{$0$}} \raisebox{-0.5em}{\makebox[1.5em][c]{$0$}} \raisebox{0.5em}{\makebox[1.5em][c]{$1$}}& \raisebox{-0.5em}{\makebox[1.9em][c]{$1$}} \raisebox{0.5em}{\makebox[1.9em][c]{$1$}} \raisebox{-0.5em}{\makebox[1.9em][c]{$1$}}& \makebox[1.0em][c]{$1$} \makebox[1.0em][c]{$1$} \makebox[1.0em][c]{$1$}\\
2179& \left(\begin{smallarray}{ccc}0&0&0\\1&1&1\\3&3&3\\8&8&18\\12&18&8\\18&12&12\end{smallarray}\right)& \makebox[1.0em][c]{$-$} \makebox[1.0em][c]{$-$} \makebox[1.0em][c]{$-$}& \raisebox{0.5em}{\makebox[1.5em][c]{$0$}} \raisebox{-0.5em}{\makebox[1.5em][c]{$0$}} \raisebox{0.5em}{\makebox[1.5em][c]{$0$}}& \raisebox{-0.5em}{\makebox[1.9em][c]{$1$}} \raisebox{0.5em}{\makebox[1.9em][c]{$1$}} \raisebox{-0.5em}{\makebox[1.9em][c]{$1$}}& \makebox[1.0em][c]{$1$} \makebox[1.0em][c]{$1$} \makebox[1.0em][c]{$1$}